\documentclass[11pt]{amsart}
\usepackage{amsmath,amssymb,amscd}
\usepackage{mathrsfs}
\usepackage{amsrefs}
\usepackage{graphicx,color}
\usepackage{amsthm}
\usepackage{latexsym}
\usepackage[all]{xy}
\usepackage{hyperref}
\usepackage[mathscr]{eucal}
\usepackage{comment}
\usepackage{tikz}
\usepackage{tikz-cd}
\usepackage{fancyhdr}
\usepackage{mathtools}

\theoremstyle{definition}
\newtheorem{defn}[equation]{Definition}

\theoremstyle{plain}
\newtheorem{thm}[equation]{Theorem}
\newtheorem{prp}[equation]{Proposition}
\newtheorem{lem}[equation]{Lemma}
\newtheorem{cor}[equation]{Corollary}

\newtheorem{notn}[equation]{Notation}
\theoremstyle{remark}
\newtheorem{rmk}[equation]{Remark}
\newtheorem{exmp}[equation]{Example}

\DeclareMathOperator{\id}{id}
\DeclareMathOperator{\rank}{rank}

\newcommand{\bB}{\mathbb{B}}
\newcommand{\bC}{\mathbb{C}}

\newcommand{\bF}{\mathbb{F}}

\newcommand{\bK}{\mathbb{K}}

\newcommand{\bN}{\mathbb{N}}

\newcommand{\bP}{\mathbb{P}}

\newcommand{\bR}{\mathbb{R}}

\newcommand{\bT}{\mathbb{T}}

\newcommand{\bX}{\mathbb{X}}

\newcommand{\bZ}{\mathbb{Z}}
\newcommand{\cA}{\mathcal{A}}

\newcommand{\cC}{\mathcal{C}}

\newcommand{\cF}{\mathcal{F}}

\newcommand{\cH}{\mathcal{H}}
\newcommand{\cI}{\mathcal{I}}
\newcommand{\cJ}{\mathcal{J}}
\newcommand{\cK}{\mathcal{K}}

\newcommand{\cM}{\mathcal{M}}

\newcommand{\cP}{\mathcal{P}}
\newcommand{\cQ}{\mathcal{Q}}

\newcommand{\cT}{\mathcal{T}}
\newcommand{\cU}{\mathcal{U}}
\newcommand{\cV}{\mathcal{V}}
\newcommand{\cW}{\mathcal{W}}
\newcommand{\cX}{\mathcal{X}}

\newcommand{\fD}{\mathfrak{D}}

\newcommand{\fF}{\mathfrak{F}}

\newcommand{\fH}{\mathfrak{H}}

\newcommand{\fL}{\mathfrak{L}}
\newcommand{\fM}{\mathfrak{M}}

\newcommand{\fc}{\mathfrak{c}}

\newcommand{\fh}{\mathfrak{h}}

\newcommand{\fs}{\mathfrak{s}}
\newcommand{\ft}{\mathfrak{t}}

\newcommand{\fv}{\mathfrak{v}}
\newcommand{\fw}{\mathfrak{w}}

\newcommand{\rmS}{\mathrm{S}}
\newcommand{\rmP}{\mathrm{P}}
\newcommand{\rmU}{\mathrm{U}}
\newcommand{\sfD}{\mathsf{D}}

\newcommand{\sA}{\mathscr{A}}

\newcommand{\sE}{\mathscr{E}}

\newcommand{\sG}{\mathscr{G}}
\newcommand{\sH}{\mathscr{H}}

\newcommand{\sK}{\mathscr{K}}

\newcommand{\sV}{\mathscr{V}}
\newcommand{\sW}{\mathscr{W}}

\newcommand{\K}{\mathrm{K}}
\newcommand{\KO}{\mathrm{KO}}
\newcommand{\KK}{\mathrm{KK}}

\newcommand{\KKR}{\mathrm{KKR}}
\newcommand{\Cl}{\mathit{Cl}}

\newcommand{\ev}{\mathrm{ev}}

\DeclareMathOperator{\Hom}{Hom}

\DeclareMathOperator{\End}{End}

\DeclareMathOperator{\Ind}{Ind}
\DeclareMathOperator{\Ad}{Ad}
\DeclareMathOperator{\Euc}{Euc}
\DeclareMathOperator{\Fred}{Fred}
\newcommand{\PD}{\mathrm{PD}}
\newcommand{\pt}{\mathrm{pt}}

\DeclareMathOperator{\hotimes}{\hat{\otimes}}

\newcommand{\Thom}{\mathrm{Thom}}
\newcommand{\blank}{\text{\textvisiblespace}}
\newcommand{\pr}{\mathrm{pr}}
\DeclareMathOperator{\qAut}{\mathrm{Aut}_{\mathrm{qtm}}}

\makeatletter
\@addtoreset{equation}{section}
\makeatother

\title[Twisted crystallographic T-duality]{Twisted crystallograpic T-duality via the Baum--Connes isomorphism}
\date{\today}

\author{Kiyonori Gomi}
\address{Department of Mathematics, 
Tokyo Institute of Technology,
2-12-1 Ookayama, Meguro-ku, Tokyo, 152-8551, 
Japan.}
\email{kgomi@math.titech.ac.jp}

\author{Yosuke Kubota}
\address{Department of Mathematical Sciences, Shinshu University, 3-1-1 Asahi, Matsumoto, Nagano, 390-8621, Japan / RIKEN iTHEMS Program, 2-1 Hirosawa, Wako, Saitama, 351-0198, Japan}
\email{ykubota@shinshu-u.ac.jp}

\author{Guo Chuan Thiang}
\address{Beijing International Center for Mathematical Research, Peking University, 5 Yiheyuan Rd, Beijing, China}
\email{thgchuan@gmail.com}

\begin{document}
\maketitle
\begin{abstract}
We establish the twisted crystallographic T-duality, which is an isomorphism between Freed-Moore twisted equivariant K-groups of the position and momentum tori associated to an extension of a crystallographic group. 
The proof is given by identifying the map with the Dirac homomorphism in twisted Chabert--Echterhoff KK-theory.
We also illustrate how to exploit it in K-theory computations.
\end{abstract}
\tableofcontents

\section{Introduction}
T-duality arose as a certain equivalence between string theories compactified on a circle and on a dual circle \cite{buscher1987symmetry}. That a rich mathematical structure lies behind \emph{topological T-duality} became apparent from \cite{bouwknegtDualityTopologyChange2004}, where it was found that the K-theories of a circle bundle twisted by an $H$-flux, coincided with those of a dually fibered one with a dual twist. Furthermore, a $C^*$-algebraic formulation \cite{mathaiDualityTorusBundles2005} revealed the connection of T-duality to the Baum--Connes assembly map \cite{baumClassifyingSpaceProper1994}.

The simplest example comes from the abelian group $\bZ$, whose classifying space $\bR/\bZ$ and the Pontrjagin dual $\hat{\bZ}$ are circles dual to each other. The basic T-duality isomorphisms $\K^*(\bR/\bZ)\cong \K^{*-1}(\hat{\bZ})$ can be formulated as a Fourier--Mukai transform, or as the Baum--Connes assembly map for $\bZ$ composed with Poincar\'{e} duality. In \cite{gomiCrystallographicTduality2019}, it was shown that every (ordinary) crystallographic space group $G$ acting on $d$-dimensional Euclidean space $V$, induces \emph{crystallographic T-duality} isomorphisms of twisted K-theories,
\begin{equation}
    \mathrm{T}_G:\K^{*+d,- \fv}_{P}(V/\Pi) \to \K_P^{*,\sigma}(\hat{\Pi}) . \label{eqn:old.crystal.T.duality}
\end{equation}
Here, $\Pi$ is the abelian subgroup of translations, $P=G/\Pi$ is the \emph{point group}, while $\fv$ and $\sigma$ are equivariant twists canonically associated to the $P$-spaces $V/\Pi$ and $\hat{\Pi}$ respectively. The key point is that $V/\Pi$ and $\hat{\Pi}$ are generally inequivalent $P$-spaces with different untwisted cohomologies, yet there exists a suitable twist $\fv$ and $\sigma$ on each side, such that the twisted equivariant K-theories are in perfect duality.

The purpose of this paper is to clarify the mathematical structure underlying (\ref{eqn:old.crystal.T.duality}), and establish a vast generalization, (\ref{eq:twisted.crystal}), which we call \emph{twisted crystallographic T-duality}. These dualities are mediated by \emph{twisted crystallographic groups}, whose associated twisted equivariant K-theories arise in modern studies of topological phases in physics, as we outline below.

From the standpoint of condensed matter physics, the topological T-duality map \eqref{eqn:old.crystal.T.duality} is a topological analogue of the Fourier transform revealing a duality between the position space torus $V/\Pi$ and the momentum space torus $\hat{\Pi}$.
This is relevant to the theory of topological insulators, because the K-group of the momentum torus (Brillouin torus) classifies topological phases protected by a given symmetry \cites{kitaevPeriodicTableTopological2009,schnyderClassificationTopologicalInsulators2008,freedTwistedEquivariantMatter2013}.

According to Wigner's theorem \cite{freedWignerTheorem2012}, a symmetry of quantum mechanics is a unitary/antiunitary $\bZ_2$-graded projective representation of a group, which is controlled by a triple $(\phi,c,\tau)$ of group cocycles called a twist. We call such a representation a $(\phi,c,\tau)$-twisted unitary representation. Therefore quantum mechanical symmetries go beyond the standard theory of unitary group representations. A powerful framework to treat topological phases protected by such quantum mechanical symmetries is the Freed--Moore K-theory \cite{freedTwistedEquivariantMatter2013}, see also \cites{thiangKtheoreticClassificationTopological2016,kubotaControlledTopologicalPhases2017,gomiFreedMooreKtheory2017}. 
Here we consider a merger of the crystallographic and quantum mechanical symmetries, which is precisely described in Section \ref{section:2}. Roughly speaking, a twisted crystallographic group is a quadruple $(G,\phi,c,\tau)$, where $G$ is a discrete group acting on the Euclidean space $V$ and $(\phi,c,\tau)$ is a twist of $G$. Moreover we assume that there is a normal subgroup $\Pi$ of $G$ acting on $V$ by translations and $(\phi,c,\tau)$ is a twist pulled back from $P:=G/\Pi$.
Freed--Moore~\cite{freedTwistedEquivariantMatter2013} proposed that topological phases protected by the twisted crystallographic group $(G,\phi,c,\tau)$ are classified by the twisted equivariant $\K$-group $\prescript{\phi}{}{\K}_P^{*,\ft + \sigma}(\hat{\Pi})$. Here, $\ft$ is short for $(c,\tau)$, and $\sigma$ is another twist arising from viewing $G$ as an extension of $P$ by $\Pi$. This is reviewed in Section \ref{section:3.3}.

The twisted crystallographic T-duality for a twisted crystallographic group $(G,\phi,c,\tau)$, the main concern of this paper, is an isomorphism
\begin{align}
\prescript{\phi}{}{\mathrm{T}} _G^{\ft } \colon \prescript{\phi}{} \K^{*+d,\ft - \fv}_{P}(V/\Pi) \to \prescript{\phi}{} \K_P^{*,\ft + \sigma}(\hat{\Pi}).    \label{eq:twisted.crystal}
\end{align}
This map is defined as a generalization of the Fourier-Mukai transform in K-theory and is constructed as the composition of standard operations in Freed--Moore twisted equivariant K-theory. 
The proof of (\ref{eq:twisted.crystal}) being isomorphic consists of two steps. First, we construct an isomorphism of the groups in both sides of (\ref{eq:twisted.crystal}) by using the technology of Kasparov's KK-theory \cite{kasparovEquivariantKKTheory1988} and the Baum--Connes isomorphism \cite{baumClassifyingSpaceProper1994}. 
Second, we show that the isomorphism constructed in the first step coincides with the twisted crystallographic T-duality map of our interest. 
We remark that the second step is first considered in this paper even in untwisted cases, which is studied in the previous work \cite{gomiCrystallographicTduality2019}. In \cite{gomiCrystallographicTduality2019}, both the Fourier--Mukai and the Baum--Connes maps  are independently used to formulate untwisted crystallographic T-duality, and only the latter was shown to be an isomorphism.

This strategy is known to work in the non-equivariant topological T-duality. It is a well-known fact in index theory that the (non-equivariant) topological T-duality map is nothing but the Baum--Connes assembly map for the lattice group $\mathbb{Z}^d$ (which is pointed out in \cite{baumClassifyingSpaceProper1994}*{Section 3}).  
Hence the Dirac-dual Dirac method \cite{kasparovEquivariantKKTheory1988} provides an alternative proof of the isomorphism of the T-duality map, which is essentially different from the one given in e.g.~\cite{mathaiTdualityCircleBundles2014}.
In order to extend this approach to (twisted) crystallographic T-duality, it is convenient to work in Chabert--Echterhoff twisted equivariant KK-theory \cite{chabertTwistedEquivariantKK2001}. 
For our use, we consider its $\phi$-twisted generalization, which is summarized in Appendix \ref{section:app}. 
This is a bivariant homology theory for $\phi$-twisted $(G,\Pi)$-C*-algebras, i.e., a C*-algebra $A$ with a $\phi$-twisted $G$-action whose restriction to $\Pi$ is implemented by a unitary representation (cf.\ Definition \ref{def:GCst}), constructed for the purpose of understanding the permanence properties of the Baum--Connes conjecture. 
A typical example of a $\phi$-twisted $(G,\Pi)$-C*-algebra is the crossed product $A \rtimes \Pi$ of a $\phi$-twisted $G$-C*-algebra $A$. 
Indeed, the crossed product functor $A \mapsto A \rtimes \Pi$ extends to the $\phi$-twisted partial descent homomorphism $\prescript{\phi}{}{j}_{G,\Pi} \colon \prescript{\phi}{}\KK^G(A,B) \to \prescript{\phi}{}\KK^{G,\Pi}(A \rtimes \Pi,B \rtimes \Pi)$.
The isomorphism considered in the first step explained above is the Kasparov product with the partial descent of the Dirac homomorphism $\prescript{\phi}{}{j}_{G,\Pi}(\sfD)$. 
For comparing the Freed--Moore and Chabert--Echterhoff K-theories, a point is that this comparison does not work in the level of categories. Indeed, the categories of $\phi$-twisted $(G,\Pi)$-C*-algebras and the twisted $P$-spaces have very small overlaps.

In the rest of the paper, Sections \ref{section:6} and \ref{section:7}, we introduce applications of the twisted crystallographic T-duality isomorphisms. The first is concerned with the computation of the twisted equivariant K-groups. A standard approach to the calculation of these groups is the Atiyah--Hirzebruch spectral sequence (AHSS), which is well-studied in the physics paper \cite{shiozakiAtiyahHirzebruchSpectral2018}. It is a problem of the AHSS approach that it does not determine the group in general because of the extension problem. 
In Section \ref{section:6} we show that the twisted crystallographic T-duality isomorphism solves the extension problem in some examples.

The second is concerned with the induction of topological insulators. If one has a subgroup $H$ of $G$, we can define the `induction of topological insulator' map 
\[\Ind_H^G \colon \prescript{\phi}{}{\K}_Q^{*,\ft+\sigma }(\hat{\Sigma }) \to \prescript{\phi}{}{\K}_P^{*,\ft+\sigma }(\hat{\Pi}), \]
where $\Sigma := H \cap \bR^d$ is the subgroup of translations in $H$ and $Q:=H/\Sigma$.
When $H$ is a finite subgroup, an element of the image of this map is called an atomic insulator and studied in the context of condensed-matter physics.
In Section \ref{section:7} we show that, through the twisted crystallographic T-duality isomorphism, this induction corresponds to the map $\K_Q^{*,\ft-\fw}(W/\Sigma) \to \K_P^{*,\ft-\fv}(V/\Pi)$ induced from the wrong-way functoriality of the twisted equivariant K-theory, where $W \subset V$ is an $H$-invariant affine subspace on which $H$ acts cocompactly.

\subsection*{Acknowledgments}
The author KG was supported by JSPS KAKENHI Grant Numbers JP15K04871, 20K03606 and JP17H06461. YK was supported by RIKEN iTHEMS and JSPS KAKENHI Grant Numbers 19K14544, JP17H06461 and JPMJCR19T2. GCT was supported by Australian Research Council grants DE170100149 and DP200100729, and thanks the University of Adelaide for hosting him as a visitor.

\section{Twisted crystallographic group}\label{section:2}
In this section we introduce the notion of twisted crystallographic group, which is a merger of quantum mechanical and crystallographic symmetries protecting topological phases of matter.

\subsection{Definition}
Let $V$ denote the $d$-dimensional Euclidean space, i.e., the real vector space $\bR^d$ equipped with the standard inner product. The set of isometries of the space $V$ forms the Euclidean motion group $\Euc (V)$. This group is isomorphic to the semi-direct product $\bR ^d \rtimes O(d)$, where $\bR^d$ and $O(d)$ acts on $V$ by translations and orthogonal transformations respectively. 
The subgroup $S \subset \Euc(V)$ is called a $d$-dimensional \emph{crystallographic group} (or a \emph{space group}) if
\begin{itemize}
\item the subgroup $\Pi:=S \cap \bR^d \subset \bR^d$ is a full-rank lattice of translations, and
\item the point group $P = S/\Pi \subset O(d)$ is a finite group.
\end{itemize}
That is, the diagram
\[
\xymatrix@R=1em{
1 \ar[r] & \bR^d \ar[r]\ar@{}[d]|{\cup} & \Euc(V) \ar[r]\ar@{}[d]|{\cup} & O(d) \ar[r]\ar@{}[d]|{\cup} & 1 \\
1 \ar[r] & \Pi \ar[r] & S \ar[r] &P \ar[r] & 1,
}
\]
commutes.

The Euclidean group and its subgroups act naturally on $L^2(V)$, but for physical reasons, we also need to consider on-site, or internal, quantum mechanical symmetries.
Let $\sK$ denote the Hilbert space of internal degrees of freedom (e.g.\ spinor, gauge, etc.). 
This is a finite rank complex Hilbert space which is decomposed as $\sK=\sK^0 \oplus \sK^1$ by a $\bZ_2$-grading $\gamma$, i.e., $\gamma = \gamma^*$ and $\gamma^2=1$. 
We write $\qAut (\sK)$ for the group of unitary/antiunitary and even/odd operators on $\sK$. Let $\rmU^0(\sK)$ denote the group of even unitaries on $\sK$, which is isomorphic to $\rmU(\sK^0) \times \rmU(\sK^1)$.
Once we fix an antilinear even involution $T$ and a linear odd involution $S$ such that $[T,S]=0$, there is a canonical isomorphism
\[\qAut(\sK) \cong \mathrm{U}^0(\sK) \rtimes (\bZ_2 \times \bZ_2),\]
where $\bZ_2 \times \bZ_2$ in the right hand side corresponds the subgroup generated by $T$ and $S$.
Wigner's theorem \cite{freedWignerTheorem2012} states that the quotient group 
\[ \qAut (\sK) /\bT \cong \rmP \rmU^0(\sK) \rtimes (\bZ_2 \times \bZ_2), \]
where $\rmP \rmU ^0 (\sK):= \rmU ^0 (\sK)/\bT$ and $\bT$ is the scalar subgroup, is canonically isomorphic to the automorphism group of the set of quantum states $\bP\sK:=(\sK \setminus \{ 0 \}) / \bC^\times $ which preserves the transition probabilities and commutes with $\gamma$. Hence it is denoted by $\qAut (\bP \sK)$ hereafter.

\begin{defn}\label{defn:twcry}
A subgroup $G$ of the direct product $\Euc(V) \times  \qAut (\bP \sK)$ is said to be a $d$-dimensional \emph{twisted crystallographic group} if
\begin{enumerate}
\item the subgroup $\Pi:= G \cap (\bR^d \times 1)$ is a full-rank lattice of translations, 
\item The quotient $P = G/\Pi $, called the twisted point group, is a finite subgroup of $O(d) \times \qAut(\bP \sK)$.
\end{enumerate}
That is, there is a commutative diagram of exact sequences
\[
\xymatrix@R=1em{
1 \ar[r] & \bR^d \ar@{}[d]|{\cup} \ar[r]  & \Euc (V ) \times  \qAut (\bP \sK) \ar[r]\ar@{}[d]|{\cup}  & O(d)\times  \qAut (\bP \sK) \ar[r]\ar@{}[d]|{\cup}  & 1 \\
1 \ar[r] & \Pi \ar[r] & G \ar[r] & P  \ar[r] & 1.
}
\]
\end{defn}

We write $i \colon G \to \mathop{\mathrm{Euc}}(V ) \times \qAut(\bP\sK )$ for the inclusion, as well as $j:=\pr_{\Euc (V)} \circ i$ and $k:=\pr_{\qAut(\bP\sK )} \circ i$.

The group $\qAut (\bP \sK)$ is naturally equipped with the following three structures:
\begin{itemize}
\item The (anti)linearity indicator; the homomorphism $\tilde{\phi} \colon \qAut(\bP \sK) \to \bZ_2$ mapping unitaries/antiunitaries to $0$/$1$ respectively.
\item The grading indicator; the homomorphism $\tilde{c} \colon \qAut(\bP\sK) \to \bZ_2$ mapping even/odd unitaries to $0$/$1$ respectively.
\item A group extension
\[ 1 \to \bT \to \mathrm{U}^0(\sK) \rtimes (\bZ_2 \times \bZ_2) \to \rmP \mathrm{U}^0(\sK) \rtimes (\bZ_2 \times \bZ_2) \to 1, \]\label{eq:twisted.extension}
which is $\tilde{\phi}$-twisted, i.e., the adjoint action of $\qAut(\bP \sK)$ on $\bT$ is
\begin{align} g\cdot z := \begin{cases} z & \text{ if $\tilde{\phi}(g) = 0 $}, \\ \bar{z } & \text{ if $\tilde{\phi}(g) =1$.} \end{cases} \label{eq:phi} \end{align}
\end{itemize}
Each of them corresponds to a group cocycle, 
\begin{itemize}
\item $\tilde{\phi} \in H^1(\qAut (\bP \sK);\bZ_2)$, 
\item $\tilde{c} \in H^1(\qAut (\bP \sK);\bZ_2)$ and 
\item $\tilde{\tau} \in H^2(\qAut (\bP \sK);\prescript{\tilde{\phi}}{}{\bT})$ 
\end{itemize}
respectively. Here we write $\prescript{\tilde{\phi}}{}{\bT}$ for the abelian group $\bT$ with the $\qAut(\bP\sK)$-action as (\ref{eq:phi}).

Let us denote by $\prescript{\tilde{\phi}}{}{\bZ}_N$ the cyclic group $\{ z \in \bC \mid z^N=1\} \subset \bT$ with the $\qAut(\bP \sK)$-action given as (\ref{eq:phi}).
We remark that the element of $H^2(\qAut(\bP \sK);\prescript{\tilde{\phi}}{}{\bZ}_N)$ corresponding to the extension 
\[ 1 \to \bZ_N \to \rmS\mathrm{U}^0(\sK) \rtimes (\bZ_2 \times \bZ_2) \to \rmP \mathrm{U}^0(\sK) \rtimes (\bZ_2 \times \bZ_2) \to 1 \]
is mapped to $\tilde{\tau}$ by the homomorphism induced from $\prescript{\tilde{\phi}}{} \bZ_N \to \prescript{\tilde{\phi}}{} \bT$.

A twisted crystallographic group $G$ is also equipped with the cocycles $(\phi,c,\tau)$ obtained by the pull-back of $(\tilde{\phi},\tilde{c},\tilde{\tau})$ with respect to $k \colon G \to \qAut(\bP \sK) $. 
Note that $(\phi,c,\tau)$ are pulled back through $P$, in other words the restrictions $\phi|_\Pi$, $c|_{\Pi}$, $\tau|_{\Pi}$ are trivial (see Subsection \ref{section:5.4} for a generalization). 

\begin{prp}\label{prp:twcry}
Let $(G, \phi,c,\tau)$ be the following data;
\begin{itemize}
\item $j \colon G \to \Euc (V)$ is a group homomorphism with finite kernel such that $S:= j(G)$ is a space group, 
\item $\phi, c \in H^1(G; \bZ_2)$ and $\tau \in \mathrm{Im} (H^2(G ; \prescript{\phi}{} \bZ_N)) \subset H^2(G;\prescript{\phi}{} \bT )$ for some $N \in \bN$.
\end{itemize}
Then there is a homomorphism $k \colon G \to \qAut(\bP\sK)$ such that
\begin{enumerate}
    \item the image $(j,k)(G)$ is a twisted crystallographic group,
    \item $k^*(\tilde{\phi},\tilde{c},\tilde{\tau}) = (\phi,c,\tau)$.
\end{enumerate}
\end{prp}
For this reason, in this paper we often specify a twisted crystallographic group by a quadruple $(G,\phi,c,\tau)$. For the proof of this proposition, we use the following lemma. 
\begin{lem}\label{lem:groupext}
Let $1 \to K \to H \xrightarrow{q} \Pi \to 1$ be an extension of groups such that $K$ is finite and $\Pi$ is free abelian of finite rank. Assume that the adjoint action $\Pi \to \mathop{\mathrm{Out}}(K)$ is trivial. Let $N:=\# K$, $\Pi_N:=\{ t^N \mid t \in \Pi\}$ and $H_N:=q^{-1}(\Pi_N)$. Then there is a section $s \colon \Pi_N \to H_N$ by homomorphism whose image commutes with $K$, i.e., $H_N \cong \Pi_N \times K$.
\end{lem}
\begin{proof}
The proof is given by induction on the rank of $\Pi$. 
When $\rank \Pi=0$, i.e., $\Pi$ is the trivial group, then the claim trivially holds. 
Assume that the claim holds for any extension with $\rank \Pi \leq n-1$. 
Let $\Pi' \leq \Pi$ be a subgroup with $\Pi/\Pi' \cong \bZ$. let $q \colon H \to \Pi$ denote the projection and set $H':=q^{-1}(\Pi')$, $\Pi_N':=\Pi_N \cap \Pi'$ and $H_N':=H_N \cap H'$. Then by the induction hypothesis there is a section $s' \colon \Pi_N' \to H_N'$ whose image commutes with $K$.  

Let us choose an element $x \in H$ which is mapped to the generator by the composition $H \to \Pi \to \Pi/\Pi' \cong \bZ$. 
Then there is an isomorphism $H=H' \rtimes _{\Ad (x)}\bZ$. 
Since $\Ad(x)$ is an inner action on $K$ by assumption, we may assume that $\Ad(x)=\id$ by replacing $x$ with $kx$ for some $k \in K$ if necessary. 
Moreover, for any $t \in \Pi_N'$, there is $k_t \in K$ such that $\Ad (x)(s'(t)) = s'(t)k_t$ holds. Hence we have 
\[\Ad (x^{l})(s'(t)) = \Ad(x^{l-1})(s'(t)k_t) = \Ad (x^{l-2}) (s'(t)k_t^2) = \cdots = s'(t) k_t^{l},\]
and in particular, $\Ad(x^{N})$ acts identically on the subgroup $s(\Pi'_N)$. 
Finally we obtain the desired section $s \colon \Pi_N \to H_N$ by $s|_{\Pi'} = s'$ and $s(q(x)^N)=x^N$.
\end{proof}

\begin{proof}[Proof of Proposition \ref{prp:twcry}]
We choose a preimage $\tau' \in H^2(G;\prescript{\phi}{} \bZ_N)$ of $\tau$. Let $G_{\tau'}$ denote the $\phi$-twisted extension of $G$ corresponding to $\tau'$ (i.e.\ the pullback of \eqref{eq:twisted.extension} under $k$) and let $H_{\tau'}$ denote the kernel of the composition 
\[G_{\tau'} \to G \xrightarrow{(j,\phi,c)} \Euc(V) \times (\bZ_2 \times \bZ_2) \to O(d) \times (\bZ_2 \times \bZ_2).\]
By assumption, $H_{\tau'} $ is a finite index normal subgroup of $G_{\tau'}$ and $j(H_{\tau'} ) \subset \Euc (V)$ is a free abelian group of rank $d$. 
We apply Lemma \ref{lem:groupext} for the extension $1 \to K_{\tau'} \to H_{\tau'} \to j(H_{\tau'} ) \to 1$. Then, we obtain a finite index normal subgroup $s(j(H_{\tau'} )_N)$ of $H_{{\tau'},N}:= j^{-1}(j(H_{\tau'})_N)$. Now we define the normal subgroup $\Pi$ of $G$ as
\[\Pi:= \bigcap_{gH_{{\tau'},N} \in G_{\tau'} /H_{{\tau'},N}} g s(j(H)_N) g^{-1}.  \] 
Then $\Pi$ is a  finite index normal subgroup of $G_{\tau'}$ and is free abelian of rank $d$. 

Set $P_{\tau'}:=G_{\tau'} /\Pi$ and $P:=G/\Pi$. Then $\phi$ and $c$ factor through $P$ and $\tau \in H^2(G;\prescript{\phi}{} \bT)$ is the pull-back of the $2$-cocycle of $P$ corresponding to the $\phi$-twisted extension $1 \to \prescript{\phi}{} \bZ_N \to P_{\tau'} \to P \to 1$. Let us choose a finite rank $(\phi,c,{\tau'})$-twisted unitary representation $\sK$ of $P$ 
(such a representation always exists, see e.g.\  \cite{kubotaNotesTwistedEquivariant2016}*{Example 2.9}). Then $k \colon P \to \qAut (\bP\sK)$ is the desired homomorphism.
\end{proof}

\begin{exmp}
Twisted crystallographic groups include the following classes of symmetries studied in the context of topological phases of matter.
\begin{enumerate}
\item Let $\sG :=\bZ_2 \times \bZ_2$ be acting on $V$ trivially and let $\phi, c \in \Hom (\sG , \bZ_2)$ be the first and the second projections respectively. Then there are 10 choices of a subgroup $\sA \subset  \sG$ and $\tau \in H^2(\sA ; \prescript{\phi}{} \bT)$, each of which corresponds to one of 2 complex and 8 real Clifford algebras (up to Morita equivalence). The classification of topological phase with the symmetry type $(\Pi \times \sA, \phi , c , \tau)$ is studied in \cites{kitaevPeriodicTableTopological2009,schnyderClassificationTopologicalInsulators2008} and summarized as the celebrated periodic table. 
\item More generally, let $S$ be a $d$-dimensional crystallographic group and let $(\sA,\phi,c,\tau)$ be as above. Set $G:= S \times \sA$. Then $(G,\phi,c,\tau)$ is a $d$-dimensional twisted crystallographic group. This class includes the pioneering work of Fu~\cite{fuTopologicalCrystallineInsulators2011} (corresponding to the case that $S = \bZ^2 \rtimes C_n$, where $C_n$ is the group generated by rotation by $2\pi/n$) and a part of the work of Chiu--Yao--Ryu~\cite{chiuClassificationTopologicalInsulators2013} and Morimoto--Furusaki~\cite{morimotoTopologicalClassificationAdditional2013} on reflection-invariant topological phases (corresponding to the case that $S=\bZ^d \rtimes \bZ_2$, where $\bZ_2$ acts on $V $ by a reflection). 
\item Magnetic space groups, discussed in detail in the next subsection. This class includes the antiferromagnetic topological insulator by Mong--Essin--Moore \cite{mongAntiferromagneticTopologicalInsulators2010}. The classification of topological phases with the symmetry of magnetic space groups are studied in \cite{okumaTopologicalClassificationNonmagnetic2019}. 
\end{enumerate}
\end{exmp}

\subsection{Magnetic space group}
An important and non-trivial class of twisted crystallographic groups is the \emph{magnetic space groups} (see \cites{lifshitzMagneticPointGroups2005,schwarzenbergerColourSymmetry1984} for example), which are also called (two-)colour symmetry groups, antisymmetry groups, dichromatic groups, Shubnikov groups, Heesch groups, Opechowski-Guccione groups, etc.
They are the groups of symmetries of crystals such that two possible values are attached to each site. 

\subsubsection{Definition}
Let us consider the simplest case where the Hilbert space $\sK$ of internal degrees of freedom is of dimension $1$ and equipped with the trivial $\bZ_2$-grading, i.e., $\sK^1=0$. In this case, the group $\qAut (\bP \sK)$ is isomorphic to $\bZ_2$ generated by complex conjugation.
\begin{defn}\label{defn:magnetic}
A \emph{magnetic space group} is a subgroup of $\Euc (V) \times \bZ_2$ such that:
\begin{enumerate}
\item the subgroup $\Pi:=G \cap \bR^d \subset \bR^d$ is a full-rank lattice of translations, 
\item The magnetic point group $P = G/\Pi \subset O(d) \times \bZ_2$ is a finite group.
\end{enumerate}
That is, there is a commutative diagram of exact sequences
\[
\xymatrix@R=1em{
1 \ar[r] & \bR^d \ar[r]\ar@{}[d]|{\cup} & \Euc(V) \times \bZ_2 \ar[r]\ar@{}[d]|{\cup} & O(d) \times \bZ_2 \ar[r]\ar@{}[d]|{\cup} & 1 \\
1 \ar[r] & \Pi \ar[r] & G \ar[r] &P \ar[r] & 1.
}
\] 
\end{defn}

In view of the background of our formulation, a symmetry $\tilde{\varphi} = (\varphi, t) \in \Euc(V) \times \bZ_2$ preserves the `time' direction if $t$ is the identity and reverses it if $t$ is the complex conjugation. In other words, the second projection $\phi=\pr_{\bZ_2} \colon \Euc(V) \times \bZ_2 \to \bZ_2$ indicates the time-reversal. On the other hand, the first projection $\pr_{\Euc(V)} \colon \Euc(V) \times \bZ_2 \to \Euc(V)$ tells us what the induced isometry on $V$ is. 
From the viewpoint of physics, the two values attached to each site can be regarded as the spin directions (up and down), and the magnetic space groups as the symmetries of such spin systems on crystals. One can also think of the possible two values as the time directions (future and past), and the magnetic space groups as the symmetries of crystals which possibly contain the time-reversal symmetry.

For a magnetic space group $G$, the subgroup $S:=j(G) \subset \Euc(V)$ is called an \emph{associated space group} of $G$. This terminology is justified by the following lemma.
\begin{lem}
A discrete subgroup $G$ of $\Euc(V) \times \bZ_2$ is a magnetic space group if and only if $S:=j(G)$ is a space group.
\end{lem}
\begin{proof}
It suffices to show that $\Pi \subset \bR^d$ is of full-rank if and only if $S \cap \bR^d \subset \bR^d$ is of full-rank. 
To see this, consider the restriction of $j$ to a surjection $ G\cap (\bR^d \times \bZ_2) \to S \cap \bR^d$. Since $G \cap (\bR^d \times \bZ_2)$ includes $\Pi$ as a finite index subgroup, the ranks of $\Pi$ and $S \cap \bR^d$ are the same.
\end{proof}

\subsubsection{Type classification and cohomology}
By means of their associated data $\phi$ and $S$, we can classify the magnetic space groups $G$ into three types, which is the standard formulation in the literature \cites{lifshitzMagneticPointGroups2005,schwarzenbergerColourSymmetry1984}. 

To state this, let $G_0:=\ker \phi$, let $\Pi_0:=G_0 \cap \Pi$ and let $P_0:=G_0/\Pi_0$. Moreover, we think of $G $ as a subgroup of $S \times \bZ_2 \subset \Euc(V) \times \bZ_2$. 
The quotient $(S \times \bZ_2)/G$ is either $\bZ_2$ or trivial. 
In the former case, we define the homomorphism $\rho \colon S \times \bZ_2 \to \bZ_2$ as the quotient $S \times \bZ_2 \to (S \times \bZ_2)/G \cong \bZ_2$. 
In the latter case, set $\rho \equiv 1$. 
We write as $\rho_S:=\rho|_S$ and $\rho_{\bZ_2}:=\rho|_{\bZ_2}$. 
Note that $G$ is now characterized as the kernel of $\rho=\rho_S \cdot \rho_{\bZ_2}$.

\begin{prp}[{\cite{schwarzenbergerColourSymmetry1984}}] \label{prop:relation_to_standard_formulation}
Let $G$ be a $d$-dimensional magnetic group, $\phi \colon G \to \bZ_2$ its time-reversal indicator, and $S$ the associated space group. Then $G$ is classified into one of the following three types:
\begin{itemize}
\item[(a)] (black group) $G$ satisfies one of the following three equivalent conditions:
\begin{itemize}
\item[(a1)] $\phi$ is trivial.
\item[(a2)] $\rho_S=1$ and $\rho_{\bZ_2} =\id_{\bZ_2}$.
\item[(a3)] $\Pi = \Pi_0 $ and $P=P_0$.
\end{itemize}
\item[(b)] (grey group) $G$ satisfies one of the following three equivalent conditions:
\begin{itemize}
\item[(b1)]$G = S \times \bZ_2$ and $\phi$ is the second projection.
\item[(b2)] $\rho_S=1$ and $\rho_{\bZ_2} = 1$.
\item[(b3)] $\Pi = \Pi_0 $ and $P=P_0 \times \bZ_2$.
\end{itemize}
\item[(c)] (black and white group) $G$ satisfies one of the following three equivalent conditions:
\begin{itemize}
\item[(c1)] $G=\{ (s , \rho_S(s)) \in \Euc(V) \times \bZ_2 \mid s \in S \}$. 
\item[(c2)] $\rho_S \neq 1$ and $\rho_{\bZ_2} =\id_{\bZ_2}$.
\item[(c3)] Either of the following two conditions holds:
\begin{itemize}
\item[(c-i)] $\Pi=\Pi_0$ and $P_0$ is an index $2$ normal subgroup of $P$.
\item[(c-ii)] $\Pi_0$ is an index $2$ subgroup of $\Pi$ and $P_0 \cong P$.
\end{itemize}
\end{itemize}
\end{itemize}
\end{prp}
Notice that the characterization of types in the above proposition omits the case of $\rho_{S} \neq 1$ and $\rho_{\bZ_2} = 1$. 
In this case, we get a magnetic space group of `black type' (i.e.\ a space group without time-reversal) whose associated space group is an index $2$ subgroup of $S$. 
Therefore this case is covered in the case of (a), though the associated subgroup may be different from $S$.

Notice also that, in the characterization (c) of black and white groups, the non-trivial homomorphism $\rho_{S}$ specifies a subgroup $\ker\rho_S \subset S$ of index $2$. Any subgroup $S' \subset S$ of index $2$ is uniquely characterized as the kernel of a non-trivial homomorphism $S \to \bZ_2$. Thus, a black and white group is characterized by an index $2$ subgroup of $S$. This characterization is adopted in \cite{schwarzenbergerColourSymmetry1984}.

Now we provide a classification of magnetic space groups via group cohomology. 

\begin{prp}\label{prop:magnetic.class}
The following hold:
\begin{enumerate}
    \item There is only one black and one grey magnetic space group for each space group $S$.
    \item For a given space group $S$ whose point group is $P_0$, the set of black and white groups of type (c-i) with $j(G)=S$ corresponds one-to-one to $H^1(P_0;\bZ_2) \setminus \{ 0 \} $. 
    \item For a given space group $G_0$ with the point group $P$, the set of black and white groups of type (c-ii) with $\ker \rho_S =G_0$ corresponds one-to-one to a subset of $\tilde{H}^1_P(\hat{\Pi}_0; \bZ_2)\setminus \{0\}$.
\end{enumerate}
\end{prp}
\begin{proof}
The claim (1) is obvious. To see (2), recall that a black and white group of type (c-i) is obtained from a non-trivial homomorphism $\rho_P \colon P \to \bZ_2$ as $G:=\ker (\rho_S, \id_{\bZ_2})$, corresponding one-to-one to a non-zero element $[\rho_P] \in H^1(P;\bZ_2)$.

We show (3). Let $G$ be a magnetic space group of type (c-ii) such that $\ker \rho_S = \ker \phi $ is the given space group $S$. 
Then, by taking the Pontrjagin duals, the $P$-equivariant extension $1 \to \Pi_0 \to \Pi \to \bZ_2 \to 1$ gives rise to a $P$-equivariant $\bZ_2$-Galois covering $\bZ_2 \to \hat{\Pi}\to \hat{\Pi}_0$ (in other words a $P$-equivariant principal $\bZ_2$-bundle) whose restriction at $0 \in \hat{\Pi}_0$ is equivariantly trivial. Hence it corresponds to an element of the reduced equivariant cohomology group $\tilde{H}^1_P(\hat{\Pi}_0;\bZ_2) \cong \ker (\iota_0^*)$, where $\iota_0 \colon \{ 0\} \to \hat{\Pi}_0$ denotes the inclusion. Moreover, if there is another magnetic space group $G'$ with $\ker \phi =S$ corresponding to the same element, then there is a commutative diagram of $P$-equivariant homomorphisms
\[
\xymatrix@R=1em@C=1em{
&H^{1}(\hat{\Pi}_0;\bZ) \cong \Pi_0 \ar[ld] \ar[rd] &\\
H^{1}(\hat{\Pi};\bZ) \cong \Pi \ar[rr]^{\cong} &&  H^{1}(\hat{\Pi}'; \bZ) \cong \Pi'}
\]
such that the skew homomorphisms are inclusions $\Pi_0 \to \Pi$ and $\Pi_0 \to \Pi' $. Hence we obtain that $\Pi$ and $\Pi'$ coincide in $\bR^d \cong \Pi_0 \otimes _\bZ \bR$, and hence $G=G'$.
\end{proof}

We say that two magnetic space groups $G_1$ and $G_2$ are equivalent if there is a (not necessarily isometric) affine transform $\varphi \colon V \to V$ preserving the orientation of $V$ such that $\mathrm{Ad}(\varphi) G_1 =G_2$. 
For a space group $S$, set 
\[ N_S:=N_{\bR^d \rtimes \mathrm{GL}_d(\bR)}(S) \cap (\bR^d \rtimes \mathrm{GL}_d^+(\bR)),\]
namely the normalizer subgroup of $S$ taken in $\bR^d \rtimes \mathrm{GL}_d^+(\bR)$. To summarize the above discussion, we obtain:
\begin{prp}
For a $d$-dimensional space group $S$, we write as $\Pi_S:=S \cap \bR^d$ and $P_S:=S/\Pi_S$. The following hold:
\begin{enumerate}
\item The set of equivalence classes of $d$-dimensional black and white magnetic space groups of type (c-i) is given by
\[\bigsqcup_S (H^1(P_{S};\bZ_2) \setminus \{0\})/N_S,\]
where $S$ runs over all $d$-dimensional space groups. 
\item The set of equivalence classes of $d$-dimensional black and white magnetic space groups of type (c-ii) is given by a subset of
\[\bigsqcup_{G_0} (\tilde{H}^1_{P_{G_0}}(\hat{\Pi}_{G_0};\bZ_2) \setminus \{0\})/N_{G_0}, \]
where $G_0$ runs over all $d$-dimensional space groups. 
\end{enumerate}
\end{prp}

We apply this proposition to the enumeration of magnetic space groups in dimensions $1$ and $2$. The calculation of the equivariant cohomology groups $H^1_P(\mathrm{pt}; \bZ_2)$ and $\tilde{H}^1_P(\hat{\Pi}; \bZ_2)$ for each space group $S$ in dimension $2$ is given in \cite{gomiTwistsTorusEquivariant2017}. By the method of calculations in this paper, one can also calculate the equivariant cohomology groups in question for the space groups in dimension $1$.  The results are listed in Table \ref{tab:mag_dim1} and Table \ref{tab:mag_dim2} below.

\begin{table}[ht]
\begin{center}
\begin{tabular}{|c|c||c|c|c|}
\hline
\mbox{Space group $S$} & $P$ & 
$H^1_P(\hat{\Pi}; \bZ_2)$ & $H^1_P(\pt; \bZ_2)$ & $\tilde{H}^1_P(\hat{\Pi}; \bZ_2)$ \\
\hline
 $\bZ$ & $1$ & $\bZ_2$ & 0 & $\bZ_2$ \\
\hline
$\bZ \rtimes O(1)$ & $O(1)$ & $\bZ_2^{\oplus 2}$ & $\bZ_2$ & $\bZ_2$ \\
\hline
\end{tabular}
\caption{The list of $H^1_P(\hat{\Pi}; \bZ_2)$ in dimension $1$}
\label{tab:mag_dim1}
\begin{tabular}{|c|c||c|c|c|}
\hline
\mbox{Space group $S$} & $P$ & 
$H^1_P(\hat{\Pi}; \bZ_2)$ & $H^1_P(\pt; \bZ_2)$ & $\tilde{H}^1_P(\hat{\Pi}; \bZ_2)$ \\
\hline
\mbox{\textsf{p1}} & $1$ & 
$\bZ_2^{\oplus 2}$ & $0$ & $\bZ_2^{\oplus 2}$ \\
\hline
\mbox{\textsf{p2}} & $\bZ_2$ &
$\bZ_2^{\oplus 3}$ & $\bZ_2$ & $\bZ_2^{\oplus 2}$ \\
\hline
\mbox{\textsf{p3}} & $\bZ_3$ &
$0$ & $0$ & $0$ \\
\hline
\mbox{\textsf{p4}} & $\bZ_4$ &
$\bZ_2^{\oplus 2}$ & $\bZ_2$ & $\bZ_2$  \\
\hline
\mbox{\textsf{p6}} & $\bZ_6$ &
$\bZ_2$ & $\bZ_2$ & $0$ \\
\hline
\mbox{\textsf{pm}/\textsf{pg}} & $D_1$ &
$\bZ_2^{\oplus 3}$ & $\bZ_2$ & $\bZ_2^{\oplus 2}$ \\
\hline
\mbox{\textsf{cm}} & $D_1$ &
$\bZ_2^{\oplus 2}$ & $\bZ_2$ & $\bZ_2$ \\
\hline
\mbox{\textsf{pmm}/\textsf{pmg}/\textsf{pgg}} & $D_2$ &
$\bZ_2^{\oplus 4}$ & $\bZ_2^{\oplus 2}$ & $\bZ_2^{\oplus 2}$ \\
\hline
\mbox{\textsf{cmm}} & $D_2$ &
$\bZ_2^{\oplus 3}$ & $\bZ_2^{\oplus 2}$ & $\bZ_2$ \\
\hline
\mbox{\textsf{p3m1}} & $D_3$ &
$\bZ_2$ & $\bZ_2$ & $0$ \\
\hline
\mbox{\textsf{p31m}} & $D_3$ &
$\bZ_2$ & $\bZ_2$ & $0$ \\
\hline
\mbox{\textsf{p4m}/\textsf{p4g}} & $D_4$ &
$\bZ_2^{\oplus 3}$ & $\bZ_2^{\oplus 2}$ & $\bZ_2$ \\
\hline
\mbox{\textsf{p6m}} & $D_6$ &
$\bZ_2^{\oplus 2}$ & $\bZ_2^{\oplus 2}$ & $0$  \\
\hline
\end{tabular}
\caption{The list of $H^1_P(\hat{\Pi}; \bZ_2)$ in dimension $2$. We follow \cite{schwarzenbergerColourSymmetry1984} for the labels of the $2$-dimensional space groups. In the column of the point group $P$, the cyclic group $\bZ_n$ is a subgroup of $\mathrm{SO}(2) \subset \mathrm{O}(2)$, and the dihedral group $D_n \not\subset \mathrm{SO}(2)$ of order $2n$ contains a reflection.}
\label{tab:mag_dim2}
\end{center}
\end{table}

In the case of dimension $1$, there is $1$ magnetic space group of type (c-i) and $2$ magnetic space groups of type (c-ii).  

In the case of dimension $2$, there are $29$ non-trivial elements of $H^1_P(\pt ; \bZ_2)$ in total. 
Among them, there are three degenerations of magnetic space groups of type (c-i) by the action of the normalizer in \textsf{pmm}, \textsf{pgg} and \textsf{cmm} space groups. All of these groups are invariant under the adjoint by the rotation by $\pi/2$. The induced action on the point group $P \cong \bZ_2 \times \bZ_2$ is the flip, and hence the induced action onto $H_P^1(\pt;\bZ_2) \cong \bZ_2 \oplus \bZ_2$ is also the flip. 

Also, there are $26$ non-trivial elements of the reduced cohomology group $\tilde{H}^1_P(\hat{\Pi}; \bZ_2)$ in total. Among them, $5$ non-trivial elements do not contribute to the enumeration of the magnetic space group of type (c-ii); $3$ comes from the \textsf{pgg} space group and $1$ from each of \textsf{p4m} and \textsf{p4g} groups. Indeed, these groups are maximal among $2$-dimensional space groups. Moreover, the \textsf{pmm} space group admits the adjoint action by the $\pi/2$ rotation. This automorphism acts on $\tilde{H}^1_P(\hat{\Pi};\bZ_2) \cong \bZ_2^{\oplus 2}$ as the flip, which identifies two non-trivial elements. 

These discussions are consistent with the result in \cite{schwarzenbergerColourSymmetry1984} that there are $26$ type (c-i) and $20$ type (c-ii) magnetic space groups in dimension $2$. Similarly, Proposition \ref{prop:magnetic.class} is consistent with the existence of $2+2+(1+2)=7$ frieze groups and $17+17+(26+20)=80$ layer groups \cite{kopsky2002international}, which are so-called subperiodic groups in one and two dimensions. Some of their associated twisted $K$-theories were investigated in \cite{gomiCrystallographicTduality2019}.

\section{Freed--Moore twisted equivariant K-theory}\label{section:3}
In this section we summarize the foundations of Freed--Moore twisted equivariant K-theory.
We recall two definitions, i.e., the Fredholm and Karoubi pictures, and introduce several fundamental operations such as the pull-back, the open embedding, the product, the Thom isomorphism and the push-forward. Although the twisted equivariant K-group is defined for proper groupoids, here we concentrate on action groupoids $X \rtimes P$, where $P$ is a compact group and $X$ is a locally compact $P$-space.
We also review the relation between the classification of topological phases protected by the symmetry of a  twisted crystallographic group and the twisted equivariant $\K^0$-group of the Brillouin torus proposed in \cite{freedTwistedEquivariantMatter2013}.

\subsection{Notations and conventions}
Here we collect some notations and conventions used throughout the paper.  

\begin{rmk}
Throughout the paper, the sesquilinear inner product on a Hilbert space is linear in its first argument and antilinear in their second. 
Conversely, the inner product of a Hilbert C*-module $A$ is antilinear in its first argument and linear in the second. 
\end{rmk}

\begin{rmk}
For $\bZ_2$-graded Hilbert spaces $\sH_1$ and $\sH_2$, their graded tensor product $\sH_1 \hotimes \sH_2$ is defined as the Hilbert space $\sH_1 \otimes \sH_2$ equipped with the $\bZ_2$-grading $\gamma_1 \otimes \gamma_2$. For a pair of (resp.\ anti) linear maps $T \colon \sH_1 \to \sH_1$ and $S \colon \sH_2 \to \sH_2$, their graded tensor product $T \hotimes S$ stands for
\[ T  \otimes S^0 + T \gamma_1 \otimes S^1, \] 
where $S^0$, $S^1$ denotes the even and odd parts of $S$. The graded tensor product of vector bundles is also defined in the same manner. 
\end{rmk}

\begin{rmk}\label{rmk:Cliffordalgebra}
Let $\Cl_{p,q}$ denote the Clifford algebra (it is customary to use $p,q$ for the integer labels in $Cl_{p,q}$, and it should not be confused with our use of $p,q\in P$) associated to $\bR^{p+q}$ with the bilinear form $\langle \cdot , \cdot  \rangle$ of signature $(p,q)$, 
that is, the $\bR$-algebra generated by elements $v$ for $v \in \bR^{p+q}$ with relations $vw+wv = \langle v,w \rangle 1$ (in other words, $\Cl_{p,q}$ is generated by mutually anticommuting elements $e_1,\cdots, e_p$ with $e_i^2=1$ and $f_1,\cdots, f_q$ with $f_i^2=-1$). 
Let $\Delta_{p,q}$ denote the real linear space $\Cl_{p,q}$ with the inner product $\langle a,b \rangle =\mathrm{tr} (b^*a)$, where $\mathrm{tr}$ is the unique tracial state on $\Cl_{p,q}$ which is even, i.e., $\mathrm{tr}(\Cl_{p,q}^{\mathrm{odd}})=0$. Let $\fc \colon \Cl_{p,q} \to \bB(\Delta_{p,q})$ denote the graded $\ast$-representation given by the multiplication from the left. 

In this paper we use the same symbols $\Cl_{p,q}$ and $\Delta _{p,q}$ for the corresponding Real C*-algebra and its Real $\ast$-representation (see Remark \ref{rmk:real}). In particular, $G$ acts on $\Cl_{p,q}$ and $\Delta _{p,q}$ by their complex conjugations through $\phi$.  
\end{rmk}

\subsection{Two definitions}
We start with a reminder of the definition of the twisted equivariant K-groups. For a detailed discussion, we refer to \cite{gomiFreedMooreKtheory2017} and \cite{kubotaNotesTwistedEquivariant2016}.

We say that a twist of the action $P \curvearrowright X$ is a triplet $(\phi,c,\sigma)$, where $\phi ,c$ are elements of the equivariant cohomology group $H^1 _P(X ; \bZ_2 )$ and $\sigma \in H^2_P(X ; \prescript{\phi}{} \bT)$ (note that the cohomology groups are not compactly-supported ones). 
For details of groupoid cohomology, refer to Appendix \ref{sec:equivariant cohomology}. 
\begin{rmk}
In contrast to some existing works such as \cite{moutuouGradedBrauerGroups2014}, in this paper we do not include the $0$-th cohomology group $H^0_P(X; \bZ_2)$ into our definition of the twist. This parameter corresponds to the degree of twisted equivariant K-theory in our definition.
\end{rmk}

\begin{rmk}
In this section, we use the letter $\sigma$ for a $2$-cocycle in $H^2_P(X ; \prescript{\phi}{} \bT)$. 
As a general convention, a left-superscript $\prescript{\phi}{}\blank$ indicates that complex conjugation is to be applied whenever $\phi(p,x)=1$.
For simplicity of notation, in this section we represent $1$-cocycles $\phi$ and $c$ by a function $P \times X \to \bZ_2$ as $\phi(p,x)$ or $c(p,x)$ (although it is possible only for $1$-cocycles which are sent to the trivial element in the non-equivariant cohomology by the forgetful map). 
\end{rmk}

For a fixed choice of $\phi$, the summation on $H^1_P(X;\bZ_2) \times H^2_P(X;\prescript{\phi}{} \bT)$ is imposed as 
\begin{align} (c,\sigma) + (c',\sigma') := (c+c', \sigma + \sigma' + \epsilon (c,c')), \label{eq:sum} \end{align}
where $\epsilon(c,c')$ is a $2$-cocycle defined by
\[ \epsilon (c,c')(p,q,x):=(-1)^{c'(p,x)c(q,x)}\]
for any $x \in X$ and $p,q \in P$. 
\begin{rmk}\label{rmk:cocycle}
A $1$-cocycle $c \in H^1_P(X;\bZ_2)$ corresponds to a $P$-equivariant principal $\bZ_2$-bundle on $X$. 
Also, a $2$-cocycle $\sigma \in H^2_P(X; \prescript{\phi}{} \bT )$ corresponds to a $\phi$-twisted extension of the action groupoid $X \rtimes P$ (cf.\ \cite{gomiFreedMooreKtheory2017}*{Definition 2.3}), comprising the following data;
\begin{itemize}
    \item a collection $\{ L_p \}_{ p \in P}$ of complex line bundles on $X$,
    \item a collection of bundle isomorphisms 
    \[ \sigma(p,q) \colon \alpha_{q^{-1}}^* L_p \otimes \prescript{\phi(p)}{}{L}_q \to L_{pq},\]
    i.e., a continuous family $\sigma(p,q)_x \colon {L}_{p,qx} \otimes \prescript{\phi(p,qx)}{}L_{q,x} \to L_{pq,x}$, with the compatibility condition 
    \[  \alpha _{r^{-1}}^* \sigma(p,q) \circ \sigma(pq,r)= \prescript{\phi(p)}{}{\sigma}(q,r ) \circ \sigma(p,qr).  \]
\end{itemize}
Here $L_{p,x}$ denotes the fiber of $L_p$ at $x \in X$.
\end{rmk}

\begin{rmk}
For simplicity of notation, hereafter we often use the letter $\ft$ for a pair $(c, \sigma) \in H^1_P(X;\bZ_2) \times H^2_P(X;\prescript{\phi}{} \bT)$.
\end{rmk}

For a twist $(\phi, \ft )$ of a $P$-space $X$, we say that a $(\phi , \ft )$-twisted $P$-equivariant hermitian vector bundle on $X$ is a $\bZ_2$-graded hermitian vector bundle $E$ on $X$ with a collection of $\bC$-linear bundle isomorphisms 
\[u_p \colon L_p \otimes \prescript{\phi(p)}{}E  \to  \alpha_{p^{-1}}^*E, \] 
i.e., a continuous family $u_{p,x} \colon L_{p,x} \otimes  \prescript{\phi(p,x)}{}E_x  \to E_{px}$, such that 
\begin{itemize}
    \item $u_{p,x} \colon L_{p,x} \otimes  \prescript{\phi(p,x)}{}E_x  \to E_{px}$ is even/odd if $c(p)=0/1$, and
    \item $u$ is $\sigma$-projective, i.e., the equality
    \[ u_{p} \circ u_{q} = u_{pq} \circ (\id_E \otimes \sigma(p,q))\] 
    holds for $p,q\in P$ as bundle maps 
    \[\alpha_{q^{-1}}^*L_{p} \otimes  \prescript{\phi(p)}{}L_{q} \otimes  \prescript{\phi(pq)}{}E  \to  \alpha_{(pq)^{-1}}^*E.\] 
\end{itemize}
In other words, $u_p$ is a continuous family of $\phi$-linear $c$-graded maps $u_{p,x} \colon E_x \otimes L_{p,x} \to E_{px}$ such that
\[u_{p,qx}\circ u_{px} = u_{pq,x} \circ (\id_E \otimes \sigma(p,q)_x). \]

Note that, for a fixed $\phi \in H^1_P(X;\bZ_2)$, the graded tensor product $E_1 \hotimes E_2$ of $(\phi,\ft_i)$-twisted $P$-equivariant hermitian vector bundles $E_i$ (for $i=1,2$) makes sense and is a $(\phi,\ft_1 + \ft_2)$-twisted $P$-equivariant hermitian vector bundle. Here $\ft_1 + \ft_2$ is the summation (\ref{eq:sum}). This is checked as
\begin{align*}
& (u_p^1 \hotimes u_p^2)(u_q^1 \hotimes u_q^2) \\
=& (-1)^{c_2(p,x)c_1(q,x)}(u_p^1u_q^1) \hotimes (u_p^2u_q^2) \\
  =& (-1)^{c_2(p,x)c_1(q,x)} (u_{pq}^1 \hotimes u_{pq}^2) (\id_{E_1 \hotimes E_2} \otimes \sigma^1(p,q) \otimes \sigma^2(p,q)). 
\end{align*}

\begin{rmk}\label{rmk:univ}
A $(\phi, \ft )$-twisted $P$-equivariant Hilbert bundle $\cH$ is said to be universal if any $(\phi , \ft )$-twisted vector bundle on $X$ is embedded into $\cH$. A universal bundle satisfies the following properties;
\begin{enumerate}
\item Any $P$-space $X$ admits a universal Hilbert $P$-bundle (cf.\ \cite{freedLoopGroupsTwisted2011}*{Corollary A.33}).
\item Any two universal Hilbert bundles are unitarily isomorphic (cf.\ \cite{freedLoopGroupsTwisted2011}*{Lemma A.26}). 
\item A universal bundle has the following absorbing property (cf.\ \cite{freedLoopGroupsTwisted2011}*{Lemma A.24}); for any $(\phi,\ft)$-twisted Hilbert bundle $\cV$, there is an even $P$-invariant unitary isomorphism
\[ V \colon \cH \oplus \cV \to \cH. \]
\end{enumerate}
\end{rmk}

Let $\bB(\cH)$, $\bK(\cH)$ and $\cU(\cH)$ denote the bundle of bounded, compact and unitary operators on $\cH$ respectively. Also, let $\mathrm{Fred}(\cH)$ denote the bundle on $X$ whose fiber at $x\in X$ is the set of odd self-adjoint operators $F_x \in \bB(\cH_x)$ such that $F_x^2-1 \in \bK(\cH_x)$, on which we put the weakest topology such that a local section $F$ is continuous if $F$ is strongly continuous and $F^2-1$ is norm continuous. 
\begin{defn}\label{defn:Fred}
We write $\Gamma_c (X,\mathrm{Fred}(\cH))^P$ for the topological space of continuous sections $F \in \Gamma (X,\mathrm{Fred}(\cH))$ such that
\begin{enumerate}
    \item $F^2-1$ is compactly supported on $X$ and
    \item $[u_p, F] =0$ holds (where $[\blank, \blank ]$ denotes the graded commutator), that is, $u_{p,x}F_{x}u_{p,x}^*=(-1)^{c(p,x)}F_{px}$ holds for any $(p,x) \in P \times X $. 
\end{enumerate}
The twisted equivariant $\K^0$-group $\prescript{\phi}{} \K^{0,\ft}_P(X)$ is defined to be the set of connected components of $\Gamma_c (X,\mathrm{Fred}(\cH))$. 
\end{defn}
Note that the set $\prescript{\phi}{}{}\K^{*,\ft}_P(X)$ forms an abelian group \cite{gomiFreedMooreKtheory2017}*{Lemma 3.3}. The additive structure imposed on this set will be described later in Definition \ref{defn:twK}. For the definition of the $\K$-group in other degrees, a symmetry of Clifford algebras is taken into account.

Let $\cH_{p,q}:=\cH \hotimes \Delta_{p,q}$ and let $\mathrm{Fred}_{p,q}(\cH)$ denote the subbundle of $\Fred (\cH_{p,q})$ consisting of operators such that $[F, \fc(v)]=0$. 
Also, we write $\Gamma_c (X , \Fred_{p,q} (\cH))^P$ for the space of continuous sections of $\mathrm{Fred}_{p,q}(\cH)$ satisfying (1), (2) of Definition \ref{defn:Fred}.

\begin{defn}\label{defn:twK}
We define the twisted equivariant K-group as 
\[ \prescript{\phi}{} \K^{q-p,\ft }_P (X) := \pi_0 (\Gamma_c (X,\mathrm{Fred}_{p,q}(\cH))^P).\]
This is an abelian group under the direct sum
\[[F_1] + [F_2] := [V(F_1 \oplus F_2)V^*], \]
where $V \colon \cH_{p,q} \oplus \cH_{p,q} \to \cH_{p,q}$ is a $P$-equivariant unitary intertwining the $Cl_{p,q}$-actions.  
\end{defn}
Note that the above summation is well-defined independent of the choice of $V$ because of the following lemma.
Here we write $\cU_{p,q}^0(\cH)$ and $\cU_{p,q}^1(\cH)$ for the fiber bundle of even/odd unitaries on $\cH_{p,q}$ which graded commutes with the action of $\Cl_{p,q}$ respectively. 
We consider the space $\Gamma (X,\cU_{p,q}^i(\cH))^P$ of graded $P$-invariant sections of these bundles. 
\begin{lem} \label{lem:unitary}
The spaces $\Gamma (X,\cU^0_{p,q}(\cH))^P$ and $\Gamma(X,\cU^1_{p,q}(\cH))^P$ are weakly contractible. 
\end{lem}
For a proof, see \cite{freedLoopGroupsTwisted2011}*{Lemma A.36} and \cite{gomiFreedMooreKtheory2017}*{Lemma 3.2}.
We remark that $\Gamma (X,\cU^1_{p,q}(\cH))^P$ is a subspace of $\Gamma_c(X,\Fred_{p,q}(\cH))^P$.

\begin{rmk}\label{rmk:Gomi}
We mention a relation of Definition \ref{defn:twK} with the definitions given in references \cite{gomiFreedMooreKtheory2017} and \cite{kubotaNotesTwistedEquivariant2016}. 
In this paper, we used the set of self-adjoint Fredholm operators as in \cite{kubotaNotesTwistedEquivariant2016}*{Definition 3.10}, instead of skew-adjoint Fredholm operators as in \cite{gomiFreedMooreKtheory2017}*{Definition 3.4}. Indeed, a self-adjoint Fredholm operator is more compatible with the Kasparov theory.

In this remark $\tau$ denotes a cocycle $H^2_P(X,\prescript{\phi}{} \bT)$ in the terminology of \cite{gomiFreedMooreKtheory2017}. It is mentioned in \cite{gomiFreedMooreKtheory2017}*{Remark 4.12} that, the twisted equivariant K-group $\prescript{\phi}{} \K_P^{(c,\tau),q-p}(X)$ given in \cite{gomiFreedMooreKtheory2017}*{Definition 3.4} is isomorphic to another group $\prescript{\phi}{}{\acute{\K}}^{(c,\acute{\tau}),p-q}_P(X)$, where $\acute{\tau}:=\tau + \epsilon (c,c)$. The latter group is the same thing as our $\prescript{\phi}{} \K_P^{q-p,c, \acute{\tau}}(X)$. That is, the convention of $\tau$ in this paper is shifted by $\epsilon (c,c)$ from \cite{gomiFreedMooreKtheory2017}. 
We also remark that the role of $p,q$ in the notation of the Clifford algebra $\Cl_{p,q}$ in this paper and \cite{gomiFreedMooreKtheory2017} are opposite, i.e.,\ in \cite{gomiFreedMooreKtheory2017} the algebra $\Cl_{p,q}$ is generated by mutually anticommuting elements $e_1,\cdots, e_p$ with $e_i^2=-1$ and $f_1,\cdots, f_q$ with $f_i^2=1$.
\end{rmk}

There is another description of the twisted equivariant $\K$-group as a twisted generalization of Karoubi's K-theory.
Here we say that a $(\phi,c, \sigma)$-twisted $P$-equivariant $\Cl_{p,q}$-vector bundle is a $\bZ_2$-graded hermitian vector bundle $E$ on $X$ with a graded $\ast$-representation $\fc \colon \Cl_{p,q} \to \bB(E)$ and a $\phi$-linear $\sigma$-twisted $P$-action $u$ such that $u_p$ is even/odd if $c(p,x)=0/1$ and the graded commutator satisfies $[\fc(v), u_p] =0$ for any $v \in \Cl_{p,q}$. 

\begin{defn}[{\cite{gomiFreedMooreKtheory2017}*{Definition 4.13}}]
Let $X$ be a locally compact $P$-space. A $(\phi, c,\sigma)$-twisted $P$-equivariant $\Cl_{p,q}$-triple on $X$ is a triple $(E,\gamma_1,\gamma_2)$, where
\begin{itemize}
\item $E$ is a $(\phi,\sigma)$-twisted $P$-vector bundle on $X$ of finite rank,
\item $E$ is equipped with a fiberwise $\ast$-representation of $\Cl_{p,q}$,
\item each $\gamma_i$ is a $\bZ_2$-grading on $E$ which makes $E$ into a $(\phi,c,\sigma)$-twisted $P$-equivariant $\Cl_{p,q}$-vector bundle on $X$, and 
\item $\gamma_1-\gamma_2$ is compactly supported.
\end{itemize}
Let $\prescript{\phi}{} \cK^{q-p,c,\sigma}_P(X)$ denote the quotient of the monoid of homotopy classes of $(\phi,c, \sigma)$-twisted $P$-equivariant triples on $X$ by its submonoid consisting of triples of the form $[E,\gamma ,\gamma]$. 
\end{defn}
This notation is justified by the fact that the group $\prescript{\phi}{} \cK^{q-p,c,\sigma}_P(X)$ depends only on the difference $q-p$ as is shown in \cite{gomiFreedMooreKtheory2017}*{Subsection 4.3}.

It was proved in \cite{gomiFreedMooreKtheory2017}*{Theorem 4.11, Theorem 4.20} (a similar result is also shown in \cite{kubotaNotesTwistedEquivariant2016}*{Theorem 5.14}) that there is an isomorphism 
\[\vartheta \colon \prescript{\phi}{} \K_P^{q-p,\ft}(X) \to \prescript{\phi}{} \cK_P^{q-p,\ft} (X). \]
The map is defined in \cite{denittisKtheoreticClassificationDynamically2019}*{Lemma 5.16} and also in \cite{kubotaNotesTwistedEquivariant2016} implicitly (the paragraph above Lemma 5.12 and the proof of Theorem 5.14). 
In short, $\vartheta$ is defined in the following way; for any element of $\prescript{\phi}{}{\K}_P^{*,\ft}(X)$, we choose its representative $F \in \Gamma_c (X, \Fred_{p,q}(\cH))^P$ in the way that $F^2-1$ is supported into a finite rank $(\phi,\ft)$-twisted $P$-equivariant subbundle $E$ of $\cH$, and define as $\vartheta ([F]) = [E, \gamma_F , \gamma ]$, where
\[ \gamma_F:= -e^{\pi F \gamma } \gamma = -\cos(\pi F) \gamma - \sin (\pi F).\] 
Note that the conventions of \cite{denittisKtheoreticClassificationDynamically2019} and \cite{kubotaNotesTwistedEquivariant2016} are different by $-1$. 
Indeed, the definition of $\vartheta ([F])$ given in \cite{kubotaNotesTwistedEquivariant2016}  is $[E, -e^{-\pi F \gamma }\gamma , \gamma]$. 
\begin{rmk}
For simplicity of notation, we identify the Fredholm picture $\prescript{\phi}{}\K_P^{*,\ft}(X)$ with the Karoubi picture $\prescript{\phi}{}{\mathcal{K}}_P^{*,\ft}(X)$ by the isomorphism $\vartheta$ and just use the same notation $\prescript{\phi}{}\K_P^{*,\ft}(X)$ for the both groups.
\end{rmk}

\subsection{Operations in twisted equivariant K-theory}
Next we introduce several operations appearing in twisted equivariant K-theory. They are compared with the corresponding operations in Kasparov's KK-theory in Section \ref{section:5}.
\subsubsection{Pull-back}
Let $X$ and $Y$ be $P$-spaces and let $(\phi , c,\sigma ) $ be a twist of $Y $. Let $f \colon X \to Y$ be a $P$-equivariant continuous map. The pull-back of $(\phi,c,\tau )$ by $f$ is defined to be $f^*(\phi,c,\sigma):=(f^*\phi,f^*c, f^*\sigma)$, where $f^* \phi, f^*c \in H^1_P(X; \bZ_2)$ and $ f^*\sigma \in H^2_P(X; \prescript{f^*\phi }{}{\bT})$. Let $\ft:=(c,\sigma)$.

Let $\cH_X$ and $\cH_Y$ denote the universal $f^*(\phi,\ft)$-twisted $P$-equivariant Hilbert bundle on $X$ and the universal $(\phi,\ft)$-twisted $P$-equivariant Hilbert bundle $Y$ respectively. Consider the pull-back $f^* \cH_Y$, which is equipped with a canonical $f^*(\phi,\ft)$-twisted action of $P$. 
Moreover, the pull-back also induces a continuous map
\[ f^* \colon \Gamma (Y, \Fred (\cH))^P \to \Gamma (X, \Fred (f^*\cH))^P.\]
We define the pull-back
\[ f^* \colon \prescript{\phi}{} \K^{*,\ft}_P(Y) \to \prescript{f^*\phi}{}\K^{*,f^*\ft}_P(X) \]
as
\begin{align}
    f^*[F] = [V (f^*F \oplus G) V^*],\label{eq:pull}
\end{align} 
where $V \colon f^* \cH_Y \oplus \cH_X \to \cH_X $ is a $P$-equivariant even unitary respecting the $\Cl_{p,q}$-action (which exists by Remark \ref{rmk:univ} (3)) and $G \in \Gamma (X;\cU_{p,q}^1(\cH_X))^P$. This definition is independent of the choice of $V$ and $G$ by Lemma \ref{lem:unitary}.

\subsubsection{Open embedding}
Let $X$ be a compact $P$-space and let $Y$ be a $P$-invariant closed subspace. Let $(\phi,c,\sigma)$ be a twist on $P \curvearrowright X$ and let $\cH$ be the universal $(\phi,c,\sigma)$-twisted Hilbert bundle of $P \curvearrowright X$. Let $\ft:=(c,\sigma)$.

Let $F \in \Gamma_c(X\setminus Y, \Fred_{p,q} (\cH) ) ^P$. Then there is an open subspace $U \subset X$ including $Y$ such that $F|_{U \setminus Y}$ is a unitary-valued section. Since the space $\Gamma(U \setminus Y, \cU_{p,q}^1(\cH))^P$ is contractible, there is a continuous path $F_t \in \Gamma(U \setminus Y, \cU_{p,q}^1(\cH))^P$ such that $F_0=F|_{U \setminus Y}$ and $F_1$ extends to a section on $U$. Let $\rho$ be a $[0,1]$-valued continuous function on $X$ supported on $U$ and $\rho|_Y \equiv 1$. Now
\[F'(x):= \begin{cases}F(x) & x \in X \setminus U, \\ F_{\rho(x)}(x) & x \in U, \end{cases} \]
is a section in $\Gamma (X, \Fred_{p,q}(\cH))^P$, which is homotopic to $F$ in $\Gamma_c (X \setminus Y, \Fred_{p,q} (\cH))^P$. 

Let $\iota \colon X\setminus Y \to X$ denote the inclusion. Now $\iota_* [F]:= [F']$ determines the well-defined open embedding map
\[\iota_* \colon \prescript{\phi}{} \K_P^{p-q, \ft }(X \setminus Y) \to \prescript{\phi}{} \K_P^{p-q, \ft }(X). \]
Here we use the same letter $\ft$ for its restriction to $X \setminus Y$ for simplicity of notation.

\subsubsection{External and internal products}\label{section:3.2.3}
Let $P$ be a compact group, let $Z$ be a compact $P$-space and let $\phi \in H_P^1(Z;\bZ_2)$. 
Let $X_1$ and $X_2$ be $P$-spaces over $Z$, that is, $X_i$ are equipped with a continuous $P$-equivariant map $f_i \colon X_i \to Z$. Then the fiber product space $X_1 \times _Z X_2$ is also a $P$-space over $Z$ by $f(x_1,x_2):=f_1(x_1)$. For $i=1,2$, we consider twists of $X_i$ of the form $(f_i^* \phi, c_i, \sigma_i)$. Let $\ft_i:=(c_i,\sigma_i)$. 
Then there is an external product operation as
\[ {\cdot} \otimes _Z {\cdot} \colon \prescript{f_1^*\phi}{} \K_{G}^{n_1 , \ft_1}(X_1) \otimes \prescript{f_2^*\phi}{} \K_{G}^{n_2 , \ft_2}(X_2) \to \prescript{f^*\phi}{} \K_{G}^{n_1+n_2, \ft_1 + \ft_2}(X_1 \times_Z X_2). \] 
Here, for simplicity of notation we use the same letter for a twist $\ft_i$ on $X_i$ and its pull-back to $X_1 \times _Z X_2$. 

In the Fredholm picture of twisted equivariant K-theory as in Definition \ref{defn:Fred}, this homomorphism is defined in the following way. 
The collection of bundle maps  
\[ u_p^1 \hotimes_Z u_p^2 \colon \cH_1 \hotimes_Z  \cH_2 \to \cH_1 \hotimes_Z \cH_2\] 
determines the structure of a $\phi$-linear $(\ft_1 + \ft_2)$-twisted $P$-equivariant bundle since
\begin{align*}
&(u_p^1 \gamma_1^{c_1(p)} \otimes u_p^2)(u_{q}^1 \gamma_1 ^{c_2(q)} \otimes u_{q}^2)\\
=& (-1)^{c_1(p)c_2(q)}\sigma_1(p,q)\sigma_2(p,q) (u_{pq}^1 \gamma^{c(pq)}_1 \otimes u_{pq}^2). 
\end{align*}
When the Clifford algebra symmetry is taken into account, we canonically identify $\Cl_{p_1,q_1} \hotimes \Cl_{p_2,q_2} $ with $\Cl_{p,q}$ where $p:=p_1 + p_2$ and $q:=q_1 + q_2$. This identification gives rise to isomorphisms $\Delta_{p_1,q_1} \hotimes \Delta_{p_2,q_2}  \cong \Delta_{p,q}$ and
\[ \ (\cH_1 \hotimes \Delta_{p_1,q_1}) \hotimes (\cH_2) \hotimes \Delta_{p_2,q_2} \to \cH_1 \hotimes_Z \cH_2 \hotimes \Delta_{p,q}. \]

\begin{lem}\label{lem:prod}
Let $F_i \in \Gamma _c(X_i, \Fred_{p_i,q_i}(\cH_i))^P$ for $i=1,2$. Then the sections
\begin{enumerate}
    \item $F_1 \hotimes_Z (1-F_2^2)^{1/2} + 1 \hotimes_Z F_2$,
    \item $(F_1 \hotimes 1 + 1 \hotimes F_2)\cdot \chi(F_1^2 \hotimes 1 + 1 \hotimes F_2^2)$, where $\chi (x):=\max \{ 1 , |x|^{-1/2}\} $,
\end{enumerate}
both lie in $\Gamma_c(X_1 \times_Z X_2, \Fred_{p,q} (\cH_1 \hotimes \cH_2))^P$, and are mutually homotopic.
\end{lem}
\begin{proof}

Since $F_1^2 \hotimes 1$ commutes with $1 \hotimes F_2^2$, the functional calculus $f(F_1^2 \hotimes 1, 1 \hotimes F_2^2)$ is defined for any continuous function $f \colon [0,1]^2 \to \bC$.  
In general, if we have a pair of continuous functions $(f_1,f_2)$ such that $f_i \colon [0,1]^2 \to [0,1]$ such that $xf_1(x,y)^2 + yf_2(x,y)^2 \equiv 1$ on $[0,1] \times \{1\} \cup \{1 \} \times [0,1]$, the section
\[ (F_1 \hotimes 1) \cdot  f_1(F_1^2 \hotimes 1, 1 \hotimes F_2^2) + (1 \hotimes F_2) \cdot f_2(F_1^2 \hotimes 1, 1 \hotimes F_2^2) \]
lies in $\Gamma_c(X_1 \times_Z X_2, \Fred (\cH_1 \hotimes \cH_2))^P$.
Indeed,
\begin{align}
    &1-((F_1 \hotimes 1) \cdot  f_1(F_1^2 \hotimes 1, 1 \hotimes F_2^2) + (1 \hotimes F_2) \cdot f_2(F_1^2 \hotimes 1, 1 \hotimes F_2^2))^2 \notag \\
    =& 1- (F_1^2 \hotimes 1) f_1(F_1^2 \hotimes 1, 1 \hotimes F_2^2)^2 - (1 \hotimes F_2^2)f_2(F_1^2 \hotimes 1, 1 \hotimes F_2^2)^2 \notag \\
    =& (1-xf_1^2 -yf_2^2)(F_1^2 \hotimes 1, 1 \hotimes F_2^2), \label{eq:cpt}
    \end{align}
is compact at any fiber since the joint essential spectrum (cf.~\cite{davidsonAlgebrasExample1996}*{p.255}) of $(F_1^2 \hotimes 1, 1 \hotimes F_2^2)$ lies in $[0,1] \times \{1\} \cup \{1 \} \times [0,1]$. Moreover, the section (\ref{eq:cpt}) is compactly supported since it vanishes at $(x_1,x_2) \in X_1 \times _Z X_2$ if either $F_1(x_1)$ or $F_2(x_2)$ is a unitary.

Now, we obtain (1) if we choose $(f_1,f_2)$ as $f_1(x,y)=(1-y)^{1/2}$ and $f_2 \cong 1$. Also, if we choose $(f_1,f_2) $ as $f_1(x,y) = f_2(x,y)=\chi(x+y)$, then we obtain (2). This finishes the proof since the space of functions $(f_1,f_2)$ as above is connected. 
\end{proof}

Let $\cH$ denote the universal $(\phi,c,\sigma)$-twisted $P$-equivariant Hilbert bundle on $X_1 \times_Z X_2$.
\begin{defn}\label{defn:prod}
We define the external product
\[ \prescript{\phi}{} \K_P^{n_1,\ft_1}(X_1) \otimes \prescript{\phi}{} \K_P^{n_2,\ft_2}(X_2) \to \prescript{\phi}{} \K_P^{n_1+n_2,\ft_1+\ft_2}(X_1 \times_Z X_2) \]
by 
\begin{align*}
[F_1] \otimes_Z [F_2] &:= [V((F_1 \hotimes _Z (1-F_2^2)^{1/2} + 1 \hotimes F_2) \oplus G)V^*]\\
&=[V((F_1 \hotimes 1 + 1 \hotimes F_2)\cdot \chi(F_1^2 \hotimes 1 + 1 \hotimes F_2^2) \oplus G)V^*],
\end{align*}
where $V \colon (\cH_{1,p_1,q_1} \hotimes _Z \cH_{2,p_2,q_2}) \oplus \cH_{p,q} \to \cH_{p,q}$ is an even $P$-equivariant unitary commuting with $\Cl_{p,q}$-action and $G \in \Gamma(X_1 \times_Z X_2, \cU_{p,q}^1(\cH))^P$. 
Particularly, in the case of $X_1=Z=X_2=X$, the homomorphism
\[ \prescript{\phi}{} \K_P^{n_1,\ft_1}(X) \otimes \prescript{\phi}{} \K_P^{n_2,\ft_2}(X) \to \prescript{\phi}{} \K_P^{n_1+n_2,\ft_1+\ft_2}(X) \]
is called the internal (or cup) product.
\end{defn}
This is well-defined independent of the choice of $V$. 
Moreover, this map is obviously additive in the first and second components. 

\begin{lem}\label{lem:ring}
The map given in Definition \ref{defn:prod} satisfies the following properties:
\begin{enumerate}
\item (associative)  $([F_1] \hotimes [F_2]) \hotimes [F_3] = [F_1] \hotimes ([F_2] \hotimes [F_3])$.
\item (graded commutative) $[F_1] \hotimes [F_2] = (-1)^nf^*[F_2] \hotimes [F_1]$, where $f \colon X_1 \times_Z X_2 \to X_2 \times_Z X_1$ is the flip and $n:=(q-p)(q'-p')$.
\end{enumerate}
\end{lem}
By this lemma, the internal product induces a graded commutative ring structure on $\bigoplus _{n,\ft} \prescript{\phi}{} \K_P^{n,\ft}(X)$. 

\begin{proof}
For simplicity of notation, in the proof of this lemma, we use the same letter $F_1$ for the tensor product $F_1 \hotimes 1$ ($F_2$, $F_3$ are also used in a similar fashion). 
To see (1), note that $(F_1 + F_2)\chi (F_1^2+F_2^2) = \psi (F_1 + F_2)$, where $\psi(t):=t \max \{ 1, |t|^{-1} \}$. Then the homotopy
\[\tilde{F}_t:=\psi(t(F_1 + F_2) + (1-t)\psi(F_1 + F_2 ) + F_3)  \]
connects $\psi(\psi(F_1 + F_2) + F_3)$ with $\psi(F_1+F_2+F_3)$ in $\Gamma _c (X,\Fred_{p,q}(\cH))^P$. 
The same argument also connects $\psi(F_1 + \psi(F_2 + F_3))$ with $\psi (F_1 + F_2 + F_3)$.

The claim (2) follows from the commutativity of the diagram
\[ \xymatrix{
\Cl_{p,q} \hotimes \Cl_{p',q'} \ar[r] \ar[d]^{\mathrm{flip}} & \Cl_{p+p',q+q'} \ar[d]^{\Ad (x)} \ar[r] & \bB(\Delta_{p+p',q+q'}) \ar[d]^{\Ad (\fc (x))}   \\
\Cl_{p',q'} \hotimes \Cl_{p,q} \ar[r] & \Cl_{p'+p,q'+q} \ar[r] & \bB(\Delta_{p+p',q+q'}),
} 
\]
where $x=v_1 \cdots v_k \in \Cl_{p,q}$ is a product of vectors in $V$ such that the composition $R_{v_1} \cdots R_{v_k}$, where $R_v$ denotes the reflection along $v$, is the linear automorphism $\theta \colon \bR^{p + p',q+q'}\to \bR^{p' + p,q'+q} $ exchanging the basis. Note that $k$ is odd if and only if $\theta $ changes the orientation, i.e., $ (p+q)(p'+q') \equiv n$ is odd. Now
\[ f^* \hotimes \fc(x) \colon (\cH_1 \hotimes _Z \cH_2) \hotimes \Delta_{p+p',q+q'} \to (\cH_2 \hotimes_Z \cH_1) \hotimes \Delta_{p'+p,q'+q}  \]
implements the map between the spaces of Fredholm sections, which maps $F$ to $(-1)^nf^*F$.
\end{proof}

\begin{rmk}\label{rmk:prod}
In this paper we only consider the case that $Z=\pt$, in other words, $\phi$ is the pull-back of an element of $ H_P^1(\pt;\bZ_2)$ (denoted by the same letter $\phi$). 
In this case the external product
\[ \prescript{\phi}{} \K_P^{n_1,\ft_1}(X_1) \otimes \prescript{\phi}{} \K_P^{n_2,\ft_2}(X_2) \to \prescript{\phi}{} \K_P^{n_1+n_2,\ft_1+\ft_2}(X_1 \times X_2) \]
is defined. Moreover, if $\phi$ is chosen in this way, the product given in Definition \ref{defn:prod} satisfies
\[ [F_1] \otimes _Z [F_2] = \Delta_Z ^* ([F_1] \hotimes [F_2]),\]
where $\Delta_Z \colon X_1 \times_Z X_2 \to X_1 \times X_2$ is the inclusion.
\end{rmk}

\subsubsection{Thom isomorphism}
The equivariant Bott periodicity and the Thom isomorphism is discussed in \cite{gomiFreedMooreKtheory2017}*{Subsection 3.6}. Here we describe the Thom isomorphism map by means of the product introduced in Subsection \ref{section:3.2.3} for our convenience. 

Let $q \colon E \to X$ be a $P$-equivariant real vector bundle on a $P$-space $X$ of rank $r$. For simplicity, here we only consider the case that an orientation and a spin structure on $E$, which is not necessarily $P$-equivariant, is fixed. Let $\mathrm{SO}(E)$ and $\mathrm{Spin}(E)$ denote the corresponding principal bundles.

Set $E':=E \oplus \underline{\bR}^{8n-r}$, where $\underline{\bR}^{8n-r}$ denotes the trivial bundle of rank $8n-r$ with the trivial $P$-action. We associate to $E$ and $E'$ the Clifford algebra bundles $\Cl(E)$ and $\Cl(E')$ respectively. That is, $\Cl(E):= \mathrm{SO}(E) \times _{\mathrm{SO}(r)} \Cl_r$ and $\Cl(E'):=\mathrm{SO}(E') \times_{\mathrm{SO}(8n)} \Cl_{8n}$. Note that $\Cl(E')$ is isomorphic to the graded tensor product $\Cl(E) \hotimes \Cl_{8n-r}$. 
Let $S$ denote the spinor bundle of $E'$, i.e., the vector bundle on $X$ whose fiber at $x$ is the unique irreducible representation of $\Cl(E')$, i.e., $S:=\mathrm{Spin}(E) \times _{\mathrm{Spin}(8n)} \Delta$.

The $\bR$-linear actions of $P$ onto $E$ and $E'$ give rise to algebra actions onto $\Cl(E')$, and hence a $(\nu,\mu) $-twisted action onto $S$ in the following way. 
Let $\nu \in H^1_P(X;\bZ_2)$ be the $1$-cocycle corresponding to the orientation bundle $\mathrm{Or}(E)$ on which $P$ acts canonically. 
Also, consider the following data;
\begin{itemize}
    \item the collection of $\bZ_2$-bundles $\{ L_p \}$ on $X$ defined as
    \[L_p:=\bigsqcup_{x \in X} \{ T \in \Hom (S_x,S_{px}) \mid \Ad (T)=\alpha_p \}, \]
    \item the unitary $\mu (p,q) \colon L_p \otimes L_q \to L_{pq}$ given by the composition, and
    \item the unitary $u_p \colon S \otimes L_p \to \alpha_p^*S$ given by $(\xi,T) \mapsto T\xi$.
\end{itemize}
Then $(L,\mu)$ corresponds to a $2$-cocycle $\mu \in H^2_P(X; \bZ_2)$. Let us use the same letter $\mu$ for its image in $H^2_P(X;\prescript{\phi}{} \bT)$ through the homomorphism induced from the inclusion of local systems $\bZ_2 \subset \prescript{\phi}{} \bT$. Then $u$ determines a $(\phi,\nu,\mu)$-twisted $P$-action on $S$.
We write as $\fv:=(\nu,\mu)$. Note that $\nu$, $\mu$ are the same thing as the $P$-equivariant Stiefel-Whitney classes $w_1^P(E)$ and $w_2^P(E)$ respectively.

Let $\xi \colon E \to q^*E$ denote the section defined by $\xi(e):=e$ for any $e \in E_x$. Let $C$ denote the bounded section
\begin{align}
    C :=\fc(\xi)(1+\|\xi \|^2)^{-1/2} \in \Gamma  (E, q^*\Cl(E')), \label{eq:C}
\end{align} 
which acts on the bundle $q^*S$ on $E$ by the fiberwise multiplication. Moreover $[C, \fc(v)]=0$ for any $v \in \Cl_{8n-r,0}$. 
Let $\cH$ be a universal $(\phi,-\fv)$-twisted $P$-equivariant bundle on $E$. We define the Bott element 
\begin{align}
 \beta_E:= [V(C \oplus G)V^* ] \in \prescript{\phi}{} \K_{P}^{r, \fv } (E), \label{eq:Bott}
\end{align}
where $V \colon \pi^*S \oplus \cH_{8n-r,0} \to\cH_{8n-r,0}$ is a $P$-equivariant even unitary respecting the $\Cl_{8n-r,0}$-action and $G \in \Gamma (E, \cU_{8n-r,0}^1(\cH))^P$. (Here we omit $8n$ in the degree of the K-group).
\begin{defn}\label{defn:Thom}
The Thom isomorphism map is defined by the cup product over $X$ as
\[
\Thom _E:=\beta_E \otimes_X \blank \colon \prescript{\phi}{} \K_P^{*,\ft}(X) \to \prescript{\phi}{} \K_P^{*+r,\ft + \fv}(E). 
\]
\end{defn}
By Remark \ref{rmk:prod} the Thom isomorphism is written as
\begin{align}
    \Thom_E(x)= \Delta_E ^*( x \otimes \beta_E), \label{eq:Thom}
\end{align}
where $\Delta_E:=(\id_E, q) \colon E \to E \times X$.

\begin{prp}
For any real $P$-equivariant vector bundle $E$ on $X$, the homomorphism $\Thom_E$ is an isomorphism. 
\end{prp}
\begin{proof}
If $E$ is the trivial bundle $X \times V$, where $V$ is a real representation of $P$, then $\Thom_{X \times V}$ is the external product with $\beta_V \in \prescript{\phi}{} \K_P^{r,\fv}(V)$ over the point. 
It will be proved in Lemma \ref{lem:BottKK} that this is an isomorphism by using Kasparov theory. 
Now a Mayer--Vietoris argument concludes the proof. 
\end{proof}

\subsubsection{Push-forward}
Let $M$ and $N$ be $P$-manifolds and let $f \colon M \to N$ be a $P$-equivariant smooth map. 
We choose a spin representation $V$ of $P$ and a $P$-equivariant embedding $j \colon M \to V$. Let $E$ denote the normal bundle of the embedding $(j,f) \colon M \to V \times N $. 

Set
\begin{align*}
     \fv&=(\mu,\nu):=(w_1^P(E) , w_2^P(E))\\
     &=f^*(w_1^P(N), w_2^P(N))-(w_1^P(M) , w_2^P(M)).
\end{align*}
Note that $f$ is said to be spin-oriented if $\fv=0$.

\begin{defn}\label{defn:push}
We define the push-forward by $\pi$ as
\[\pi_!:= \Thom_{N \times V}^{-1} \circ \iota_* \circ \Thom _E \colon \prescript{\phi}{} \K_P^{*,f^*\ft - \fv}(M) \to \prescript{\phi}{} \K_P^{-*,\ft }(N). \]
\end{defn}
It is seen in the same way as \cite{atiyahIndexEllipticOperators1968a}*{Section 3} that the push-forward is independent of the choice of $j$.

\subsection{Classification of gapped topological phases}\label{section:3.3}
In condensed matter physics, a (non-interacting) quantum system on a lattice with space group symmetry is studied through the Pontrjagin dual of the lattice (Brillouin torus) with an action of the point group. We generalize in this section the action of the point group on the Pontrjagin dual, to that of the twisted point group.

Let $(G,\phi,c,\tau)$ be a twisted crystallographic group in the sense of Definition \ref{defn:twcry}.
Firstly we define the action of the twisted crystallographic group $G$ on the Pontrjagin dual $\hat{\Pi}:= \Hom (\Pi , \bT)$ of $\Pi$ 
as
\begin{align}
     \rho_g (\chi) (t) := \left\{ \begin{array}{ll} \overline{\chi(g^{-1}tg)} & \text{ if $\phi(g)=1$,} \\  \chi(g^{-1}tg) & \text{ if $\phi(g)=0$,} \end{array} \right. \label{eq:Piact}
\end{align} 
for $\chi \in \hat{\Pi}$ and $g \in G$. 
Since $\rho_t$ is trivial for any $t \in \Pi$, this $\rho$ is reduced to the action of the twisted point group $P$. 

We also define the $2$-cocycle $\sigma \in H^2_P(\hat{\Pi},\prescript{\phi}{} \bT )$ in the following way: Let $L_p$ be the trivial bundle $\underline{\bC}$ over $\hat{\Pi}$.
Let us fix a set-theoretic section $s \colon P  \to G$ and define the continuous map $\sigma \colon P \times P \to \End(\underline{\bC})$ as
\begin{align}
    \sigma(\chi , p,q) :=\chi(s(p)s(q)s(pq)^{-1}) \label{eq:sigma}
\end{align}
for any $p,q \in P$. 
It is straightforward to check that this $\sigma$ is a $2$-cocycle of the action $\rho \colon P \curvearrowright \hat{\Pi}$ with coefficients in $\prescript{\phi}{} \bT $.

This $(\rho,\sigma)$ is related to unitary representations of $G$ through the Fourier transform in the following way. 
Let $(G,\phi,c,\tau)$ be a twisted crystallographic group acting on the Euclidean space $V$ and the internal Hilbert space $\sK$. So there is a $\phi$-twisted projective unitary representation of $G$ on $L^2(V)\otimes\sK$. The full dynamics (Hamiltonian operator which graded-commutes with $G$) is usually reducible at low energy to a sub-representation space, spanned by an orthonormal set of localized orbitals centered on lattice points thought of as atomic positions. Such a basis is called a (localized) Wannier basis, see also the discussion of atomic insulators in Section \ref{sec:atomic.insulator}. This is formalized as follows.

Let $\bX \subset V$ be a $G$-invariant discrete subspace and let us fix a choice $\bX_0$ of a fundamental domain of the $\Pi$-action onto $\bX$, i.e., the image of a section of the quotient map $\bX \to \bX/\Pi$. Let $\sH$ denote the Hilbert space $\ell^2(\bX, \sK)$ with the $(\phi,c,\tau)$-twisted unitary representation $U:=\lambda \otimes v$ of $G$. Here $\lambda $ is the regular representation on $\ell^2(\bX)$, i.e., $(\lambda_g\xi)(x) :=\xi (g^{-1}x)$, and $v \colon G \to \qAut(\sK)$ is a lift of $k \colon G \to \qAut(\bP\sK)$. 
By using this $U$ we identify $\sH$ with $\ell^2(\Pi,\tilde{\sK})$, where $\tilde{\sK}
:=\ell^2(\bX_0,\sK)$. 

The canonical basis of $\ell^2(\Pi)$, along with a basis for $\tilde{\sK}$, should be thought of as mutually orthogonal wave functions in $L^2(V)\otimes\sK$ localized at the points of $\bX$, and the effective Hilbert space $\sH$ as a subrepresentation in $L^2(V)\otimes\sK$. When one works at this effective level, the (discrete) Fourier transform gives a unitary isomorphism
\[ \cF \colon \ell^2(\Pi,\tilde{\sK}) \to L^2(\hat{\Pi}, \tilde{\sK}), \]
where the $L^2$-space on $\hat{\Pi}$ is defined by the Haar measure normalized as $\mathrm{vol}(\hat{\Pi})=1$. 
\begin{lem}\label{lem:bdlK}
For each $g \in G$, there is a bundle map $u_g \colon \hat{\Pi} \times \tilde{\sK} \to \hat{\Pi} \times \tilde{\sK}$ covering $\rho_g$ such that $\cF U_g \cF^* = u_g$ as a $\phi$-linear unitary on $L^2(\hat{\Pi}, \tilde{\sK})$.
\end{lem}
\begin{proof}
Firstly we remark that, for any $t\in \Pi$, $\cF U_t \cF^*$ is the multiplication with $t$ regarded as a continuous function on $\hat{\Pi}$ by $\hat{t} (\chi):=\chi(t)$.

Let us decompose $\tilde{\sK}$ into the direct sum of subspaces $\sK_x=\ell^2(\{ x \}, \sK)$ where $x$ runs over $\bX_0$. For $\xi \in L^2(\hat{\Pi},\tilde{\sK})$, $\xi_x$ denotes its $L^2(\hat{\Pi},\sK_x)$-component with respect to this decomposition. Let $\lambda^0$ denote the regular representation of $G$ acting on $\bX/\Pi \cong \bX_0$. 
Then, for any $\xi \in L^2(\hat{\Pi},\tilde{\sK})$ and $g\in G$, we have
\[\cF U_g \cF^* \xi = \sum_{x \in \bX_0} U_{t(x,g)} \cdot (\lambda_g^0 \otimes v_g) \xi_x, \]
where $t(x,g) \in \Pi$ is the element characterized by $gx \in t(x,g)\bX_0$. 
Therefore, 
\[u_g(\chi, \xi):= \Big(\rho_g x, \sum_{x \in \bX_0} \chi(t(x,g)) \cdot (\lambda_g^0 \otimes v_g)\xi_x \Big) \]
is the desired bundle map.
\end{proof}
Let $s \colon P \to G$ be a set-theoretic section. Then we have
\[ u_{s(p)}u_{s(q)}u_{s(pq)}^{-1}=\sigma(\chi , p,q) \cdot   \tau (p,q) \]
for any $p,q \in P$. That is, the collection of bundle maps $\{ u_{s(p)} \}_{p \in P}$ forms a $(\phi,c,\tau+\sigma)$-twisted $P$-equivariant vector bundle on $\hat{\Pi} \times \sK$.

Let $\sH^N$ denote the direct sum of $N$ copies of $\sH$, which is isomorphic to $\ell^2(\bX,\sK^N)$. 
For $x,y \in \bX$ and $H \in \bB(\sH^N)$, $H_{xy}$ denotes the operator $\delta_y  H \delta_x$, where $\delta_x$ denotes the delta function on $x$, regarded as a bounded operator on $\sK^N$. The following lemma is a standard fact in Fourier analysis.
\begin{lem}\label{lem:TP}
Let $H$ be a bounded operator on $\sH^N$ satisfying the following two properties:
\begin{enumerate}
\item There are constants $C_1, C_2>0$ such that, for any $x,y \in \bX$, the coefficient $H_{xy}$ satisfies $\|H_{xy} \| \leq C_1e^{-C_2d(x,y)}$. 
\item The operator $H$ is $\Pi$-invariant, i.e., $U_t H U_t^*=H$ for any $t \in \Pi$.
\end{enumerate}
Then $\cF H \cF^* \in \bB(L^2(\hat{\Pi} , \tilde{\sK}))$ is the multiplication operator with a smooth function $h \colon \hat{\Pi} \to \bB(\tilde{\sK})$. 
\end{lem}

Let $\cH_N(G,\phi,\ft)$ denote the set of operators on $\sH^N$ for some $N$, satisfying (1) of Lemma \ref{lem:TP} and 
\begin{enumerate}
    \item[(2')] $H$ is graded $G$-invariant, i.e., $U_g H U_g^*=(-1)^{c(g)}H$ for any $g \in G$,
    \item[(3)] $H$ is self-adjoint and invertible.
\end{enumerate}
We embed $\cH_N(G,\phi,\ft)$ into $\cH_{N+1}(G,\phi,\ft)$ by $H \mapsto \big( \begin{smallmatrix} H & 0 \\ 0 & \gamma \end{smallmatrix} \big)$ and define the set of topological phases with the symmetry $(G,\phi,\ft)$ as
\[ \mathscr{TP}(G,\phi,\ft) := \Big( \bigcup_{N} \cH_N (G,\phi,\ft)\Big) /\sim, \]
where the equivalence relation is given by the homotopy of operators. 
The set $\mathscr{TP}(G,\phi,\ft)$ is obviously related to Karoubi's picture of the twisted equivariant K-theory.   
The following proposition is a slight modification of \cite{freedTwistedEquivariantMatter2013}*{Theorem 10.15} which is essentially proved in \cite{kubotaControlledTopologicalPhases2017}*{Proposition 3.4}. 
\begin{prp} \label{prp:TP}
The map $[H] \mapsto [\underline{\sK}^N , \mathrm{sgn}(h) , \gamma]$, where $h \colon \hat{\Pi} \to \bB(\sK^N)$ is the continuous function as in Lemma \ref{lem:TP}, is an isomorphism between $\mathscr{TP}(G,\phi,\ft)$ and $\prescript{\phi}{} K^{0,\ft+\sigma}_P(\hat{\Pi})$. 
\end{prp}
In particular, this proposition shows that the twisted equivariant K-group $\prescript{\phi}{} K^{0,\ft+\sigma}_P(\hat{\Pi})$ is independent of the choice of $\Pi$.

\begin{rmk}
When $\phi,c,\tau$ are trivial, a concrete meaning to the class of $H$ being non-trivial is as follows. Namely, the range of the negative-energy projection $\frac{1-\mathrm{sgn}(h)}{2}$ cannot be spanned by the Fourier transform of a localized Wannier basis, see \cite{ludewigGoodWannierBases2020}.
\end{rmk}

\section{Crystallographic T-duality}\label{section:4}
In this section we introduce the main research target of the paper, two twisted equivariant K-groups associated to a twisted crystallographic group and the crystallographic T-duality homomorphism between them.  

Let $(G,\phi,c,\tau)$ be a $d$-dimensional twisted crystallographic group in the sense of Definition \ref{defn:twcry} acting on a $d$-dimensional Euclidean space $V$ and let  $\Pi:= G \cap \bR^d \times \{1\}$.  
Let $\rho$ denote the action of $P=G/\Pi$ onto $\hat{\Pi}$ given in (\ref{eq:Piact}) and let $\sigma \in H_P^2(\hat{\Pi} ; \prescript{\phi}{} \bT)$ denote the $2$-cocycle given in (\ref{eq:sigma}). 
We also consider the $P$-action onto $V/\Pi$ induced from the $G$-action onto $V$.

Let $\fv=(\mu , \nu)$ be the $P$-equivariant Stiefel-Whitney class $(w_1^P(V), w_2^P(V)) \in H^1_P(\pt ;\bZ_2) \oplus H^2_P(\pt; \bZ_2)$ of $V$ as a $P$-bundle over the point.
Since the tangent bundle $T(V/\Pi) \to V/\Pi$ is $P$-equivariantly trivial, that is, $T(V/\Pi) \cong V/\Pi \times V$, its equivariant Stiefel-Whitney class is the pull-back $\pi^*\fv$ by the collapsing map $\pi \colon V/\Pi \to \pt$.

The crystallographic T-duality map is schematically represented by the diagram
\[ \xymatrix@R=2em{ &  V/\Pi \times \hat{\Pi} \ar@(lu,ru)[]^{[\cP] \otimes } \ar[ld]_{\hat \pi} \ar[rd]^{\pi } & \\ V/\Pi && \hat{\Pi}. } \]
Here, $\pi \colon V/\Pi \to \pt $ and $\hat{\pi} \colon \hat{\Pi} \to \pt$ are collapsing maps. Another ingredient is the element
\[ [\cP]  \in \prescript{\phi}{} \K_P^{0,\pi^*\sigma}(V/\Pi \times \hat{\Pi} )\] 
represented by the Poincar\'e line bundle in the following way. 

\begin{lem}[{\cite{gomiCrystallographicTduality2019}*{Theorem 4.3}}]
The $\phi$-twisted $G$-action on the Poincar\'e line bundle
\[\cP := (V \times \hat{\Pi} \times \bC)/\{ (v,\chi, z) \sim (v + t,\chi, \overline{\chi(t)}z )  \} \]
given by 
\[g \cdot [v,\chi,z]=[gv,\chi,\prescript{\phi(g)}{}z]\]
determines a $(\phi,0,\hat \pi ^*\sigma)$-twisted $P$-equivariant vector bundle on $V/\Pi \times \hat{\Pi}$.
\end{lem}
\begin{proof}
It suffices to show that the $\Pi$-action on $\cP$ is given by the multiplication by $\hat{\pi}^*\sigma_t \in C(V/\Pi\times\hat{\Pi},\bT)$, where $\sigma_t$ denotes the function $\hat{t} $ as before (i.e., the function defined by $\hat{t}(\chi):=\chi(t)$). This is checked as 
\begin{align*}
t \cdot [v,\chi,z]&= [v+t,\chi , z] = [v,\chi,\chi (t)z]=[v,\chi,\sigma_t(\chi) z]. \qedhere
\end{align*}
\end{proof}
The bundle $\cP$ determines an element of the Freed--Moore K-group (in the Karoubi picture) as 
\begin{align*} [\cP,1,-1] \in \prescript{\phi}{} \K^{0,\pi^*\sigma}_P(V/\Pi \times \hat{\Pi}), 
\end{align*}
which is simply written as $[\cP]$ hereafter.

\begin{defn}
Let $\ft \in H^1(P;\bZ_2) \oplus H^2(P;\prescript{\phi}{} \bT )$. We call the composition 
\begin{align}
\prescript{\phi}{}{\mathrm{T}} _G^{\ft }:= \pi_! \circ ([\cP] \otimes_{V/\Pi \times \hat{\Pi}} \blank )  \circ \hat{\pi}^* \colon \prescript{\phi}{} \K^{*+d,\ft - \fv}_{P}(V/\Pi) \to \prescript{\phi}{} \K_P^{*,\ft + \sigma}(\hat{\Pi})  \label{eq:Tdual}
\end{align}
the \emph{crystallographic T-duality map}.
\end{defn}
Here the homomorphisms 
\begin{enumerate}
\item $\hat \pi^* \colon \prescript{\phi}{} \K_P^{*+d,\ft - \fv}(V/\Pi) \to \prescript{\phi}{} \K_P^{*+d,\ft-\fv}(V/\Pi \times \hat{\Pi})$,
\item $[\cP] \otimes_{V/\Pi \times \hat{\Pi}} \blank \colon \prescript{\phi}{} \K_P^{*+d,\ft-\fv}(V/\Pi \times \hat{\Pi}) \to \prescript{\phi}{} \K_P^{*+d,\ft-\fv+\sigma}(V/\Pi \times \hat{\Pi})$, and 
\item $\pi_! \colon \prescript{\phi}{} \K_P^{*+d,\ft-\fv + \sigma} (V/\Pi \times \hat{\Pi} ) \to \prescript{\phi}{} \K_P^{*,\ft+\sigma}(\hat{\Pi})$
\end{enumerate}
are the pull-back (\ref{eq:pull}), the internal tensor product (Definition \ref{defn:prod}) with $[\cP]$ and the push-forward (Definition \ref{defn:push}) respectively, and we have suppressed the pullback notation for the twists.

Now we state the main theorem of the paper. 
\begin{thm}\label{thm:crystal}
The crystallographic T-duality map $\prescript{\phi}{}{\mathrm{T}}_G^{\ft} $ is an isomorphism for any $\phi \in H^1(P;\bZ_2)$ and $\ft \in H^1(P;\bZ_2) \oplus H^2(P;\prescript{\phi}{} \bT )$.
\end{thm}

\begin{rmk}\label{rmk:module}
Let $\prescript{\phi}{} R(P)$ denote the Grothendieck ring of finite dimensional $\phi$-twisted representations of $P$. In other words, $\prescript{\phi}{} R(P) :=\prescript{\phi}{} \K_P^0(\pt)$ with the ring structure given by Definition \ref{defn:prod}. The crystallographic T-duality map (\ref{eq:Tdual}) is actually a homomorphism of $\prescript{\phi}{} R(P)$-modules.
\end{rmk}

Here we introduce a simple application of Remark \ref{rmk:module}. Suppose that $G$ acts on $V$ freely, which implies that the induced action of $P$ onto $V/\Pi$ is also free. Hence the Lemma \ref{lem:AS} below, which is a generalization of the Atiyah--Segal completion theorem, implies that there is $n \in \bZ_{>0} $ such that $\prescript{\phi}{} I_P^n \cdot \prescript{\phi}{} \K_P^{*,\ft}(V/\Pi)=0$. 
Here $\prescript{\phi}{} I_P$ denotes the augmentation ideal $\ker (\prescript{\phi}{} \K_P^0(\pt) \to \K^0(\pt))$. Now Theorem \ref{thm:crystal} implies that the $\prescript{\phi}{} R(P)$-module $\prescript{\phi}{} \K_P^{*,\ft - \fv +\sigma}(\hat{\Pi})$ also satisfies $\prescript{\phi}{} I_P^n \cdot \prescript{\phi}{} \K_P^{*,\ft -\fv + \sigma}(\hat{\Pi})=0$, although $\hat{\Pi}$ is no longer a free $P$-space.

\begin{lem}\label{lem:AS}
Let $X$ be a free finite $P$-CW-complex. Then the $\prescript{\phi}{} R(P)$-module $\prescript{\phi}{} \K_P^{*,\ft}(X)$ satisfies $\prescript{\phi}{} I_P^n \cdot \prescript{\phi}{} \K_P^{*,\ft}(X)=0$ for some $n \in \bZ_{>0}$.
\end{lem}
\begin{proof}
By the assumption, there is an $n \in \bZ_{>0}$ and a continuous $P$-equivariant map $X \to E_nP$, where $E_nP$ is the join $P \ast \cdots \ast P$ of $n$ copies of $P$. This defines the $\prescript{\phi}{} \K_P^{0}(E_nP)$-module structure on $\prescript{\phi}{}\K_P^*(X)$. 
Hence it suffices to show that the pull-back by $ \pi _n \colon E_nP \to \pt$ factors through $\pi_n^* \colon \prescript{\phi}{} R(P)/\prescript{\phi}{} I_P^n \to \prescript{\phi}{} \K_P^{0}(E_nP) $. We show this by induction on $n$. 
First we observe that the claim holds when $n=1$. In this case $E_1P = P$ has the twisted equivariant K-group $\prescript{\phi}{}\K^{0}_P(P) \cong \K^0(\pt) \cong \bZ$ and hence  $\ker \pi_1^* = \prescript{\phi}{}I_P$.   

Next, assume that $\pi_k^*$ factors through $\prescript{\phi}{}R(P) / \prescript{\phi}{}I_P^k$ for $k=1, \cdots, n-1$.
Let $x_1, \cdots , x_n \in \prescript{\phi}{} I_P$ and set $x':=x_1 \cdot \cdots \cdot x_{n-1} \in \prescript{\phi}{} I_P^{n-1}$, $x:=x' \cdot x_n \in \prescript{\phi}{} I_P^n$. It suffices to show that $\pi_n^*(x)=0$.  
Let us consider the diagram 
\[
\xymatrix{
&\prescript{\phi}{} R (P) \ar@{=}[r] \ar[d]^{\pi_n^*} & \prescript{\phi}{} R(P) \ar[d]^{\pi_{n-1}^*} \\
\prescript{\phi}{} \K_P^0(E_nP,E_{n-1}P) \ar[r]  & \prescript{\phi}{} \K_P^0(E_nP) \ar[r]^{i_n^*} & \prescript{\phi}{} \K_P^0(E_{n-1}P), \\
}
\]
where $i_n \colon E_{n-1}P \to E_nP$ denotes the inclusion. 
Since the right square commutes, $i_n^* \circ \pi_n^*(x')=\pi_{n-1}^*(x')=0$ by the induction hypothesis. The exactness of the second row implies that $\pi_n^*(x')=\prescript{\phi}{} j_n^*(\xi)$ for some $\xi \in \prescript{\phi}{} \K^0_P(E_nP,E_{n-1}P)$. 
On the other hand, since $E_nP \setminus E_{n-1}P$ is $P$-equivariantly homeomorphic to the direct product $P \times C(\pi_{n-1})$ (where $C(\pi_{n-1})$ is the mapping cone space), we have $\prescript{\phi}{} I_P \cdot \prescript{\phi}{} \K_P^0(E_nP,E_{n-1}P)=0$. 
Hence we obtain
\[\pi_n^*(x)=x_n \cdot \pi_n^*(x') = x_n \cdot \prescript{\phi}{} j_n^*(\xi) = \prescript{\phi}{} j_n^*(x_n \cdot  \xi)=0. \qedhere \]
\end{proof}

\section{Proof of Theorem \ref{thm:crystal}}\label{section:5}
In this section we give a proof of Theorem \ref{thm:crystal}. To this end, we firstly identify the Freed--Moore twisted equivariant K-groups under consideration with the corresponding Chabert--Echterhoff KK-group \cite{chabertTwistedEquivariantKK2001}, and then identify the twisted crystallographic T-duality map with the Kasparov product with an element, which is known to be a KK-equivalence. 
After the proof, we also give two related discussions; the irrational generalization of the twisted crystallographic T-duality and a relation between the twisted crystallographic T-duality and the Baum--Connes assembly map.  

\subsection{Comparison of twisted equivariant K-theories}
As is summarized in Appendix \ref{section:app}, the $\phi$-twisted Chabert--Echterhoff KK-theory is a bivariant homology theory for a pair of $\phi$-twisted $(G,\Pi)$-C*-algebras. 
Here a $\phi$-twisted $(G,\Pi)$-C*-algebra is a triple $(A,\alpha, \sigma)$, where $A$ is a C*-algebra, $\alpha$ is a $\phi$-linear action of $G$ onto $A$, and $\sigma \colon \Pi \to U(\cM (A))$ such that $\sigma_t a \sigma_t^* = \alpha_t(a)$ and $\alpha_t(\sigma_t)=\sigma_{gtg^{-1}}$ for any $t \in \Pi$, $a \in A$ and $g\in G$. 
In this section we deal with the following three kinds of $\phi$-twisted $(G,\Pi)$-C*-algebras and their tensor product:
\begin{enumerate}
\item For a locally compact $P$-space $X$ and a $(\phi,\ft)$-twisted $P$-equivariant vector bundle $(\sV, v)$ on $X$, the C*-algebra $C_0(X, \bK(\sV)) \cong C_0(X) \hotimes \bK(\sV)$ of continuous sections of the algebra bundle $\bK(\sV)$ has a $\phi$-twisted $P$-C*-algebra structure given by $\alpha_p(T)(x)=v_pT(px)v_p^*$ for any $T \in C_0(X,\bK(\sV))$.
\item Let $\hat{\Pi}$ be the Pontrjagin dual of $\Pi$, on which $G$ acts by $\rho$ as in (\ref{eq:Piact}). We put a $G$-action $\prescript{\phi}{} \rho^*_g(f)= \prescript{\phi(g)}{}{(\rho^*_g(f))}$ on the C*-algebra $C(\hat{\Pi})$ (cf.\ Remark \ref{rmk:real}). Moreover, let us define the implementing unitaries $\sigma_t \in C(\hat{\Pi}) $ as $\sigma _t = \hat{t}$, i.e., $\sigma_t(\chi)=\chi(t)$ for $\chi \in \hat{\Pi}$. Then $(C(\hat{\Pi}), \prescript{\phi}{} \rho^* ,  \sigma )$ is a $(G,\Pi)$-C*-algebra, \emph{which is denoted by $C(\hat{\Pi})_\sigma$ hereafter}. 
\item The complex number field $\bC$ equipped with the $G$-action $z \mapsto \prescript{\phi(g)}{}z$, which is simply denoted by $\bC$ (cf.\ Remark \ref{rmk:real}). More generally, we deal with Clifford algebras regarded as a Real C*-algebra (cf.\ Remark \ref{rmk:Cliffordalgebra}).
\end{enumerate}

Let $A$ be a $\phi$-twisted $\bZ_2$-graded $(G,\Pi)$-C*-algebra. As is shown in Corollary \ref{cor:twKK},  the twisted equivariant K-group is defined as
\[ \prescript{\phi}{} \KK^{G,\Pi}(\Cl_{p,q}, A) :=\pi_0 \Fred_{p,q} (\prescript{\phi}{} \sH_A^{G,\Pi})^G, \]
where $\prescript{\phi}{} \sH_A^{G,\Pi}$ is the $\phi$-twisted $(G,\Pi)$-equivariant Hilbert $A$-module as in (\ref{eq:Hilbmod}). 
Note that if $A$ is a $P$-C*-algebra, then $\prescript{\phi}{} \sH_A^{G,\Pi}$ is isomorphic to $\prescript{\phi}{}  \sH_A^P := (\prescript{\phi}{} \ell^2(P) ^{\oplus \infty} \oplus (\prescript{\phi}{} \ell^2(P) ^{\oplus \infty})^{\mathrm{op}}) \hotimes A$. (Here $\prescript{\phi}{} \ell^2(P)$ denotes the space of $\ell^2$-functions on $P$ equipped with the usual summation and the complex multiplication determined by $(\lambda \cdot f)(p)=\prescript{\phi(p)}{}{\lambda} \cdot f(p)$, on which $P$ acts by the left regular representation.)
This is consistent with Remark \ref{rmk:twistedKK}. Throughout this section, we identify $\prescript{\phi}{}\KK^{G,\Pi}(A,B)$ and $\prescript{\phi}{}\KK^P(A,B)$ if both $A$ and $B$ are $\phi$-twisted $P$-C*-algebras (regarded as $(G,\Pi)$-C*-algebras as in Remark \ref{rmk:twisted}).

\begin{lem}\label{lem:move}
Let $X$ be a locally compact $P$ space. For $\ft=(c,\tau) \in H^1_P(X;\bZ_2) \oplus H^2_P(X;\prescript{\phi}{} \bT)$, let $\acute{\ft}$ denote the pair $(c,\acute{\tau})$ where $\acute{\tau}:=\tau + \epsilon(c,c)$. 
Let $(\acute{\sV}, \acute{v})$ be a $(\phi, \acute{\ft})$-twisted $P$-equivariant vector bundle on $X$. Then there is an isomorphism
\[\prescript{\phi}{} \K^{q-p,\ft}_P (X) \cong \prescript{\phi}{} \KK^P( \Cl_{p,q} ,C_0(X, \bK(\acute{\sV})) ). \]
\end{lem}
Note that for any twist $(\phi,c,\tau)$ of a $P$-space $X$, there is a twisted $P$-vector bundle on $X$ (cf.\ \cite{kubotaNotesTwistedEquivariant2016}*{Example 2.9}).
\begin{proof}
Let $(\sV, v)$ denote the $(\phi,\ft)$-twisted $P$-equivariant vector bundle on $X$ given by $\sV:=\acute{\sV}$ and $v_p:=\gamma^{c(p)}\acute{v}_p$. Then we have
\[ \sV \hotimes \acute{\sV}^* \cong \acute{\sV} \otimes \acute{\sV}^* \cong \bK(\acute{\sV}) \]
as $\phi$-twisted Hilbert $\bK(\acute{\sV})$-modules. Set
\[\cH:= X \times (\prescript{\phi}{} \ell^2(P)^{\oplus \infty} \hotimes \sV ). \]
Then this is a universal $(\phi,\ft)$-twisted $P$-equivariant Hilbert bundle.
Moreover, there is a canonical isomorphism
\[\prescript{\phi}{} \sH_{C_0(X) \hotimes \bK(\acute{\sV})}^P \cong \Gamma_0 (X, \cH \hotimes \acute{\sV}^* )\]
as Hilbert $C_0(X) \hotimes \bK(\acute{\sV})$-modules, where $\Gamma_0(-)$ denotes the set of continuous sections vanishing at infinity.

We write $\bB(\cH \hotimes \acute{\sV}^*)$ for the bundle whose fiber at $x \in X$ is the bounded adjointable operator algebra over the Hilbert $\bK(\acute{\sV})$-module $\cH_x \hotimes \acute{\sV}^*$. Then, by the above discussion, there is a canonical identification 
\[ \bB( \prescript{\phi}{} \sH_{C_0(X) \hotimes \bK(\acute{\sV})}^P ) \cong \Gamma_b (X, \bB(\cH \hotimes \acute{\sV}^*)),\]
where $\Gamma_b(-)$ denotes the set of uniformly bounded continuous sections. Moreover, the map 
\[ \Gamma_b (X, \bB(\cH)) \to \Gamma_b (X, \bB(\cH \hotimes \acute{\sV}^*)), \ \ T \mapsto T \hotimes 1 \]
is a $P$-equivariant bundle isomorphism. 
Therefore, we obtain a $P$-equivariant isomorphism $\Gamma_b (X, \bB(\cH)) \to \bB(\prescript{\phi}{} \sH_{C_0(X) \hotimes \bK(\acute{\sV})}^P)$, which restricts to a homotopy equivalence $\Gamma_c (X, \Fred (\cH))^P \to \Fred(\prescript{\phi}{} \sH_{C_0(X) \hotimes \bK(\acute{\sV})}^P )^P$. 
This finishes the proof for the $p=q=0$ case by taking the $\pi_0$ of both sides. The general Clifford-equivariant version is proved in the same way. 
\end{proof}

\begin{rmk}\label{rmk:Cl}
In general, for a real $P$-equivariant vector bundle $E$, the Clifford bundle $\Cl(E)$ is $P$-equivariantly Morita equivalent to $\Cl(E \oplus \bR^{8n}) \cong \Cl(E \oplus \bR^{8n-d}) \hotimes \Cl_{d,0}$. The Clifford algebra $\Cl(E \oplus \bR^{8n-d})$ is $P$-equivariantly isomorphic to $\bK(S)$, where $S$ is the irreducible representation of $\Cl(E \oplus \bR^{8n-d})$. The group $P$ acts on $S$ as a $(\phi,\fv)$-twisted representation. Note that $\acute{\fv}=-\fv$ holds since $2\nu=0$ and $2\mu=0$ in the $\bZ_2$-coefficient cohomology groups. 
\end{rmk}

\begin{lem}\label{lem:movemap}
Under the isomorphisms given in Lemma \ref{lem:move},  the following hold:
\begin{enumerate}
\item For a continuous map $f \colon X \to Y$, the pull-back $f^* \colon \prescript{\phi}{} \K^{*,\ft}_P (Y) \to \prescript{\phi}{} \K_P^{*,\ft}(X)$ corresponds to the Kasparov product with 
\[ [f^*]:=[C(X), f^*, 0] \in \prescript{\phi}{} \KK^{G,\Pi}(C(Y), C(X)).\]
\item For $[F_i] \in \prescript{\phi}{} \K_P^{*,\ft_i}(X_i) $ ($i=1,2$), the external product $[F_1] \otimes  [F_2] \in \prescript{\phi}{} \K_P^{*,\ft_1+\ft_2}(X_1 \times X_2)$ corresponds to the Kasparov product $[F_1] \hotimes  [F_2]$.
\item The Bott element $\beta_E \in \prescript{\phi}{} \K^{r, \fv}_P(E)$ corresponds to the KK-element 
\begin{align*}
    [C_0(E, q^* S) , \fc , C] \in \prescript{\phi}{} \KK^{P} (\Cl_{8n-r,0}  , C_0(E, q^*\Cl(E'))), \label{eq:BottKK}
\end{align*}
where $C$ is the operator defined in \eqref{eq:C}.
\end{enumerate}
\end{lem}
\begin{proof}
The claims (1) and (3) are obvious from the definition. To see (2), notice that the first definition of the external product $[F_1] \otimes [F_2]$ in Lemma \ref{lem:prod} (1) satisfies the conditions of the Kasparov product given in Definition \ref{defn:Kasparov}. 
\end{proof}

Let $\cK$ denote the bundle $\hat{\Pi} \times \tilde{\sK}$ with the $G$-action as in Lemma \ref{lem:bdlK}. This bundle is viewed in two ways:
\begin{itemize}
\item Let us choose a section $s \colon P \to G$. Then $(\cK, u_{s(p)})$ is a $(\phi,\sigma)$-twisted $P$-equivariant vector bundle on $\hat{\Pi}$. 
\item The space $C(\hat{\Pi}, \cK)$ of continuous sections is a $(G,\Pi)$-equivariant Hilbert $C(\hat{\Pi})_\sigma$-module.
\end{itemize}

\begin{lem}
The $\phi$-twisted $(G,\Pi)$-C*-algebras $C(\hat{\Pi},\bK(\cK))$ and $C(\hat{\Pi})_\sigma$ are equivariantly Morita equivalent via the imprimitivity bimodule $C(\hat{\Pi},\cK)$ (cf.\ Example \ref{exmp:Morita}). In particular, 
\[ [\cK]:=[C(\hat{\Pi},\cK), \id, 0] \in \prescript{\phi}{} \KK^{G,\Pi}(C(\hat{\Pi},\bK(\cK)),C(\hat{\Pi})_\sigma) \]
is a $\prescript{\phi}{} \KK^{G,\Pi}$-equivalence. 
\end{lem}

\subsection{Twisted crystallographic T-duality via the Kasparov product}
We apply the above discussion to the twisted equivariant K-groups of our interest. Before that, we summarize our notations on Dirac operators. 
\begin{notn}\label{notn:Dirac}
let $\hat{S}$ denote the spinor representation of the Clifford algebra $\Cl(V \oplus -V)$, where $-V$ denotes the linear space $V$ with the negative definite inner product. We have $\ast$-homomorphisms $\fc \colon \Cl(V) \to \bB(\hat{S})$ and $\fh \colon \Cl(-V) \to \bB(\hat{S})$. We write $m \colon C(V/\Pi ) \to \bB(L^2(V/\Pi, \hat{S}))$ for the multiplication $\ast$-representation. Let us define the Dirac operator $D:= \sum \fh(v_i) \partial_{v_i}$ on $V/\Pi$, where $v_i$ runs over an orthonormal basis of $T(V/\Pi)$. Set $F:=D(1+D^2)^{-1/2}$. 

Similarly, we write $D_E$ and $F_E$ for the Dirac operator twisted by a vector bundle $E$. Also, we write $\tilde{D}:=\sum_{v_i} \fh(v_i) \partial_{v_i}$ for the Dirac operator on $V$ and set $\tilde{F}:=\tilde{D}(1+\tilde{D}^2)^{-1/2}$. 
\end{notn}

\begin{lem}\label{lem:BottKK}
The Bott map $\beta_V \otimes \blank \colon \prescript{\phi}{} \K_P^{*,\ft-\fv}(\pt) \to \prescript{\phi}{} \K_P^{*,\ft}(V)$ is an isomorphism. 
\end{lem}
\begin{proof}
This follows from Lemma \ref{lem:movemap} (3). Indeed, it is proved in the same way as \cite{kasparovOperatorFunctorExtensions1980} that the element $\beta_V$ has the $\prescript {\phi}{} \KK^P$-inverse
\[\alpha_V:=[L^2(V,\hat{S}),m,\tilde{F}] \in \prescript{\phi}{} \KK^P (C_0(V), \bC ). \]
(Indeed, the elements $\alpha_V$ and $\beta_V$ are mutually inverse in Real KK-theory. Compare this fact with Remark \ref{rmk:RealKK}.)
\end{proof}

Now we apply Lemma \ref{lem:move}, Remark \ref{rmk:Cl} and Lemma \ref{lem:bdlK} for the spaces of our interest. Let $(\phi , \ft)$ be a twist on the group $P$ and let $\acute{\sV}$ be a $(\phi,\acute{\ft})$-twisted representation of $P$ as in Lemma \ref{lem:move}.
\begin{prp}\label{lem:FMtoKK}
There are canonical isomorphisms
\begin{align*}
\begin{split}
\prescript{\phi}{} \K_P^{n+d,\ft -\fv}(V/\Pi) &\cong \prescript{\phi}{} \KK^P(\Cl_{0,n}, C(V/\Pi) \hotimes \bK(\acute{\sV}) \hotimes \Cl(V)), \\
\prescript{\phi}{} \K_P^{n+d,\ft -\fv}(V/\Pi \times \hat{\Pi}) &\cong \prescript{\phi}{} \KK^P(\Cl_{0,n}, C(V/\Pi  \times \hat{\Pi}) \hotimes \bK(\acute{\sV}) \hotimes \Cl(V)), \\
\prescript{\phi}{} \K_P^{n+d,\ft -\fv + \sigma}(V/\Pi \times \hat{\Pi}) &\cong \prescript{\phi}{} \KK^{G,\Pi}(\Cl_{0,n}, C(V/\Pi  \times \hat{\Pi})_\sigma \hotimes \bK(\acute{\sV}) \hotimes \Cl(V)), \\
\prescript{\phi}{} \K_P^{n,\ft +\sigma }(\hat{\Pi}) &\cong \prescript{\phi}{} \KK^{G,\Pi }(\Cl_{0,n}, C(\hat{\Pi})_\sigma \hotimes \bK(\acute{\sV})). 
\end{split}
\end{align*}
\end{prp}

We describe the twisted crystallographic T-duality map as a Kasparov product.
\begin{lem}
Through the isomorphisms in Proposition \ref{lem:FMtoKK}, the pull-back $\hat{\pi}^* \colon \prescript{\phi}{} \K_P^{*,\ft-\fv} (V/\Pi) \to \prescript{\phi}{} \K_P^{*,\ft-\fv} (V/\Pi \times \hat{\Pi})$ corresponds to the Kasparov product with $[\hat{\pi}^*] \hotimes \id_{C(V/\Pi)}$, where
\[[\hat{\pi}^*] := [C(\hat{\Pi}), \pi^*, 0] \in \prescript{\phi}{} \KK^P( \bC,C(\hat{\Pi})). \]
\end{lem}
This is a special case of Lemma \ref{lem:movemap} (1). 

\begin{lem}
Through the isomorphisms in Proposition \ref{lem:FMtoKK}, the internal product $[\cP] \hotimes _{V/\Pi \times \hat{\Pi}} \blank \colon \prescript{\phi}{} \K_P^{*,\ft-\fv} (V/\Pi \times \hat{\Pi}) \to \prescript{\phi}{} \K_P^{*,\ft-\fv +\sigma} (V/\Pi \times \hat{\Pi})$ corresponds to the Kasparov product with 
\[ [\! [\cP] \!] :=[C(V/\Pi \times \hat{\Pi},\cP), m, 0] \in \prescript{\phi}{} \KK^{G,\Pi}(C(V/\Pi \times \hat{\Pi}) , C(V/\Pi \times \hat{\Pi})_\sigma). \] 
\end{lem}
\begin{proof}
This follows from Lemma \ref{lem:movemap} (1), (3) and the fact that the cup product $[\cP] \otimes _{V/\Pi \times \hat{\Pi}} x $ coincides with $\Delta^*([\cP] \otimes x)$, where $\Delta \colon V/\Pi \times \hat{\Pi} \to (V/\Pi \times \hat{\Pi})^2$ denotes the diagonal embedding. Indeed, $\Delta^*([\cP] \otimes x)$ corresponds to
\[ x \hotimes_{C(V/\Pi) \hotimes C(\hat{\Pi})} ([\cP] \hotimes_{C(V/\Pi) \hotimes C(\hat{\Pi})} [\Delta_\sigma^*])=x \hotimes _{C(V/\Pi) \hotimes C(\hat{\Pi})} [\![ \cP ]\!], \]
where $\Delta_\sigma^* \colon C(V/\Pi \times \hat{\Pi}) \otimes C(V/\Pi \times \hat{\Pi})_\sigma \to C(V/\Pi \times \hat{\Pi})_\sigma$ is the pull-back by the diagonal embedding.  
\end{proof}

\begin{lem}
Through the isomorphisms in Lemma \ref{lem:FMtoKK}, the push-forward $\pi_! \colon \prescript{\phi}{} \K_P^{*,\ft-\fv +\sigma}(V/\Pi \times \hat{\Pi}) \to \prescript{\phi}{} \K_P^{*,\ft + \sigma }(\hat{\Pi})$ corresponds to the Kasparov product with the element $[\pi_!] \otimes \id_{C(\Pi)_\sigma }$, where 
 \[ [\pi_!] :=[L^2(V/\Pi, \hat{S}), m \otimes \fc , F] \in \prescript{\phi}{} \KK^P (C(V/\Pi) \otimes \Cl(V) , \bC ). \]
\end{lem}
\begin{proof}
Let us choose a $P$-equivariant embedding $V/\Pi \to W$. Let $E$ denote the normal bundle and let $\iota \colon E \to W$ denote the embedding. 
Then the normal bundle of $V/\Pi \times \hat{\Pi} \subset W \times \hat{\Pi}$ is $E \times \hat{\Pi}$ with the open embedding $\iota \times \id_{\hat{\Pi}}$.

Now, by Definition \ref{defn:push}, (\ref{eq:Thom}) and Lemma \ref{lem:movemap} (2), we have
\begin{itemize}
    \item $\Thom_{E \times \hat{\Pi}} (x) =  x \hotimes_{C(V/\Pi \times \hat{\Pi})_\sigma } \beta_{E \times \hat{\Pi}} =   (x \hotimes_{C(V/\Pi)} \beta_{E}) \hotimes \id_{C(\hat{\Pi})_\sigma} $,
    \item $(\iota \times \id_{\hat{\Pi}})_*(y) =y \hotimes _{C_0(E)}([\iota_*] \hotimes \id_{C(\hat{\Pi})_\sigma})$,
    \item $\Thom_{\hat{\Pi} \times W} ^{-1}(z)=z \hotimes _{C_0(W)} \alpha_W$, 
\end{itemize}
for any $x \in \prescript{\phi}{} \K_P^{*,\ft-\fv}(V/\Pi \times \hat{\Pi})$, $y \in \prescript{\phi}{} \K_P^{*,\ft}(E \times \hat{\Pi})$ and $z \in \prescript{\phi}{} \K_P^{*,\ft}(W \times \hat{\Pi})$ under the identification as in Proposition \ref{lem:FMtoKK}. Here $\Delta_E$ denotes the proper map $(\id_{E}, q ) \colon E \to E \times V/\Pi $.

Hence the map $\pi_!$ corresponds to the Kasparov products over $C(V/\Pi)$ with the KK-element
\[ (\beta_E \hotimes_{C_0(E)} [\Delta_E^*]) \hotimes_{C_0(E)} [\iota_*] \hotimes_{C_0(W) \hotimes \Cl(V)} \alpha_W \]
of $\prescript{\phi}{} \KK^{G,\Pi}(C(V/\Pi ) \hotimes  \Cl(V) , \bC )$. 
Note that the Bott element $\beta_E$ is identified with 
\[ [C_0(E, S_W), \fc , C] \in  \prescript{\phi}{}\KK^P(\Cl(V), C_0(E)), \]
where $S_W$ denotes the spinor bundle on $W$, through the isomorphism
\[ \prescript{\phi}{} \KK^{P}( \Cl_{8n-r,0}, C_0(E, q^*\Cl(E'))) \cong \prescript{\phi}{} \KK^P(\Cl(V), C_0(E)) \]
given by $\Cl(E') \hotimes  \Cl(V) \cong \Cl(W) \hotimes \Cl_{8n-r,0}$. 
Now it is proved that this element coincides with $[\pi_!]$ in the same way as the push-forward map in non-equivariant K-theory, given in \cite{connesLongitudinalIndexTheorem1984}*{Proposition 2.9} (which is essentially due to \cite{connesLongitudinalIndexTheorem1984}*{Lemma 2.4}). 
\end{proof}

In summary, the crystallographic T-duality map $\prescript{\phi}{}{\mathrm{T}}_{G}^{*,\ft}$ is given by the Kasparov product with the element
\[ [\hat{\pi}^*] \otimes_{C(\hat{\Pi})} [\! [\cP]\! ] \otimes_{C(V/\Pi)} [\pi_!] \in   \prescript{\phi}{} \KK^{G,\Pi}(C(V/\Pi) \otimes \Cl(V), C(\hat{\Pi})_\sigma ). \]

Now we give an explicit representative of this $\KK$-element. Let $\fH_\cP$ denote the Hilbert bundle on $\hat{\Pi}$ obtained by the fiberwise $L^2$-completion of $C(\hat{\Pi} \times V/\Pi, \hat{S} \otimes \cP)$ and let $\fD_\cP$ denote the fiberwise Dirac type operator twisted by $\cP$, acting on the Hilbert $C(\hat{\Pi})$-module $C(\hat{\Pi}, \fH_\cP)$ as an unbounded regular odd self-adjoint operator (the regularity is proved in \cite{hankeCodimensionTwoIndex2015}*{Theorem 1.5}). Set $\fF_\cP:=\fD_\cP (1+\fD_\cP ^2)^{-1/2}$. We write $m_{V/\Pi}$ for the multiplication $\ast$-representation of $C(V/\Pi)$ onto $C(\hat{\Pi}, \fH_\cP )$.
\begin{lem}\label{lem:Tdual}
We have
\begin{align} 
[\hat{\pi}^*] \otimes_{C(\hat{\Pi})} [\! [\cP]\! ] \otimes_{C(V/\Pi)} [\pi_!]
= [C(\hat{\Pi}, \fH_\cP ), m \otimes \fc , \fF_\cP  ].\label{eq:TdualKK}
\end{align}
\end{lem}
\begin{proof}
The Kasparov product $[\hat{\pi}^*] \otimes [\! [\cP ]\!] $ is represented by the Kasparov bimodule
\[ [C(V/\Pi \times \hat{\Pi},\cP) , m_{V/\Pi} , 0 ] \in \prescript{\phi}{} \KK^{G,\Pi}(C(V/\Pi), C(V/\Pi \times \hat{\Pi})_\sigma). \]
Hence the only thing we have to see is that the operator $\fF_\cP $ is a $F$-connection on 
\[C(V/\Pi \times \hat{\Pi} , \cP) \otimes_{C(V/\Pi)} L^2(V/\Pi, S) \cong C(\hat{\Pi}, \fH_\cP ). \]
It holds true because the principal symbol $\sigma (\fF_\cP)$ coincides with $\sigma(F) \otimes \id_{\cP}$. 
\end{proof}

\subsection{T-duality and Dirac morphism}
The proof of Theorem \ref{thm:crystal} is now reduced to the invertibility of the twisted equivariant KK-morphism (\ref{eq:TdualKK}). In this subsection we show this by relating (\ref{eq:Tdual}), equivalently \eqref{eq:TdualKK}, with the Dirac element (in the sense of \cite{kasparovEquivariantKKTheory1988}), which is known to be invertible.

\begin{defn}
Let $\mathsf{D} \in \prescript{\phi}{} \KK^G(C_0(V) \hotimes \Cl (V), \bC )$ be the Dirac element 
\[ \mathsf{D}:= [L^2(V, \hat{S} ) , m \otimes \fc, \tilde{F}], \]
where $\hat{S}$ and $\tilde{F}$ are as in Notation \ref{notn:Dirac}.
\end{defn}
It is proved by Higson--Kasparov \cite{higsonTheoryKKTheory2001} that this element is a $\prescript{\phi}{} \KK^G$-equivalence (see also Remark \ref{rmk:RealKK}). 
By the functoriality of the partial descent homomorphism, the KK-element
\[ \prescript{\phi}{} j_{G,\Pi}(\mathsf{D}) \in \prescript{\phi}{} \KK^{G,\Pi} ((C_0(V) \hotimes \Cl(V)) \rtimes \Pi , C^*(\Pi)), \]
represented by the Kasparov bimodule $[L^2(V) \rtimes \Pi, (m \hotimes \fc) \rtimes \Pi ,  F_\Pi ]$ as in Definition \ref{defn:pdesc}, is also a $\prescript{\phi}{} \KK^{G,\Pi}$-equivalence. 

The group C*-algebra $C^*\Pi$ with the twisted $(G,\Pi)$-action as in Remark \ref{rmk:cross} is isomorphic to $C(\hat{\Pi})_\sigma$. 
Also, the crossed product C*-algebra $C_0(V) \rtimes \Pi$ is $\prescript{\phi}{} \KK^{G,\Pi}$-equivalent to $C(V/\Pi)$. Indeed, let $\cX$ denote the bundle of Hilbert spaces $V \times _\Pi \ell^2(\Pi)$ on $V/\Pi$. The $G$-action on $V$ induces that on $\cX$ as $g \cdot [v,\xi]:=[g\cdot v,\xi]$ for $v \in V$ and $\xi \in \ell^2(\Pi)$. 
The continuous section space $C(V/\Pi,\cX)$ is $G$-equivariantly isomorphic to the completion of $C_c(V)$ by the $C(V/\Pi)$-valued inner product
\[\langle \xi, \eta \rangle(v + \Pi):=\sum_{t \in \Pi}\overline{\xi(v + t)}\eta(v+t). \]
This identification extends to the $G$-equivariant unitary 
\[ W \colon L^2(V) \to L^2(V/\Pi , \cX). \]

\begin{lem}\label{lem:Morita}
There is a $(G,\Pi)$-equivariant $\ast$-isomorphism 
\[ \varphi \colon C_0(V) \rtimes \Pi \cong \bK(C(V/\Pi, \cX)).\]
\end{lem} 
\begin{proof}
The compact operator algebra $\bK(C(V/\Pi,\cX))$ is isomorphic to the continuous section algebra $C(V/\Pi , V \times _\Pi \bK(\ell^2(\Pi)))$. 
By the Takesaki--Takai duality $\bK(\ell^2(\Pi)) \cong c_0(\Pi) \rtimes \Pi$, we have 
\begin{align*}
\bK(C(V/\Pi, \cX))& \cong  C(V/\Pi , V \times _\Pi \bK(\ell^2\Pi) )\\
&\cong C(V/\Pi , V \times _\Pi (c_0(\Pi) \rtimes \Pi) ) \\
& \cong C(V/\Pi , V \times _\Pi c_0(\Pi)) \rtimes \Pi = C_0(V) \rtimes \Pi. 
\end{align*}
This gives a $\ast$-representation $\varphi \colon C_0(V) \rtimes \Pi \to \bK(C(V/\Pi, \cX))$, which is explicitly written as
\[ \varphi ((\sum f_t u_t)\xi) (v)=\sum f_t(v)\xi(v-t)  \]
for $\xi$ in the dense subalgebra $C_c(V) \subset C(V/\Pi , \cX)$. Hence it is $G$-equivariant and $u_t(\xi)=\varphi(u_t)\xi$, that is, $\varphi$ is $(G,\Pi)$-equivariant.
\end{proof}
Lemma \ref{lem:Morita} means that $C_0(V) \rtimes \Pi$ is $\phi$-twisted $(G,\Pi)$-equivariantly Morita equivalent with $C(V/\Pi)$ via the imprimitivity bimodule $\cX$ (cf.\ Example \ref{exmp:Morita}). In particular, the element 
\[ [\cX]:=[C(V/\Pi,\cX),\varphi ,0] \in \prescript{\phi}{} \KK^{G,\Pi }(C_0(V) \rtimes \Pi , C(V/\Pi))\]
is a $\prescript{\phi}{} \KK^{G,\Pi}$-equivalence with the $\KK$-inverse represented by its adjoint $\cX^*$.

\begin{lem}\label{lem:Dirac}
The $\prescript{\phi}{} \KK^{G,\Pi}$-equivalence 
\[ [\cX^*] \otimes_{C_0(V) \rtimes \Pi} \prescript{\phi}{} j_{G,\Pi} (\mathsf{D}) \in \prescript{\phi}{} \KK^{G,\Pi}(C(V/\Pi) \otimes \Cl(V), C(\hat{\Pi})_\sigma )\]
is represented by the Kasparov bimodule $[C(\hat{\Pi}, \fH_\cP ), m \otimes \fc , \fF_\cP]$.
\end{lem}
\begin{proof}
We identify the $G$-equivariant Hilbert $C(\hat{\Pi})_\sigma$-module $L^2(V) \rtimes \Pi$ with $L^2(V/\Pi , \cX ) \otimes C(\hat{\Pi})_\sigma$ by the unitary
\[ U \colon L^2(V) \rtimes \Pi \ni \sum_{t \in \Pi} \xi_t u_t \mapsto \sum_{t \in \Pi} W\xi_t \otimes \chi_t \in L^2(V/\Pi,\cH) \otimes C(\hat{\Pi
})_\sigma,\]
where $\chi_t \in C(\hat{\Pi})_\sigma$ is the function $\chi_t(\xi):=\xi(t)$. Then we have
\begin{align*}
U(m\rtimes \Pi )\Big( \sum_{h \in \Pi} f_hu_h \Big) U^* \Big( \sum_{g \in \Pi} \xi_g \otimes \chi_g \Big) &= \sum_{g,h \in \Pi} \varphi (f_hu_h)(\xi_g) \otimes \chi_{hg},\\
U (F_\Pi) U^* \Big( \sum_{g \in \Pi} \xi_g \otimes \chi_g \Big) &=\sum_{g \in \Pi} WFW^*\xi_g \otimes \chi_g.
\end{align*}
In other words, $\tilde{\varphi}:=\Ad(U) \circ (m \rtimes \Pi)$ satisfies
\[ 
 \tilde{\varphi} \Big( \sum_{g \in \Pi} f_gu_g \Big) (\alpha)  = \sum_{g \in \Pi} \alpha(g)\varphi (f_gu_g)\]
for any $\alpha \in \hat{\Pi}$, and $UF_\Pi U^*=WFW^* \otimes 1$. The operator $WFW^*$ coincides with $F_\cX=D_\cX(1+D_\cX^2)^{-1/2}$, where $D_\cX$ is the Dirac operator on $V/\Pi$ twisted by $\cX$ with respect to its flat connection. 

Let $\cQ:= \cX^* \otimes_{\tilde{\varphi}} \cX $, let $D_\cQ$ be the Dirac operator twisted by $\cQ$ and let $F_\cQ:=D_\cQ (1+D^2_\cQ)^{-1/2}$. Then the above discussion shows that
\begin{align*}
&[\cX]^* \otimes  _{C_0(V) \rtimes \Pi} \prescript{\phi}{} j_{G,\Pi} (\sfD) \\
=&[\cX^*] \otimes _{C_0(V) \rtimes \Pi} [L^2(V/\Pi, \cX \otimes \hat{S}_V) \otimes C(\hat{\Pi})_\sigma , \tilde{\varphi}, F_\cX \otimes  1_{C(\hat{\Pi})_\sigma } ] \\
=&[C(\hat{\Pi} , \fL^2(V/\Pi, \cQ \otimes \hat{S})) , m_{V/\Pi} , F_{\cQ}].
\end{align*}

The remaining task is to show that $\cQ$ is $G$-equivariantly isomorphic to $\cP$. For $x=v + \Pi \in V/\Pi$ and $\alpha \in \hat{\Pi}$, the bundle map
\[ \cX^* \otimes_{\tilde{\varphi}} \cX \ni  [x,\alpha, s^* \otimes t] \mapsto \Big[ v, \alpha, \sum _{g \in \Pi}\alpha(g)\overline{s(v+g)}t(v+g) \Big] \in (V \times \hat{\Pi} \times \bC)/\Pi \]
is well-defined independent from the choice of $v$. It is everywhere non-zero and hence is an isomorphism.
\end{proof}

\begin{proof}[Proof of Theorem \ref{thm:crystal}]
Now the theorem follows from Lemma \ref{lem:Tdual} and Lemma \ref{lem:Dirac}.
\end{proof}

\subsection{Generalization to irrational twists}\label{section:5.4}
The description using the Dirac morphism given in the above subsection extends the twisted crystallographic T-duality to the `irrational flux' case, i.e., the case that the twist $(\phi,c,\tau)$ of $G$ is not necessarily obtained as Proposition \ref{prp:twcry}, or in other words $\tau$ is not necessarily trivial on $\Pi$. The right side of the T-duality will then involve a noncommutative torus rather than $\hat{\Pi}$.
Such a situation may arise when $\Pi$ is projectively represented as \emph{magnetic} translation operators, as happens when an external magnetic field is applied perpendicular to $V$. 

Let $G$ be a discrete group acting properly and cocompactly on the affine space $V$ and let $(\phi,c,\tau)$ be an arbitrary twist of $G$. Let $\Pi$ be a free abelian finite index subgroup of $G$ such that $\phi|_\Pi$ and $c|_\Pi$ are trivial (such $\Pi$ exists by Lemma \ref{lem:groupext}). We write $G_\tau$ and $\Pi_\tau$ for the $\bT$-extension of $G$ and $\Pi$ corresponding to $\tau$ respectively. Let $(\pi, \sH)$ be a $(\phi,c,\tau)$-twisted unitary representation of $G$.
The Kasparov product with 
\[ \prescript{\phi}{}j_{G,\Pi} (\sfD \otimes \id_{\bK(\sH)}) \in \prescript{\phi}{}{\KK}^{G,\Pi}( (C_0(V) \otimes \bK(\sH)) \rtimes \Pi , \bK(\sH) \rtimes \Pi )  \]
induces an isomorphism 
\begin{align}
    \prescript{\phi}{}{\K}^{G,\Pi}((C_0(V) \otimes \bK(\sH)) \rtimes \Pi) \cong \prescript{\phi}{}{\K}^{G,\Pi}(\bK(\sH) \rtimes \Pi). \label{eq:irrBC}
\end{align} 

The equivariant K-groups on the left and right hand sides of the above isomorphism are simplified.
Firstly, let $\cA$ denote the $\phi$-linear $P$-equivariant bundle $V \times _{\Pi , \Ad(\pi)} \bK(\sH)$ of $\bZ_2$-graded compact operator algebras over $V/\Pi$. By the construction, the equivariant Dixmier--Douady class is
\[ \mathrm{DD}(\cA) = f^*\ft \in H^1_P(V/\Pi; \bZ_2) \oplus H^2_P(V/\Pi; \prescript{\phi}{}{\bT}), \]
where $f \colon V \to \pt$, through the canonical identification $H^1_P(V/\Pi; \bZ_2) \oplus H^2_P(V/\Pi; \prescript{\phi}{}{\bT})  \cong H^1_G(V; \bZ_2) \oplus H^2_G(V; \prescript{\phi}{}{\bT})$.
Now it is checked in the same way as Lemma \ref{lem:Morita} that there is a $\phi$-twisted $(G,\Pi)$-equivariant Morita equivalence
\[ (C_0(V) \otimes \bK(\sH)) \rtimes \Pi  \sim _{\mathrm{Morita }} C(V/\Pi, \cA).  \]

Secondly, we consider the non-commutative torus $C^*_{\bar{\tau}} \Pi$ (where $\bar{\tau}$ denotes the inverse of $\tau$), i.e., the C*-algebra generated by unitaries $\{ u_t \}_{t \in \Pi}$ with the relation $u_su_t = \bar{\tau}(s,t)u_{st}$. 
Let $\alpha$ denote the $\phi$-twisted $G$-action on $C^*_{\bar{\tau}} \Pi$ determined as the $\phi$-linear extension of 
\begin{align} 
\alpha_g(u_t) = \bar{\tau}(g,t)\bar{\tau}(gt,g^{-1})u_{gtg^{-1}} \label{eq:irr}
\end{align}
for any $t\in \Pi $ and $g \in G_\tau $ (note that this $\alpha$ satisfies $\alpha_g \circ \alpha_h=\alpha_{gh}$). 
Also, set $ \sigma _t:= u_t$. Then $\alpha_t = \Ad \sigma_t$  and $\sigma_s\sigma_t = \bar{\tau}(s,t)\sigma_{st}$ hold for any $t \in \Pi$. 
Hence the pair $(\alpha,\sigma)$ obviously gives rise to a $(G_\tau, \Pi_\tau )$-action onto $C^*_{\bar{\tau}} \Pi$ such that $\sigma_z =\bar{z}$ for any $z \in \bT \subset \Pi_\tau$. 
\begin{lem}
The C*-algebras $\bK(\sH) \rtimes \Pi$ and $C^*_{\bar{\tau}} \Pi$ are $(G_\tau, \Pi_\tau)$-equivariantly Morita equivalent (in the sense of Example \ref{exmp:Morita}). 
\end{lem}
\begin{proof}
Before starting the proof we remark that the group C*-algebra $C^*\Pi_\tau$ of the central extension group $\Pi_\tau$ is decomposed as the direct sum $\bigoplus_{n \in \bZ} C^*_{\tau^n}\Pi$ since the subalgebra $C^*\bT \cong c_0(\bZ)$ of $C^*\Pi_\tau $ lies in the center.
This decomposition respects the $\phi$-twisted $(G_\tau,\Pi_\tau)$-C*-algebra structure of $C^*\Pi_\tau$. 

As is mentioned in Example \ref{exmp:Moritadescent}, the $\phi$-twisted $G_\tau$-equivariant $\bK(\sH)$-$\bC$ imprimitivity bimodule $\sH$ induces a $\phi$-twisted $(G_\tau, \Pi_\tau)$-equivariant $\bK(\sH) \rtimes \Pi_\tau$-$C^*\Pi_\tau$ imprimitivity bimodule $\sH \rtimes \Pi_\tau$. 
Let $\varphi \colon C^*\Pi_\tau \to C^*_{\bar{\tau}} \Pi$ denote the quotient onto the direct summand. 
Then 
\[\sH \rtimes_{\bar{\tau}} \Pi := (\sH \rtimes \Pi_\tau )\otimes_\varphi C^*_{\bar{\tau}}\Pi \]
is a $G_\tau$-equivariant Hilbert $\bK(\sH) \rtimes \Pi_\tau$-$C^*_{\bar{\tau}}\Pi$ bimodule. Moreover, the left action of $\bK(\sH) \rtimes \Pi_\tau$ descends to the quotient $\bK(\sH) \rtimes \Pi$ since the central subgroup $\bT \subset \Pi_\tau$ acts on $\sH \rtimes _\tau \Pi$ trivially from the left. That is, $\sH \rtimes_\tau \Pi$ is a $\bK(\sH) \rtimes \Pi$-$C^*_{\bar{\tau}}\Pi$ imprimitivity bimodule. Since it is $(G_\tau, \Pi_\tau)$-equivariant by construction, this finishes the proof. 
\end{proof}

Since an equivariant Morita equivalence induces an isomorphism of the equivariant K-group (cf.\ Example \ref{exmp:Morita}), the above discussion and the isomorphism (\ref{eq:irrBC}) conclude the following theorem. 
\begin{thm}
Let $(\phi,c,\tau)$ be a twist of a discrete group $G$ acting properly and cocompactly on $V$. Let $\sW$ be a $(\phi,c,\epsilon(c,c))$-twisted representation of $P$. There is an isomorphism
\[ \prescript{\phi}{}\K^{*,f^*\ft-\fv}_P(V/\Pi) \xrightarrow{\cong} \prescript{\phi}{}{\KK}_{*-d}^{G_\tau ,\Pi_\tau }(\bC, C^*_{\bar{\tau}} \Pi \hotimes \bK(\sW)). \]
\end{thm}
In short, hereafter we write the right hand side of the above isomorphism as $\prescript{\phi}{}{\KK}_{*-d,c}^{G_\tau ,\Pi_\tau }(\bC, C^*_{\bar{\tau}} \Pi)$ (this is consistent with the notation of \cite{kubotaNotesTwistedEquivariant2016}*{Section 2}).

Even in the irrational case, the K-group lying in the right hand side of the above isomorphism is viewed as the set of topological phases with the symmetry type $(G,\phi,c,\tau)$. Let us define a Hilbert space $\prescript{\phi}{}\ell^2(G, \tau)$ as the completion of
\[\prescript{\phi}{}{C}_c(G, \tau):=\{ \xi \in C_c(G_\tau) \mid \xi(z g) = z^{-1} \xi(g) \text{ for any $z\in \bT$}  \} \]
equipped with the $L^2$-inner product on $G_\tau$ with respect to the Haar measure and the $\bC$-vector space structure given by the complex multiplication $(\lambda \cdot \xi)(g)=\prescript{\phi(g)}{}{\lambda} \cdot \xi(g)$. 
By choosing a section $s \colon G \to G_{\tau }$, the left regular representation $(\pi_l(h)\xi)(g):=\xi(s(h)^{-1}g)$ is regarded as a $(\phi , \tau)$-twisted representation of $G$. Similarly, the right regular representation $(\pi_r(h)\xi)(g):= \xi(gs(h))$ is a $(\phi, \bar{\tau} )$-twisted representation of $G$. 

A choice of sections $P \to G$ identifies the underlying complex Hilbert space $\prescript{\phi}{}{\ell}^2(G, \tau)$ with the direct sum of subspaces $\ell^2( \Pi g, \tau)$. Then $\pi_l(\Pi)$ and $\pi_r (\Pi)$ acts by the regular representations on each component. 
The set of operators on $\prescript{\phi}{}{\ell}^2(G, \tau  )$ satisfying assumption (1) of Lemma \ref{lem:TP} and commuting with $\pi_l(\Pi)$ forms a $\ast$-algebra. Let $A_\tau$ denote its closure. It is generated by $\pi_r(\Pi)$ and $\bK(\prescript{\phi}{}{\ell}^2(P))$, that is, 
\[A_\tau \cong  C^*_{\bar{\tau} }\Pi \otimes \bK(\prescript{\phi}{}{\ell}^2(P))\]
as C*-algebras. 
The $\phi$-twisted $G_\tau $-action $\Ad (\pi_l)$ on $\bB(\prescript{\phi}{}{\ell^2}(G,\tau))$ descends to a $\phi$-twisted $P$-action onto $A_\tau$. 

In the same way as Subsection \ref{section:3.3}, let us define $\cH_N(G,\phi,\ft)$ as the set of bounded operators on $(\prescript{\phi}{}\ell^2(G,\tau ) \hotimes \acute{\sV})^{\oplus N}$ satisfying the conditions (1), (2') and (3) enumerated below Lemma \ref{lem:TP}, and define the set of topological phases as $\mathscr{TP}(G,\phi,\ft):=\big( \bigcup_N \cH_N(G,\phi,\ft)\big) /\sim $.
This set is the same thing as the Karoubi picture of the twisted equivariant K-group $\K_{0,c}^P(A_\tau)$, which is defined and proved to be isomorphic to $\KK_{0,c}^P(\bC, A_\tau)$ in \cite{kubotaNotesTwistedEquivariant2016}*{Theorem 5.14}. 

\begin{prp}\label{prp:irr}
The $(G_\tau,\Pi_\tau)$-C*-algebra $C^*_{\bar{\tau}} \Pi$ and the $P$-C*-algebra $A_\tau$ are $(G_\tau,\Pi_\tau)$-equivariantly Morita equivalent.
\end{prp}
\begin{proof}
Following the idea of \cite{roeComparingAnalyticAssembly2002}*{Lemma 2.1}, we define the $\phi$-twisted $(G_\tau,\Pi_\tau)$-equivariant Hilbert $C^*_{\bar{\tau}}\Pi$-module $\prescript{\phi}{} \ell^2_\Pi(G,\tau )$ as the closure of $\prescript{\phi}{}C_c(G, \tau )$ with respect to the inner product
\[ \langle \xi , \eta \rangle_{C^*_{\bar{\tau}} \Pi} := \sum_{t \in \Pi} \langle  \pi_l(t^{-1}) \xi, \eta \rangle_{\ell^2} \cdot u_t, \]
where $\langle {\cdot },{ \cdot } \rangle _{\ell^2}$ is the inner product of $\prescript{\phi}{}\ell^2(G,\tau )$. It is a routine work to see that the above inner product $\langle \cdot, \cdot \rangle_{C^*_{\bar{\tau}}\Pi}$ is compatible with the right $\bC_{\bar{\tau}}[\Pi]$-module structure given by 
\[ \xi \cdot u_t:=\pi_l(t^{-1}) \xi.\] 
Moreover, the external tensor product $\prescript{\phi}{} \ell^2_\Pi(G, \tau) \otimes_{C^*_{\bar{\tau}} \Pi} \ell^2(\Pi, \tau)$ (where $C^*_{\bar{\tau}}\Pi$ acts on $\ell^2(\Pi, \tau)$ by the right regular representation) is canonically isomorphic to $\prescript{\phi}{}{\ell}^2(G,\tau)$. 

Now the same argument as \cite{roeComparingAnalyticAssembly2002}*{Lemma 2.3} shows that 
\[{\cdot }\otimes_{C^*_{\bar{\tau}} \Pi} 1 \colon \bB(\prescript{\phi}{} \ell^2_\Pi(G,\tau )) \to \bB(\prescript{\phi}{} \ell^2(G,\tau ))\]
maps the compact operator algebra $\bK(\prescript{\phi}{} \ell^2_\Pi(G,\tau))$ faithfully onto $A_\tau$. 
This means that $\prescript{\phi}{} \ell^2_\Pi(G,\tau)$ is an equivariant $A_\tau$-$C^*_{\bar{\tau}}\Pi$ imprimitivity bimodule.
\end{proof}

Together with the above discussion, we obtain the desired isomorphism
\[\mathscr{TP}(G,\phi,\ft) \cong \prescript{\phi}{}\KK^{P}_{0,c}(\bC, A_\tau  ) \cong \prescript{\phi }{}\KK^{G_\tau ,\Pi_\tau }_{0,c}(\bC, C^*_{\bar{\tau}} \Pi). \]

\subsection{Relation to the partial Baum--Connes assembly map}
Following \cite{chabertTwistedEquivariantKK2001}, we define a $\phi$-twisted generalization of the partial Baum--Connes assembly map. Throughout this section we use the following notation 
\[ \prescript{\phi}{} \KK^G_\ft(A,B):= \prescript{\phi}{} \KK^G(A , B \hotimes \bK(\acute{\sV}))\]
for any $\bZ_2$-graded $G$-C*-algebras $A$ and $B$, where $\acute{\sV}$ is a $(\phi,\acute{\ft})$-twisted unitary representation of $G$ (cf.\ Lemma \ref{lem:move}). This notation is consistent with the one given in \cite{kubotaNotesTwistedEquivariant2016} (because of \cite{kubotaNotesTwistedEquivariant2016}*{Proposition 3.10}). 
\begin{defn}
The partial assembly map with coefficients in a $\phi$-twisted $G$-C*-algebra $A$ is defined by the composition
\[ \prescript{\phi}{} \mu_G^\Pi:= [\cH] \otimes_{C_0(V) \rtimes \Pi} \prescript{\phi}{} j_{G,\Pi}(\blank)  \colon \prescript{\phi}{} \KK^G(C_0(V) , A) \to \prescript{\phi}{} \KK^{G,\Pi}(\bC , A \rtimes \Pi  ).\]
\end{defn}
In particular, when $A = \bK(\acute \sV)$ we get a homomorphism from $\prescript{\phi}{} \KK ^G_\ft(C_0(V), \bC)$ to $\prescript{\phi}{} \KK^{G,\Pi}_\ft (\bC , C(\hat{\Pi})_\sigma)$. Here we relate this map with the Dirac morphism, and hence the twisted  crystallographic T-duality map.

The domains of the Baum--Connes and T-duality maps are identified by the following two isomorphisms.
\begin{enumerate}
    \item A $G$-equivariant $C_0(V)$-$\bK(\acute{\sV})$ bimodule is canonically viewed as a $(G,\Pi)$-equivariant $C_0(V) \rtimes \Pi$-$\bK(\acute{\sV})$ bimodule.
    Hence there is a canonical isomorphism 
    \[ \cJ_\Pi \colon \prescript{\phi}{} \KK ^G_\ft(C_0(V), \bC ) \to \prescript{\phi}{} \KK^{G,\Pi}_\ft (C_0(V) \rtimes \Pi, \bC).\]
    By definition of the partial descent map, this $\cJ_\Pi$ coincides with the composition $\iota_0^* \circ \prescript{\phi}{} j_{G,\Pi}$, where $\iota_0 \colon \pt \to \hat{\Pi}$ denotes the map to the unit (i.e., the trivial character). 
    \item The Poincar\'e duality isomorphism 
    \[ \PD_{V/\Pi} \colon \prescript{\phi}{} \KK^P_\ft(\bC, C(V/\Pi) \otimes \Cl(V)) \cong \prescript{\phi}{} \KK^P_\ft(C(V/\Pi), \bC)\]
    is given by the Kasparov product with
    \[ [\Delta^*] \otimes_{C(V/\Pi)} [L^2(V/\Pi , S),m \hotimes \fc, F] \in \prescript{\phi}{} \KK^P (C(V/\Pi) \otimes C(V/\Pi) \otimes \Cl(V) ,\bC ), \]
    where $\Delta \colon V/\Pi \to V/\Pi \times V/\Pi$ is the diagonal embedding. 
    It is seen in the same way as \cite{kasparovOperatorFunctorExtensions1980}*{Theorem 7} that this Kasparov product is an isomorphism by constructing its inverse as a Kasparov product. Note that $[L^2(V/\Pi , S),m, F]$ coincides with $\prescript{\phi}{} j_{G,\Pi}(\sfD) \otimes [\iota_0^*]$.
\end{enumerate}

To summarize, the composition $\mu_G^\Pi \circ \cJ_\Pi \circ \PD_{V/\Pi}$ is a homomorphism from $\prescript{\phi}{} \KK^P_{\ft-\fv}(\bC, C(V/\Pi)) $ to $\prescript{\phi}{} \KK^{G,\Pi}_\ft(\bC, C(\hat{\Pi})_\sigma)$.  

\begin{thm}\label{thm:BC}
The map $\mu_G^\Pi \circ \cJ_\Pi \circ \PD_{V/\Pi}$ is the same as the crystallographic T-duality map $\prescript{\phi}{}{\mathrm{T}}_{G}^{*,\ft}$. 
\end{thm}

Let $p$ denote the projection $V \to V/\Pi$ and let $\Delta_V :=(\id_V, p) \colon V \to V \times V/\Pi$. This is a proper $G$-equivariant map (by regarding $V/\Pi$ as a $G$-space through the quotient $G \to P$) and hence induces 
\[ \Delta_V ^* \colon C_0(V ) \otimes C(V/\Pi) \to C_0(V). \]

\begin{lem}\label{lem:delta}
We have 
\[ [\cX^*] \otimes_{C_0(V) \rtimes \Pi } \prescript{\phi}{}j_{G,\Pi} ([\Delta_V ^*]) \otimes_{C_0(V) \rtimes \Pi } [\cX] =[\Delta ^*]. \]
\end{lem}
\begin{proof}
By definition we have $\prescript{\phi}{} j_{G,\Pi} ([\Delta_V^*]) = [\Delta_V^* \rtimes \Pi]$, where $\Delta_V^* \rtimes \Pi$ is the induced $\ast$-homomorphism $(C_0(V) \otimes C(V/\Pi)) \rtimes \Pi \to C_0(V) \rtimes \Pi$. This $\ast$-homomorphism is explicitly described as
\[ (\Delta_V^* \rtimes \Pi)(x \otimes f) = xf, \]
for any $x \in C_0(V) \rtimes \Pi$ and $f \in C(V/\Pi)$. 
Here the product $xf$ is taken in $\bB(C(V/\Pi , \cX ))$, where $C_0(V) \rtimes \Pi  \cong \bK(C(V/\Pi, \cX ) )$ acts on $C(V/\Pi, \cX)$ as in Lemma \ref{lem:Morita} and $C(V/\Pi)$ acts by multiplication. 
Therefore, the Kasparov product 
$[\Delta_V^* \rtimes \Pi] \otimes_{C_0(V) \rtimes \Pi} [\cX] $ is represented by 
$ C(V/\Pi, \cX ) $. This shows that 
\[[\Delta_V^* \rtimes \Pi] \otimes_{C_0(V) \rtimes \Pi}[\cX] = ([\cX] \otimes \id_{C(V/\Pi)}) \otimes_{C(V/\Pi)} [\Delta^*], \]
which finishes the proof.
\end{proof}

\begin{proof}[Proof of Theorem \ref{thm:BC}]
In the proof, we omit the KK-equivalence $[\cX]$ and identify $C_0(V) \rtimes \Pi$ with $C(V/\Pi)$ for simplicity of notation.

Let $\Phi \colon \prescript{\phi}{} \KK^{G,\Pi}_\ft (\bC, C(V/\Pi) ) \to \prescript{\phi}{}  \KK^{G,\Pi}_\ft (C(V/\Pi),C(V/\Pi))$ denote the group homomorphism defined by  
\[ \Phi(\xi):=(\xi \otimes \id_{C(V/\Pi)}) \otimes_{C(V/\Pi ) \otimes C(V/\Pi)} [\Delta^*]. \]
Similarly we define $\Phi_V \colon \prescript{\phi}{} \KK^{G,\Pi}_\ft(\bC, C(V/\Pi) ) \to \prescript{\phi}{}  \KK^{G,\Pi}(C_0(V),C_0(V))$ given by $\Phi_V(\xi):=(\xi \otimes \id_{C_0(V)}) \otimes_{C(V/\Pi ) \otimes C_0(V)} [\Delta_V^*]$. Here Lemma \ref{lem:delta} is rephrased as
\[ \Phi = \prescript{\phi }{}{} j_{G,\Pi} \circ  \Phi_V. \]

Let us consider the diagram
\[
\xymatrix{
&\ar@{}[d]|{\circlearrowright} &\\
\prescript{\phi}{} \KK^{G}_\ft(C_0(V), C_0(V)) \ar[d]^{\blank \otimes_{C_0(V)} \sfD } \ar[r]^{\prescript{\phi}{} j_{G,\Pi} \hspace{3ex}} \ar@{}[rd]|{\hspace{2em} \circlearrowright}& \prescript{\phi}{} \KK^{G,\Pi}_\ft(C(V/\Pi) , C(V/\Pi)) \ar[r]^{\pi^*} \ar[d]^{\blank \otimes_{C(V/\Pi)} \prescript{\phi}{} j_{G,\Pi}(\sfD )}\ar@{}[rd]|{\hspace{4em} \circlearrowright} & \prescript{\phi}{} \KK^{G,\Pi}_\ft(\bC, C(V/\Pi)) \ar[d]^{ \blank \otimes_{C(V/\Pi)} \prescript{\phi}{} j_{G,\Pi} (\sfD )} \ar@/_3.5em/[ll]_{\Phi_V} \ar@/_1.5em/[l]_{\Phi}  \\
\prescript{\phi}{} \KK^{G}_\ft(C_0(V), \bC ) \ar[r]^{\prescript{\phi}{} j_{G,\Pi}\hspace{5ex}} \ar[dr]_{\cJ_\Pi} \ar@{}[rd]|{\hspace{5em} \circlearrowright }& \prescript{\phi}{} \KK^{G,\Pi}_\ft(C(V/\Pi) , C(\hat{\Pi})_\sigma ) \ar[r]^{\pi^*} \ar[d]^{ \blank \hotimes_{C(\hat{\Pi})_\sigma }[\iota_0^*]} & \prescript{\phi}{} \KK^{G,\Pi}_\ft (\bC, C(\hat{\Pi})_\sigma )  \\
&\prescript{\phi}{} \KK ^{P}_\ft (C(V/\Pi), \bC ). &
}
\]
The functoriality of $\prescript{\phi}{}j_{G,\Pi}$ and the above arguments show that parts of the above diagram indicated by the circular arrows are commutative. 
Moreover, by Lemma \ref{lem:Dirac} we have $\prescript{\phi}{} j_{G,\Pi} (\sfD) \hotimes_{C(\hat{\Pi})_\sigma} [\iota_0^*] =[L^2(V/\Pi,S), m \hotimes \fc, F]$, which shows that
\begin{align*}
    \PD_{V/\Pi}(\xi) 
    &= \Phi(\xi) \hotimes_{C(V/\Pi)} (\prescript{\phi}{} j_{G,\Pi}(\sfD) \hotimes_{C(\hat{\Pi})_\sigma} [\iota_0^*]).
\end{align*}
Therefore a diagram chasing shows that 
\begin{align*}
(\prescript{\phi}{} \mu_G^\Pi \circ (\cJ _\Pi)^{-1} \circ \PD_{V/\Pi})(\xi) & = \xi \hotimes \prescript{\phi}{} j_{G,\Pi}(\sfD),
\end{align*}
which finishes the proof by Theorem \ref{thm:crystal}.
\end{proof}

\section{T-duality between Atiyah--Hirzebruch spectral sequences}\label{section:6}
A promising application of Theorem \ref{thm:crystal} is the calculation of the twisted equivariant K-theory, which is well studied in the literature of condensed-matter physics \cites{shiozakiTopologicalCrystallineMaterials2017,shiozakiAtiyahHirzebruchSpectral2018}. 
The crystallographic T-duality implies  that there are two different spectral sequences converging to the same group; the Atiyah--Hirzebruch spectral sequence (AHSS) of $\prescript{\phi}{} \K_P^{*+d,\ft -\fv}(V/\Pi)$ and that of $\prescript{\phi}{} \K_P^{*,\ft + \sigma}(\hat{\Pi})$. 
These two spectral sequences are very different, in the sense that the map $\prescript{\phi}{}{\mathrm{T}}_{G}^\ft$ does not respect the Atiyah--Hirzebruch filtrations. Indeed, the (noncommutative) topology of the spaces on either side may be different, as is observed in non-equivariant setting in \cite{bouwknegtTopologyFluxDual2004} and \cites{mathaiTDualitySimplifiesBulkBoundary2016,hannabussTdualitySimplifiesBulkboundary2016}. 
For trivial $\phi,\ft$, the (ordinary) crystallographic T-duality isomorphism was exploited in \cite{gomiCrystallographicTduality2019} Section 8.3, to solve AHSS extension problems on one side by comparing with the other side where no extension problem occurs.
Here we provide some new examples, involving twisted crystallographic groups, of twisted equivariant K-group calculations via the comparison of Atiyah--Hirzebruch spectral sequences. 

We refer the reader to the work of Davis--L\"uck \cite{davisTopologicalKtheoryCertain2013}, which calculates the (non-twisted) K-group of the group C*-algebra of a large class of crystallographic groups in a similar spirit as our work, by a combination of topological methods including AHSS and the Baum--Connes isomorphism.

\subsection{$1$-dimensional black and white magnetic space group}
Here we consider the case that $V$ is of dimension $1$, $G = \bZ$ acting on $V$ by translation, and $\phi \colon G \to \bZ_2$ is the surjective homomorphism. 
Then $G$ is a twisted crystallographic group in the sense of Definition \ref{defn:twcry}, where the groups $\Pi$ and $P$ are $\Pi = 2\bZ$ and $P \cong \bZ_2$ respectively. We write $p \in P$ for the unique non-trivial element.  

The group $P$ acts on the torus $V/\Pi \cong S^1$ as rotation by $\pi$. Moreover, the linear $P$-action on $V$ is trivial and hence $\fv=(w_1^P(V),w_2^P(V)) =0$. We put the $P$-equivariant cell decomposition of $V/\Pi$ as:
\[
\xygraph{
!{<0cm,0cm>;<2cm,0cm>:<0cm,2cm>::}
!{(0,0) }*+{\circ }="1" ([]!{+(0,-0.2)} {\scriptstyle e_1^0})
!{(1,0) }*+{\circ }="2" ([]!{+(0,-0.2)} {\scriptstyle e_2^0})
!{(2,0) }*+{\circ} ="3" ([]!{+(0,-0.2)} {\scriptstyle e_1^0})
"1"-"2"_{e_1^1} "2"-"3"_{e_2^1}
}
\]
The action of $P=\bZ_2$ maps $e_1^0$ to $e_2^0$, $e_2^0$ to $e_1^0$, $e_{1}^1$ to $e_{2}^1$ and $e_{2}^1$ to $e_{1}^1$ respectively. 

The $P$-action on the Pontrjagin dual $\hat{\Pi} \cong U(1)$ as in (\ref{eq:Piact}) is the complex conjugation $z \mapsto \bar{z}$. We put the $P$-equivairant cell decomposition as:
\[
\xygraph{
!{<0cm,0cm>;<2cm,0cm>:<0cm,2cm>::}
!{(0,0) }*+{\circ }="1" ([]!{+(0,-0.2)} {\scriptstyle e_1^0})
!{(1,0) }*+{\circ }="2" ([]!{+(0,-0.2)} {\scriptstyle e_2^0})
!{(2,0) }*+{\circ} ="3" ([]!{+(0,-0.2)} {\scriptstyle e_1^0})
"1"-"2"_{e_1^1} "2"-"3"_{e_2^1}
}
\]
The action of $P=\bZ_2$ fixes $e_1^0$ and $e_2^0$, and maps $e_{1}^1$ to $e_{2}^1$ and $e_{2}^1$ to $e_{1}^1$ respectively. 
Moreover, the $2$-cocycle $\sigma \in H^2_P(\hat{\Pi}; \bZ_2)$ corresponding to the extension $1 \to \Pi \to G \to \bZ_2 \to 1$ as in (\ref{eq:sigma}) is determined by $\sigma (z,p,p)=z$ for any $z \in U(1)$.

\subsubsection{Type AI: AHSS for $\prescript{\phi}{} \K_P^*(V/\Pi)$}
We start with the case that $\ft:=(c,\tau)$ is trivial. Let us consider the AHSS for $\prescript{\phi}{} \K_P^*(V/\Pi)$. Recall that $\prescript{\phi}{} \K_{\bZ_2}^*(\bZ_2)$ is identified with $\K^*(\pt)$. Hence the $E_1$-page of this AHSS is as the left of Figure \ref{fig:AHSS1} (note that in this table $p$ and $q$ denote the rows and columns respectively). Moreover, by the following Lemma \ref{lem:E1}, we obtain the $E_2$-page of this AHSS as the right of Figure \ref{fig:AHSS1}.
\begin{figure}[h]
\begin{tabular}{c|cccc}
$7$ & $0$ && $0$ &  \\
$6$ & $\bZ$& &$\bZ$ & \\ 
$5$ & $0$ && $0$ &  \\
$4$ & $\bZ$ &&$\bZ$ & \\ 
$3$ & $0$ && $0$ &  \\
$2$ & $\bZ$ &&$\bZ$ & \\ 
$1$ & $0$ && $0$ &  \\
$0$ & $\bZ$ &&$\bZ$ & \\ 
\hline 
$E_1^{pq}$ & $0$ && $1$
\end{tabular}
\hspace{2em}
\begin{tabular}{c|cccc}
$7$ & $0$ && $0$ &  \\
$6$ & $0$ &&$\bZ_2$ & \\ 
$5$ & $0$ && $0$ &  \\
$4$ & $\bZ$ &&$\bZ$ & \\ 
$3$ & $0$ && $0$ &  \\
$2$ & $0$ &&$\bZ_2$ & \\ 
$1$ & $0$ && $0$ &  \\
$0$ & $\bZ$ &&$\bZ$ & \\ 
\hline 
$E_2^{pq}$ & $0$ && $1$
\end{tabular}
\caption{$E_1$ and $E_2$ pages of AHSS for $\prescript{\phi}{} \K_P^{*,-\fv}(V/\Pi)$}
\label{fig:AHSS1}
\end{figure}
\begin{lem}\label{lem:E1}
Under the above identification $E_1^{p,q} \cong \bZ$ for $p=0,1$ and even $q$, the differential $d_1^{0,q} \colon E_1^{0,q} \to E_1^{1,q}$ is 
\[d_1^{0,q} = \left\{ \begin{array}{ll} 0 & q \equiv 0 \text{ mod } 4, \\ 2 & q \equiv 2 \text{ mod }4.  \end{array} \right. \]
\end{lem}
\begin{proof}
We fix a Real structure, i.e., a complex conjugation on $\sH$ and extend it to $\sH_n:= \sH \otimes \Delta_{n,0}$. The complex Clifford algebra $\bC l_n:=\Cl_{n,0} \otimes _\bR \bC$ acts on $\sH_n$ and $\K^{n}(\pt)$ is isomorphic to the set of homotopy classes of the space of odd self-adjoint operators $F$ with $F^2-1 \in \bK(\sH_n)$ and $[F,\fc(v)]=0$ for any $v \in \bC l_n$. Now the complex conjugation induces an involution $\cT \colon [F] \mapsto [\bar{F}] $ on $\K^n(\pt)$. This involution is
\begin{align}
    \cT x = (-1)^nx \text{ for any $x \in \K^{2n}(\pt)$}. \label{eq:conjugate}
\end{align}

We use the trivial bundle $\cH:=V/\Pi \times \sH$, on which $\bZ_2$ acts by the composition of the rotation by $\pi$ and the complex conjugation, as the universal $\phi$-twisted $P$-equivariant Hilbert bundle. 
Now an element of $\prescript{\phi}{} \K_{\bZ_2}^{2n}(e_1^0 \cup e_2^0)$ is represented by $[F \cup \bar{F}]$ for some $F \in \Fred _{2n}(\sH)$. 
By definition of the boundary homomorphism
\[ \partial \colon \prescript{\phi}{} \K_P^{2n}(e_1^0 \cup e_2^0) \to \prescript{\phi}{} \K_P^{2n+1}( e_1^1 \cup e_2^1 ) \cong \K^{2n+1}(e_1^1 ), \]
it maps the element $[F \cup \bar{F}]$ to $([F] - [\bar{F}]) \otimes \beta $. Now (\ref{eq:conjugate}) finishes the proof.
\end{proof}

\begin{rmk}
Here we give a detail on the proof of (\ref{eq:conjugate}).
Let $\fM_n^\bC$ (resp.~$\fM_n^\bR$) denote the free abelian group generated by $\bZ_2$-graded representations of the complex (resp. real) Clifford algebra $\bC l_n$ (resp.~$\Cl_{n,0}$) and let $i_* \colon \fM_{n+1}^\bF \to \fM_{n}^\bF$ denote the forgetful map.
The Atiyah--Bott--Shapiro isomorphism $\K^n(\pt) \cong \fM_n^\bC /i_*\fM_{n+1}^\bC $ (resp.~$\KO^n(\pt) \cong \fM_n^\bR /i_*\fM_{n+1}^\bR $) is now given by mapping $[F]$ to the $\bZ_2$-graded representation $\ker (F)$ of $\bC l_n$. 
This correspondence is compatible with the complex conjugation, i.e., $[\bar{F}]$ is mapped to $\ker \bar{F} \cong \overline{\ker F}$. 

In $n=4k$, the complex conjugation acts trivially on $\K^n(\pt)$ since it must be trivial on the complexification map $\KO^n \to \K^n$. 
In $n=2$, the generator of $\fM_n / i_* \fM_{n+1}$ is $S=\bC^2$ on which the Clifford generators $e_1, e_2$ acts as $\big( \begin{smallmatrix} 1 & 0 \\ 0 & -1 \end{smallmatrix}\big) $ and $\big( \begin{smallmatrix} 0 & 1 \\ 1 & 0 \end{smallmatrix}\big) $ respectively, and the $\bZ_2$-grading is given by $\gamma =\big( \begin{smallmatrix} 0 & i \\ -i & 0 \end{smallmatrix}\big) $. Now $\bar{S}$ is the same representation of $\bC l_n$ with the opposite $\bZ_2$-grading. That is, $[\bar{S}]=-[S]$. 
Now we obtain (\ref{eq:conjugate}) since the complex conjugation is compatible with the multiplication in K-theory.
\end{rmk}

Finally we obtain that
\begin{align}\label{eqn:left.side}
\prescript{\phi}{} \K_P^*(V/\Pi) \cong \left\{ \begin{array}{ll} \bZ & * \equiv 0,1 \text{ mod }4,\\ 0 & * \equiv 2 \text{ mod }4,\\ \bZ_2 & * \equiv 3 \text{ mod }4.\\ \end{array} \right. 
\end{align}

\subsubsection{Type AI: AHSS for $\prescript{\phi}{} \K_P^{*+\sigma}(\hat{\Pi})$}
We also consider the AHSS for $\prescript{\phi}{} \K_P^{*,\sigma}(\hat{\Pi})$. Since $\sigma(e_1^0,p,p) =+1$ and $\sigma(e^0_2,p,p)=-1$, we have $\prescript{\phi}{} \K_P^{*,\sigma}(e_1^0) \cong \KO^*$ and $\prescript{\phi}{} \K_P^{*,\sigma}(e_2^0) \cong \KO^{*+4}$. 
Similarly we also have $\prescript{\phi}{} \K_P^{*,\sigma}(e^1_1 \cup e^1_2) =  \prescript{\phi}{} \K_P^*(e_1^1 \cup e_2^1)=\K^*(\pt)$ since $\sigma \in H_P^2(e_1^1 \cup e_2^1; \prescript{\phi}{} \bT) \cong 0$.
Hence the $E_1$-page is as in the left of Figure \ref{fig:AHSS2}. Moreover, the differentials $d_1^{0,q}$ are given by
\[d_1^{0,q} = ( \cC \ \cC) \colon \KO^q \oplus \KO^{q+4} \to \K^q, \]
where $\cC \colon \KO^* \to \K^*$ denotes the complexification homomorphism. Hence $d_1^{0,q}$ are surjective for $q=0,4$ and otherwise trivial. 
Therefore, the $E_2$-page of this AHSS is as the right of Figure \ref{fig:AHSS2}. 
 \begin{figure}[h]
\begin{tabular}{c|cccc}
$7$ & $\bZ_2$ && $0$ &  \\
$6$ & $\bZ_2$ && $\bZ$ & \\ 
$5$ & $0$ && $0$ &  \\
$4$ & $\bZ^2$ &&$\bZ$ & \\ 
$3$ & $\bZ_2$ && $0$ &  \\
$2$ & $\bZ_2$ &&$\bZ$ & \\ 
$1$ & $0$ && $0$ &  \\
$0$ & $\bZ^2$ &&$\bZ$ & \\ 
\hline 
$E_1^{pq}$ & $0$ && $1$
\end{tabular}
\hspace{2em}
\begin{tabular}{c|cccc}
$7$ & $\bZ_2$ && $0$ &  \\
$6$ & $\bZ_2$ &&$\bZ$ & \\ 
$5$ & $0$ && $0$ &  \\
$4$ & $\bZ$ && $0$ & \\ 
$3$ & $\bZ_2$ && $0$ &  \\
$2$ & $\bZ_2$ &&$\bZ$ & \\ 
$1$ & $0$ && $0$ &  \\
$0$ & $\bZ$ && $0$ & \\ 
\hline 
$E_2^{pq}$ & $0$ && $1$
\end{tabular}
\caption{$E_1$ and $E_2$ pages of AHSS for $\prescript{\phi}{} \K_P^{*,\sigma}(\hat{\Pi})$} \label{fig:AHSS2}
\end{figure}

Now, thanks to Theorem \ref{thm:crystal}, a comparison with \eqref{eqn:left.side} shows that the extension $0 \to E_2^{1,q-1} \to \prescript{\phi}{} \K_P^{q,\sigma}(\hat{\Pi}) \to E_2^{0,q} \to 0$ is nontrivial if $q \equiv 3$ modulo $4$, and hence
\begin{align}
\prescript{\phi}{} \K_P^{*,\sigma}(\hat{\Pi}) \cong \left\{ \begin{array}{ll} \bZ & * \equiv 0,3 \text{ mod }4,\\ 0 & * \equiv 1 \text{ mod }4,\\ \bZ_2 & * \equiv 2 \text{ mod }4.\\ \end{array} \right. \label{eq:1d}
\end{align}

\subsubsection{Other symmetry types} 
We consider some variations of the $1$-dimensional black and white magnetic space group studied above. 
\begin{enumerate}
    \item The case that $G$, $\phi$ are as above, $c$ is trivial and $\tau \in H^2(P;\prescript{\phi}{} \bT)$ is determined by $\tau(p,p)=-1$ (type AII). 
    \item The case that $G$, $\phi$ are as above and $c \colon G \to \bZ_2$ is the surjection. There are two choices of $\tau$ determined by $\tau(p,p)=\pm 1$ (types C and D).
    \item The case that $G = \bZ \times \bZ_2$, $\phi$ is the composition of the projection $\pr_{\bZ}$ and the surjection $\bZ \to \bZ_2$ and $c :=\pr_{\bZ_2}$. The twisted point group is $P=\mathbb{Z}_2\times\mathbb{Z}_2$ generated by $p$ and $c$, and there are four choices for $\tau$ given by $\tau(p,p)=\pm 1$ and $\tau(c,c)=\pm 1$, corresponding to symmetry types BDI, DIII, CI, and CII.
\end{enumerate}
In each of these cases, the $E_1$ and $E_2$ pages of the AHSS for $\prescript{\phi}{} \K_P^{*,\ft-\fv}(V/\Pi)$ indicated in Figure \ref{fig:AHSS1} is shifted by $m \in \bZ/8\bZ$ depending on the Altland--Zirnbauer (AZ) symmetry type as Table \ref{table:sym}. Indeed, as is well-understood in the study of topological insulators (see e.g.~\cite{kubotaNotesTwistedEquivariant2016}*{Corollary 5.16}), the functor $\prescript{\phi}{}\K_P^{*,\ft-\fv}$ is naturally identified with the type AI (i.e.~Real) K-functor $\prescript{\phi}{}\K_{\bZ_2}^{*+m,-\fv}$.
\begin{table}[h]
    \centering
    \begin{tabular}{c|cccccccc}
        Cartan & AI & BDI & D & DIII & AII & CI & C & CII \\ \hline
        $m$ & $0$ & $1$ & $2$&$3$& $4$ & $ 5 $ & $ 6 $ & $ 7 $  \\  
    \end{tabular}
    \caption{The degree shift corresponding to the symmetry type.}
    \label{table:sym}
\end{table}

Therefore, the resulting group $\prescript{\phi}{} \K_P^{*,\ft}(V/\Pi)$ is isomorphic to (\ref{eq:1d}) shifted by $m$, that is, 
\[\prescript{\phi}{} \K_P^{*,\ft}(V/\Pi) \cong \left\{ \begin{array}{ll} \bZ & *+m \equiv 0,3 \text{ mod }4,\\ 0 & *+m \equiv 1 \text{ mod }4,\\ \bZ_2 & *+m \equiv 2 \text{ mod }4.\\ \end{array} \right. \]

\subsection{$2$-dimensional \textsf{pg} grey magnetic space group}
The second example is the case that $G$ is the $2$-dimensional magnetic space group of type $\mathsf{pg}$. That is, $G$ is the subgroup of $\Euc(V)$ generated by $a \colon (x,y) \mapsto (x+1,y)$ and $b \colon (x,y) \mapsto (-x,y+1)$. 
Let $\Pi=\langle a,b^2\rangle \subset G $. Then $\Pi$ is a free abelian subgroup of $G$ acting on $V$ by translation. 
Let $\phi \colon G \to  \bZ_2$ be the quotient $G \to P:= G/\Pi \cong \bZ_2$. 
Then $(G,\phi)$ is a magnetic space group in the sense of Definition \ref{defn:magnetic}, and hence $(G,\phi,0,0)$ is a twisted crystallographic group in the sense of Definition \ref{defn:twcry}. 
Note that, since $b$ reverses the orientation of $V$, the twist $\fv= (w_1^P(V),w_2^P(V))$ (cf.\ Section \ref{section:4}) is non-trivial. 
 
\subsubsection{Type AI: AHSS for $\prescript{\phi}{} \K_P^*(V/\Pi)$}
Let us consider the AHSS for $\prescript{\phi}{} \K_P^{*,-\fv}(V/\Pi)$ with respect to the following $P$-equivariant cell decomposition of $V/\Pi$:
\[
\xygraph{
!{<0cm,0cm>;<1.5cm,0cm>:<0cm,1.5cm>::}
!{(0,0)}*+{\circ}="a" ([]!{+(-0.2,0)} {\scriptstyle e_1^0})
!{(2,0)}*+{\circ}="b" ([]!{+(+0.2,0)} {\scriptstyle e_1^0})
!{(0,1)}*+{\circ}="c" ([]!{+(-0.2,0)} {\scriptstyle e_2^0})
!{(2,1)}*+{\circ}="d" ([]!{+(+0.2,0)} {\scriptstyle e_2^0})
!{(0,2)}*+{\circ}="e" ([]!{+(-0.2,0)} {\scriptstyle e_1^0})
!{(2,2)}*+{\circ}="f" ([]!{+(+0.2,0)} {\scriptstyle e_1^0})
"a"-"c" |{e_3^1}
"c"-"e" |{e_4^1}
"b"-"d" |{e_3^1}
"d"-"f" |{e_4^1}
"a"-"b" |{e_1^1}
"c"-"d" |{e_2^1}
"e"-"f" |{e_1^1}
!{(1,0.5)}*+{\scriptstyle e_1^2}
!{(1,1.5)}*+{\scriptstyle e_2^2}
}
\]
The $P$-action maps $e_1^0$ to $e_2^0$ and $e_3^1$ to $e_4^1$. It also maps $e_1^1$ to $e_2^1$ and $e_1^2$ to $e_2^2$ while reversing the $x$-coordinate.

Whatever the twist $(\phi, \fv)$ is, the group $\prescript{\phi}{} \K_{\bZ_2}^{*,-\fv}(\bZ_2)$ is identified with the complex K-group $\K^*(\pt)$. Hence we have
\begin{align*}
    \prescript{\phi}{} \K_P^{*,-\fv}(e_1^0 \cup e_2^0) \cong &\K^*,\\
    \prescript{\phi}{} \K_P^{*,-\fv}(e_1^1 \cup e_2^1) \cong &\K^{*+1} \cong \prescript{\phi}{} \K_P^{*,-\fv}(e_3^1 \cup e_4^1), \\
    \prescript{\phi}{} \K_P^{*,-\fv}(e_1^2 \cup e_2^2)\cong &\K^{*+2},
\end{align*}
and hence the $E_1$-page of this AHSS is as the left of Figure \ref{fig:AHSS3}. 

Let $T:=e_1^0 \cup e_2^0 \cup e_3^1 \cup e_4^1$. Then the $P$-equivariant maps 
\[ T \xrightarrow{\mathrm{inclusion}} V/\Pi \xrightarrow{\mathrm{projection}} T \]
induce a direct sum decomposition 
\[ \prescript{\phi}{} \K_{\bZ_2}^{*,-\fv}( V/\Pi) \cong \prescript{\phi}{} \K_{\bZ_2}^{*,-\fv}( T) \oplus \prescript{\phi}{} \K_{\bZ_2}^{*,-\fv}( V/\Pi \setminus T).\]
We apply Lemma \ref{lem:E1} to obtain that the differential $d_1^{pq}$ of the AHSS for $V/\Pi$ restricted to the subcomplex $T$ is
\[
d_1^{0q} = \left\{ \begin{array}{ll} 2 \colon \bZ \to \bZ & q \equiv 0 \text{ mod } 4, \\ 0 \colon \bZ \to \bZ & q \equiv 2 \text{ mod }4,  \end{array} \right. 
\]
and otherwise zero (note that the degree is shifted by $2$ due to the twist $-\fv$). 
We also remark that the Thom isomorphism (Definition \ref{defn:Thom}) shows that $\prescript{\phi}{} \K_{\bZ_2}^{*,-\fv}(\bR^{0,1} \times X) \cong \prescript{\phi}{} \K_{\bZ_2}^{*+3}(X)$, where $\bR^{0,1}$ is the real line $\bR$ with the $\bZ_2$-action given by the reflection.  
Since $V/\Pi \setminus T \cong \bR^{0,1} \times T$, the differential $d_1^{pq}$ of the AHSS for $V/\Pi$ restricted to the subcomplex $V/\Pi \setminus T$ is
\[
d_1^{1q} = \left\{ \begin{array}{ll} 0 \colon \bZ \to \bZ & q \equiv 0 \text{ mod } 4, \\ 2\colon \bZ \to \bZ & q \equiv 2 \text{ mod }4,  \end{array} \right. 
\]
and otherwise zero. 
Therefore, the $E_2$-page of this AHSS is the right side of Figure \ref{fig:AHSS3}.
\begin{figure}[h]
\begin{tabular}{c|cccccc}
$7$ & $0$ && $0$ && $0$ &  \\
$6$ & $\bZ$& &$\bZ^2$ && $\bZ $ &  \\ 
$5$ & $0$ && $0$ && $0$ &  \\
$4$ & $\bZ$ &&$\bZ^2$ && $\bZ$ & \\ 
$3$ & $0$ && $0$ && $0$ &  \\
$2$ & $\bZ$ &&$\bZ^2$ && $\bZ$ & \\ 
$1$ & $0$ && $0$ && $0$ &  \\
$0$ & $\bZ$ &&$\bZ^2$ &&$\bZ$ &  \\ 
\hline 
$E_1^{pq}$ & $0$ && $1$ && $2$ &
\end{tabular}
\hspace{2em}
\begin{tabular}{c|cccccc}
$7$ & $0$ && $0$ && $0$ &  \\
$6$ & $\bZ$ &&$\bZ$ &&$\bZ_2$ &  \\ 
$5$ & $0$ && $0$ && $0$ &  \\
$4$ & $0$& &$\bZ_2 \oplus \bZ$ && $\bZ $ &  \\ 
$3$ & $0$ && $0$ && $0$ &  \\
$2$ & $\bZ$ &&$\bZ$ && $\bZ_2$ & \\ 
$1$ & $0$ && $0$ && $0$ &  \\
$0$ & $0$ &&$\bZ_2 \oplus \bZ$ && $\bZ$ & \\ 
\hline 
$E_2^{pq}$ & $0$ && $1$&& $2$ &
\end{tabular}
\caption{$E_1$ and $E_2$ pages of AHSS for $\prescript{\phi}{} \K_P^{*,-\fv}(V/\Pi)$}
\label{fig:AHSS3}
\end{figure}
Now the differential $d_2^{pq}$ is zero and there is no extension problem. Hence we obtain that  
\begin{align}
\prescript{\phi}{} \K_P^{*,-\fv}(V/\Pi) \cong \left\{ \begin{array}{ll} \bZ_2 & * \equiv 0 \text{ mod }4,\\\bZ \oplus \bZ_2 & * \equiv 1 \text{ mod }4,\\ \bZ^2 & * \equiv 2 \text{ mod }4,\\ \bZ & * \equiv 3 \text{ mod }4.\\ \end{array} \right. \label{eq:2dpg}
\end{align}
We remark that essentially the same calculation as above is done by Baraglia \cite{baragliaConformalCourantAlgebroids2013} in his study of topological T-duality of twisted $\mathrm{KR}$-theory.

\subsubsection{Type AI: AHSS for $\prescript{\phi}{} \K_P^{*+\sigma}(\hat{\Pi})$}
Next we calculate the AHSS of $\prescript{\phi}{} \K_P^{*,\ft}(\hat{\Pi})$ with respect to the following cell decomposition. Let $(k_x,k_y)$ denote the coordinate of $\hat{\Pi}$ dual to the standard $(x,y)$-coordinate of $V$.  We put the cell decomposition of $\hat{\Pi}$ as:
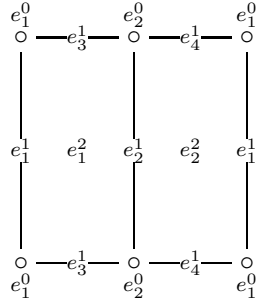
\[
\xygraph{
!{<0cm,0cm>;<1.5cm,0cm>:<0cm,1.5cm>::}
!{(0,0)}*+{\circ}="a" ([]!{+(0,-0.2)} {\scriptstyle e_1^0})
!{(1,0)}*+{\circ}="b" ([]!{+(0,-0.2)} {\scriptstyle e_2^0})
!{(2,0)}*+{\circ}="c" ([]!{+(0,-0.2)} {\scriptstyle e_1^0})
!{(0,2)}*+{\circ}="d" ([]!{+(0,0.2)} {\scriptstyle e_1^0})
!{(1,2)}*+{\circ}="e" ([]!{+(0,0.2)} {\scriptstyle e_2^0})
!{(2,2)}*+{\circ}="f" ([]!{+(0,0.2)} {\scriptstyle e_1^0})
"a"-"b"|{e_3^1}
"b"-"c"|{e_4^1}
"d"-"e"|{e_3^1}
"e"-"f"|{e_4^1}
"a"-"d"|{e_1^1}
"b"-"e"|{e_2^1}
"c"-"f"|{e_1^1}
!{(0.5,1)}*+{\scriptstyle e_1^2}
!{(1.5,1)}*+{\scriptstyle e_2^2}
}
\]
on which $P$ acts by the reflection along $k_y$-axis. The $2$-cocycle $\sigma$ is determined by $\sigma(\blank, p,p)=e^{2\pi ik_y}$. Hence we have 
\begin{align*}
    \prescript{\phi}{} \K_P^{*,\sigma}(e_1^0) \cong &\KO^* \cong \prescript{\phi}{} \K_P^{*+1,\sigma}(e^1_1),\\
    \prescript{\phi}{} \K_P^{*,\sigma}(e_2^0) \cong &\KO^{*+4} \cong \prescript{\phi}{} \K_P^{*+1,\sigma}(e_2^1),\\
    \prescript{\phi}{} \K_P^{*+1,\sigma}(e_3^1 \cup e_4^1) \cong &\K^* \cong \prescript{\phi}{} \K_P^{*+2} (e_1^2 \cup e_2^2).
\end{align*}
Therefore the $E_1$-page is as in the left of Figure \ref{fig:AHSS2}. By comparing this AHSS with that of $e_1^0 \cup e_1^1 \cong S^1$, $e_2^0 \cup e_2^1 \cong S^1$, $e_1^0 \cup e_2^0 \cup e_3^1 \cup e_4^1 \cong S^1$ and $e_1^1 \cup e_2^1 \cup e_1^2 \cup e_2^2 \cong \bR^{0,1} \times S^1$ along the homomorphism induced from the inclusion maps, we obtain that the differential $d_1^{pq}$ is
\begin{align*}
d_1^{0q} = \Big( \begin{smallmatrix}\cC & \cC \\ 0 & 0 \\ 0 & 0 \end{smallmatrix} \Big) &\colon \KO ^* \oplus \KO^{*+4} \to \K^* \oplus \KO ^* \oplus \KO^{*+4},  \\
d_1^{1q} = (\begin{smallmatrix}0 & \cC & \cC \end{smallmatrix}) & \colon \K^* \oplus  \KO ^* \oplus \KO^{*+4} \to \K^*.
\end{align*}
Hence the $E_2$-page is as in the right of Figure \ref{fig:AHSS4}.

\begin{figure}[h]
\begin{tabular}{c|cccccc}
$7$ & $\bZ_2$ && $0 \oplus \bZ_2$ && $0$ &  \\
$6$ & $\bZ_2$& &$\bZ \oplus \bZ_2$ && $\bZ $ &  \\ 
$5$ & $0$ && $0 \oplus 0$ && $0$ &  \\
$4$ & $\bZ^2$ &&$\bZ \oplus \bZ^2$ && $\bZ$ & \\ 
$3$ & $\bZ_2$ && $0 \oplus \bZ_2$ && $0$ &  \\
$2$ & $\bZ_2$ &&$\bZ \oplus \bZ_2$ && $\bZ$ & \\ 
$1$ & $0$ && $0 \oplus 0$ && $0$ &  \\
$0$ & $\bZ^2$ &&$\bZ \oplus \bZ^2$ &&$\bZ$ &  \\ 
\hline 
$E_1^{pq}$ & $0$ && $1$ && $2$ &
\end{tabular}
\hspace{2em}
\begin{tabular}{c|cccccc}
$7$ & $\bZ_2$ && $\bZ_2$ && $0$ &  \\
$6$ & $\bZ_2$& &$\bZ \oplus \bZ_2$ && $\bZ $ &  \\ 
$5$ & $0$ && $0$ && $0$ &  \\
$4$ & $\bZ$ &&$\bZ$ && $0$ & \\ 
$3$ & $\bZ_2$ && $\bZ_2$ && $0$ &  \\
$2$ & $\bZ_2$ &&$\bZ \oplus \bZ_2$ && $\bZ$ & \\ 
$1$ & $0$ && $0$ && $0$ &  \\
$0$ & $\bZ$ &&$\bZ$ &&$0$ &  \\ 
\hline 
$E_2^{pq}$ & $0$ && $1$&& $2$ &
\end{tabular}
\caption{$E_1$ and $E_2$ pages of AHSS for $\prescript{\phi}{} \K_P^{*,\sigma}(\hat{\Pi})$} \label{fig:AHSS4}
\end{figure}

It is automatic from Figure \ref{fig:AHSS4} that the higher differentials $d_n^{pq}$ ($n \geq 2$) vanish. 
Although the extension problem at $p+q \equiv 0 $ and $3$ modulo $4$ are nontrivial, we obtain  
\[\prescript{\phi}{} \K_P^{*,\sigma}(\hat{\Pi}) \cong \left\{ \begin{array}{ll} \bZ^2 & * \equiv 0 \text{ mod }4,\\ \bZ & * \equiv 1 \text{ mod }4,\\ \bZ_2 & * \equiv 2 \text{ mod }4,\\ \bZ \oplus \bZ_2 & * \equiv 3 \text{ mod }4,\\ \end{array} \right. \]
by comparing Figure \ref{fig:AHSS4} with (\ref{eq:2dpg}) through the twisted crystallographic T-duality isomorphism.

\begin{rmk}
The reasons that $\prescript{\phi}{}\K^{n,-\fv}(V/\Pi)$ and $\prescript{\phi}{}\K^{n, \sigma}(\hat{\Pi})$ are periodic with period $4$ are as follows:
\begin{itemize}
\item Let $\tau _\phi \in H^2(P;\prescript{\phi}{}\bT)$ be the cocycle determined by $\tau(e,e)=\tau(e,p)=\tau(p,e)=1$ and $\tau_\phi(p,p)=-1$. 
Then $\prescript{\phi}{}\K^{n,\tau_\phi}_P(V/\Pi)$ is identified with the Quaternionic $\K$-theory. The product line bundle $L = V / \Pi \times \bC$ gives rise to a Quaternionic line bundle on $V /\Pi$ by the Quaternionic structure $(x, y, z) \mapsto (-x, y+ 1, e^{\pi i y} \bar{z})$. 
This line bundle defines an element $[L] \in \prescript{\phi}{}\K^{n,\tau_\phi}_P(V/\Pi)$, and its internal tensor product induces an isomorphism $\prescript{\phi}{}\K^{n,-\fv}_P(V/\Pi) \to \prescript{\phi}{}\K^{n,-\fv+\tau_\phi}_P(V/\Pi)$, whose inverse is the tensor product with the dual line bundle $L^*$. It is known \cite{gomiFreedMooreKtheory2017}*{Subsection 3.3} that $\tau_\phi$ has
the effect of degree shift by $4$. Combining these isomorphisms, we get the periodicity $\prescript{\phi}{}\K^{n,-\fv}_P (V/\Pi) \cong \prescript{\phi}{}\K^{n+4,-\fv}_P (V/\Pi)$. (The same argument can be seen in \cite{baragliaConformalCourantAlgebroids2013}.)
\item 
Let $\iota \colon \hat{\Pi} \to \hat{\Pi}$ denote the homeomorphism $\iota (k_x, k_y) = (k_x, k_y + 1/2)$. Since it is $P$-equivariant, its pull-back induces an isomorphism $\iota^* \colon \prescript{\phi}{}\K^{n,\sigma}_P(\hat{\Pi}) \to \prescript{\phi}{}\K^{n,\iota^* \sigma}_P(\hat{\Pi})$. 
Since $\iota^*\sigma(p,p)(k_x,k_y) = e^{2\pi i (k_y + 1/2)} =-\sigma(p,p)(k_x,k_y)$, we have $\iota^*\sigma = \sigma +\tau_\phi$, and $\tau_\phi$ has the effect of the
degree shift by $4$. 
Combining them, we get the periodicity $\prescript{\phi }{}\K^{n,\sigma}_P(\hat{\Pi}) \to \prescript{\phi }{}\K^{n+4,\sigma}_P(\hat{\Pi})$.
\end{itemize}
We remark that the same argument shows that the twisted equivariant $\K$-groups (\ref{eqn:left.side}) and (\ref{eq:1d}) have the periodicity with period 4.
\end{rmk}

\subsubsection{Other symmetry types}
In the same way as the $1$-dimensional case, we may consider some variations of $2$-dimensional \textsf{pg} grey magnetic space group. Each of them corresponds to one of eight real Cartan labels as following.
\begin{enumerate}
    \item The case that $G$, $\phi$, $c$ are as above and $\tau \in H^2(P;\prescript{\phi}{} \bT)$ is determined by $\tau(p,p)=-1$ (type AII).
    \item The case that $G$, $\phi$ are as above and $c \colon G \to \bZ_2$ is the surjection. There are two choices of $\tau$ determined by $\tau(p,p)=\pm 1$ (type C and D).
    \item The case that $G = \mathsf{pg} \times \bZ_2$, $\phi$ is the composition of $\pr_{ \mathsf{pg} }$ and the surjection $ \mathsf{pg}  \to \bZ_2$ and $c :=\pr_{\bZ_2}$. There are four choices of $\tau$ determined by $\tau(p,p)=\pm 1$ and $\tau(c,c) = \pm 1$, where $p , c$ are generators of $P=G/\Pi \cong \bZ _2 \times \bZ_2$, corresponding to symmetry types BDI, DIII, CI, CII.
\end{enumerate}
The same calculation as the case of type AI symmetry shows that in each of these cases the twisted equivariant K-group of $V/\Pi$ is isomorphic to (\ref{eq:2dpg}) shifted by $m$, where $m \in \bZ/8\bZ$ is the one indicated in Table \ref{table:sym}.

\section{Functoriality of the twisted crystallographic T-duality}\label{section:7}
In this section we discuss the functoriality of the twisted crystallographic T-duality map. For the Baum--Connes assembly map, it was proved by Valette \cite{valetteBaumConnesAssemblyMap2003} that, for a homomorphism $\phi \colon \Gamma_1 \to \Gamma_2$ of groups, there is a commutative diagram
\[
\xymatrix{
\K_*^{\Gamma_1}(\underline{E}\Gamma_1) \ar[r]^{\mu_{\Gamma_1}} \ar[d] & \K_*(C^*\Gamma_1) \ar[d] \\
\K_*^{\Gamma_2}(\underline{E}\Gamma_2) \ar[r]^{\mu_{\Gamma_2}} & \K_*(C^*\Gamma_2). 
}
\]
Theorem \ref{thm:BC} suggests that a similar functoriality holds for twisted crystallographic T-duality. 

Here we treat the following setting: Let $G$ be a twisted crystallographic group and let $H$ be a subgroup of $G$. We write $\Sigma := H \cap \Pi$ and $Q:=H/\Sigma \subset P$. Then $H$ is also regarded as a twisted crystallographic group by the following lemma.  
\begin{lem}\label{lem:affsub}
There is an affine subspace $W \subset V $ of rank $d'$ such that $\alpha_h(W) = W$ for any $h \in H$ and $\Sigma \subset \Euc(W) \cap \bR^d \cong \bR^{d'}$ is of full-rank. 
\end{lem}
\begin{proof}
Let $\Pi_\bR$ and $\Sigma _\bR$ denote the $\bR$-linear spans of $\Pi$ and $\Sigma$ respectively. Note that $\Pi_\bR$ is nothing but the group $\bR^d$ of translations. Since $h \Sigma_\bR h^{-1} = \Sigma_\bR$, the action of $H$ on $V$ is reduced to the $Q$-action of the affine space $V/\Sigma_\bR$. Since this is an affine action of a finite group, there is a fixed point in $V/\Sigma_\bR$, which is represented by a $\Sigma_\bR$-orbit $W:=v \Sigma_\bR \in V/\Sigma_\bR$. This is the desired affine subspace of $V$. 
\end{proof}

\subsection{Induction of topological insulators}
Now we define the induction of topological insulators with the symmetry of twisted crystallographic groups.  Let $\bX:=G \cdot x$ be a $G$-orbit in $V$. We choose the reference point $x \in V$ as an element of $W$.
Here we employ the internal degree of freedom $\sK$ over $\bX$ as the one of the form $\sK' \hotimes \ell^2(G_x)$, in the way that $\ell^2(\bX ,\sK) \cong \ell^2(G, \sK')$. 
Then it is decomposed as the direct sum $\bigoplus_{gH \in G/H} \ell^2(H,\sK')$.
\begin{defn}\label{defn:induction}
Let $h \in \cH_N(H,\phi,\ft)$. We define the induction of topological insulator as 
\[ \Ind _H^G (h):= \bigoplus_{gH \in G/H} U_ghU_g^*  \in \bB \Big(\bigoplus_{gH \in G/H} \ell^2(H, \sK' )\Big).  \]
We also use the same letter $\Ind_H^G$ for the induced map $\mathscr{TP}(H,\phi,\ft) \to \mathscr{TP}(G,\phi,\ft)$. 
\end{defn}
Under the identification in Proposition \ref{prp:TP}, this homomorphism is described as a composition of twisted equivariant K-theory operations. To see this, we decompose the induction into two steps. For $G$ and $H$ as above, let $H'$ denote the intermediate subgroup $q^{-1}(Q) \subset G$, where $q \colon G \to P$ denote the projection. Then $H'$ is also a twisted crystallographic group acting on $V$ and $\sK$, such that $H' \cap \bR^d = \Pi$ and has point group $Q$. By definition we have 
\begin{align}
\Ind _H^G = \Ind_{H'}^G \circ \Ind_H^{H'}. \label{eq:indstep}
\end{align} 
Here we write $\sigma_H$, $\sigma_{H'}$ and $\sigma_G$ for the $2$-cocycles on $Q \curvearrowright \hat{\Sigma}$, $Q \curvearrowright \hat{\Pi}$ and $P \curvearrowright \hat{\Pi}$ associated to the extensions $H$, $H'$ and $G$ as in (\ref{eq:sigma}) respectively. 

Firstly we describe the map $\Ind_{H'}^G$, i.e., the induction in the case that $\Sigma = \Pi$.
\begin{lem}
Let $P$ be a finite group and let $Q \leq P$ be a subgroup. For a $Q$-space $X$ and a twist $(\phi,\fs)$ of $Q \curvearrowright X$, there is a twist $(\tilde{\phi},\tilde{\fs})$ of $P \curvearrowright P \times_Q X$ and an isomorphism
\[\cI_Q^P \colon \prescript{\phi}{} \K_Q^{*,\fs} (X) \to \prescript{\tilde{\phi}}{} \K_P^{*,\tilde{\fs}}(P \times _Q X). \]
Moreover, if $X$ is a $P$-space and $(\phi,\fs)$ is the restriction of a twist of $P \curvearrowright X$, then $(\tilde{\phi}, \tilde{\fs})$ is the pull-back of $(\phi,\fs)$ by the $P$-equivariant map $\rho \colon P \times _Q X \to X$ given by the multiplication $[p,x] \mapsto px$.
\end{lem}
\begin{proof}
This follows from the fact that the inclusion of action groupoids $X \rtimes Q \to  (P \times _Q X) \rtimes P$ is a local equivalence. Indeed, the map
\[ \Gamma _c(P \times _Q X,\Fred (P \times _Q \cH ))^P \to \Gamma_c(X, \Fred(\cH))^Q \] 
given by the restriction of the section to $X \subset P \times _Q X$ is a homeomorphism. The bundle $P \times _Q \cH$ over $P \times _Q X$ is a universal twisted $P$-equivariant Hilbert bundle, to which a twist $(\tilde{\phi}, \tilde{\fs})$ is associated. Finally, consider the case that $X$ is a $P$-space and $\cH$ is a twisted $P$-equivariant Hilbert bundle on $X$. Then $P \times _Q \cH$ is isomorphic to $\rho^*\cH$, and hence $(\tilde{\phi}, \tilde{\fs}) = \rho^*(\phi,\fs)$. 
\end{proof}

\begin{defn}
Let $X$ be a $P$-space and let $(\phi,\fs)$ be a twist on $P \curvearrowright X$. We define the induction map as the composition
\[\Ind _Q^P:= \rho_! \circ \cI_Q^P \colon \prescript{\phi}{} \K_Q^{*,\fs} (X) \cong \prescript{\rho^*\phi}{} \K_P^{*,\rho^*\fs} (P \times _Q X) \to \prescript{\phi}{} \K_P^{*,\fs}(X).\]
Here $\rho_!$ is the push-forward map with respect to $\rho$. 
\end{defn}

\begin{lem}
Through the isomorphism in Proposition \ref{prp:TP}, the induction $\Ind_{H'}^G$ in the sense of Definition \ref{defn:induction} is identified with $\Ind_Q^P \colon \prescript{\phi}{} \K_Q^{*,\ft+\sigma_{H'}}(\hat{\Pi}) \to \prescript{\phi}{} \K_P^{*,\ft +\sigma_G}(\hat{\Pi})$.
\end{lem}
\begin{proof}
First of all, the statement makes sense because $\sigma_{H'}$ is the restriction of $ \sigma_{G}$ to $Q \curvearrowright \hat{\Pi}$, which follows by definition (\ref{eq:sigma}). 
Now the lemma follows from the following presentation of the push-forward $\rho_!$ with respect to a finite covering map: Let $\rho_!\cH$ denote the $(\phi,\ft)$-twisted $P$-equivariant Hilbert bundle on $X$ defined as $(\rho_!\cH)_x = \bigoplus_{\rho(\bar{x})=x} \cH_{\bar{x}} $ and let $\rho_!F \in \Gamma_c(X,\Fred(\rho_!\cH))^P$ given by
\[\rho_! (F)_x= \bigoplus_{\rho(\bar{x})=x } F_{\bar{x}}  \in \Fred(\rho_!\cH)_x. \]
Then we have $\rho_![F]=[\rho_!(F)]$. This is proved in the same way as the case of genuine push-forward, given in  \cite{connesLongitudinalIndexTheorem1984}*{Proposition 2.9}.
\end{proof}

Next we describe the map $\Ind_H^{H'}$, i.e., the induction in the case when the twisted crystallographic groups have the same point group $Q$. 
Here, the inclusion $\Sigma \to \Pi$ induces the surjection of Pontrjagin duals $\theta \colon \hat{\Pi} \to \hat{\Sigma}$. 
This map is $Q$-equivariant and $\theta^*\sigma_H = \sigma|_{H'}$ holds, which is checked as 
\[\sigma_{H'}(p,q)_\chi = \chi(s(p)s(q)s(pq)^{-1}) = \sigma_H(p,q)_{\theta (\chi)}, \]
since $s(p)s(q)s(pq)^{-1} \in \Sigma$ and $\theta(\chi):=\chi|_{\Sigma}$. 
Hence the pull-back by $p$ induces a homomorphism between twisted equivariant K-groups under consideration.
\begin{lem}
Through the isomorphism in Proposition \ref{prp:TP}, the induction $\Ind_{H}^{H'}$ is identified with $\theta^*  \colon \prescript{\phi}{} \K_Q^{*,\ft+\sigma_{H}}(\hat{\Sigma}) \to \prescript{\phi}{} \K_Q^{*,\ft +\sigma_{H'}}(\hat{\Pi})$.
\end{lem}
\begin{proof}
Through the Fourier transform, a gapped Hamiltonian $h \in \cH_N(H,\phi,\ft)$ corresponds to a continuous function on $\hat{\Sigma}$ taking values in $\bB(\tilde{\sK})$. Now it is a standard fact of the Fourier transform that $\Ind_H^{H'} h = \theta^*h$ as $\bB(\tilde{\sK})$-valued functions on $\hat{\Pi}$.
\end{proof}

In summary, we obtain the following description of the induction map. For simplicity of notation, hereafter we just write $\sigma$ instead of $\sigma_G$, $\sigma_H$ or $\sigma_{H'}$ when the subscript is clear from the context. 
\begin{prp}
Under the isomorphism in Proposition \ref{prp:TP}, the induction $\Ind_H^G$ in the sense of Definition \ref{defn:induction} corresponds to the composition
\[ \Ind_Q^P \circ \theta^* \colon  \prescript{\phi}{} \K_Q^{*,\ft + \sigma }(\hat{\Sigma}) \to \prescript{\phi}{} \K_Q^{*,\ft + \sigma }(\hat{\Pi}) \to \prescript{\phi}{} \K_P^{*,\ft + \sigma}(\hat{\Pi}). \]
\end{prp}

\subsection{Push-forward and functoriality}
Now we demonstrate the main theorem of this section. 
The inclusion $W \to V$ extends to a $G$-equivariant continuous map
\[ \tilde{\iota} \colon G \times _H W \to V,\quad \ \ [g,w] \mapsto gw,\]
which induces the $P$-equivariant map 
\[\iota \colon P \times _Q W/\Sigma \to V/\Pi. \]
The push-forward by $\iota$ gives a group homomorphism
\begin{align}
\iota_! \colon \prescript{\phi}{} \K^{*,\ft}_Q(W/\Sigma) \cong \prescript{\phi}{} \K^{*,\ft}_P(P \times _Q W) \to \prescript{\phi}{} \K^{*,\ft}_P(V/\Pi).
\end{align}

\begin{thm}\label{thm:ind}
The diagram
\begin{align}\begin{split}
\xymatrix{
\prescript{\phi}{} \K^{*-m,\ft -\fw}_Q(W/\Sigma ) \ar[r]^{\prescript{\phi}{}{\mathrm{T}}_{H}^{\ft}} \ar[d]^{\iota_!} & \prescript{\phi}{} \K_Q^{*,\ft + \sigma  }(\hat{\Sigma }) \ar[d]^{\Ind _H^G } \\
\prescript{\phi}{} \K^{*-n,\ft -\fv}_P(V/\Pi) \ar[r]^{\prescript{\phi}{}{\mathrm{T}}_{G}^{\ft} } & \prescript{\phi}{} \K_P^{*,\ft +\sigma }(\hat{\Pi})
}
\end{split}\label{eq:ind}
\end{align}	
commutes. 
\end{thm}

\begin{lem}\label{lem:ind1}
The diagram (\ref{eq:ind}) commutes when $\Pi = \Sigma$.
\end{lem}
\begin{proof}
In this case we have $\iota_! = \Ind_Q^P$ by definition. Now the lemma follows from the commutativity of $\Ind _Q^P$ with the pull-back, the product and the push-forward with respect to $P$-equivariant maps.
\end{proof}

\begin{lem}\label{lem:ind2}
The diagram (\ref{eq:ind}) commutes when $P=Q$.
\end{lem}
\begin{proof}
In this case $\Ind _H^G = \theta^*$ holds. Let $\pi'$ and $\hat{\pi}'$ denote the projections from $W/\Sigma \times\hat{\Sigma}$ to $\hat{\Sigma}$ and $W/\Sigma$ respectively. The lemma follows from the commutativity of the diagram
\[ 
\xymatrix@C=1em{\scriptstyle
\prescript{\phi}{} \K^{*-m,\ft - \fw}_P(W/\Sigma ) \ar[dd]^{\scriptscriptstyle \iota_!} \ar[r]^{\scriptscriptstyle (\hat{\pi}')^*} \ar[rd]^{\scriptscriptstyle \hat{\pi}^*} 
&\scriptstyle \prescript{\phi}{} \K^{*-m,\ft - \fw }_P(W/\Sigma \times \hat{\Sigma}) \ar[d]^{\scriptscriptstyle  \theta^*} \ar[r]^{\scriptscriptstyle [\![ \cQ ]\!] }  
&\scriptstyle \prescript{\phi}{} \K^{*-m,\ft - \fw+\sigma }_P(W/\Sigma \times \hat{\Sigma}) \ar[r]^{\scriptscriptstyle \pi'_!} \ar[d]^{\theta^*} 
&\scriptstyle \prescript{\phi}{} \K^{*,\ft+\sigma}_P(\hat{\Sigma}) \ar[dd]^{\scriptscriptstyle  \theta^*} \\
&\scriptstyle \prescript{\phi}{} \K^{*-m,\ft - \fw}_P(W/\Sigma \times \hat{\Pi}) \ar[r]^{\scriptscriptstyle [\![ \cP ]\!]} \ar[d]^{\iota_!}  
&\scriptstyle \prescript{\phi}{} \K^{*-m,\ft - \fw + \sigma }_P(W/\Sigma \times \hat{\Pi}) \ar[d]^{\iota_!} \ar[rd]^{\pi'_!} 
 \ar@{}[ru]|{(2)} &\\
\scriptstyle \prescript{\phi}{} \K^{*-n,\ft - \fv}_P(V/\Pi ) \ar[r]^{\scriptscriptstyle \hat{\pi}^*} \ar@{}[ru]|{(1)}
&\scriptstyle \prescript{\phi}{} \K^{*-n,\ft - \fv}_P(V/\Pi \times \hat{\Pi}) \ar[r]^{ \scriptscriptstyle [\![ \cP]\!]}
&\scriptstyle \prescript{\phi}{} \K^{*-n,\ft - \fv + \sigma}_P(V/\Pi \times \hat{\Pi}) \ar[r]^{\scriptscriptstyle \pi_!}
&\scriptstyle \prescript{\phi}{} \K^{*-m,\ft - \fw}_P(\hat{\Pi}). 
}
\]
Among them, the square (1) commutes since $ \iota_! $ and $\hat{\pi}^*$ are given by the Kasparov product with the elements of the form $[\iota_!] \hotimes \id$ and $\id  \hotimes [\hat{\pi}^*]$ respectively. The square (2) also commutes by the same reason.  
\end{proof}

\begin{proof}[Proof of Theorem \ref{thm:ind}]
The theorem follows from Lemma \ref{lem:ind1} and Lemma \ref{lem:ind2} since the diagram
\[
\xymatrix{
\prescript{\phi}{} \K^{*,\ft -\fw }_Q(W/\Sigma) \ar[r]^{\iota_!} \ar[d]^{\prescript{\phi}{}{\mathrm{T}}_{H}^{\ft}} & \prescript{\phi}{} \K^{*-\fv,\ft}_Q(V/\Pi)  \ar[d]^{\prescript{\phi}{}{\mathrm{T}}_{H'}^{\ft}} \ar[r]^{\iota_!} & \prescript{\phi}{} \K_P^{*,\ft-\fv}(V/\Pi) \ar[d]^{\prescript{\phi}{}{\mathrm{T}}_{G}^{\ft}} \\
 \prescript{\phi}{} \K_Q^{-*,\ft + \sigma }(\hat{\Sigma }) \ar[r]^{\Ind _H^{H'}} & \prescript{\phi}{} \K_Q^{-*,\ft + \sigma}(\hat{\Pi}) \ar[r]^{\Ind _{H'}^{G}} & \prescript{\phi}{} \K_P^{-*,\ft + \sigma} (\hat{\Pi})
}
\]	
commutes. 
\end{proof}

Combining Theorem \ref{thm:ind} with Theorem \ref{thm:BC}, we obtain the following corollary.
\begin{cor}\label{cor:BCind}
The diagram 
\[
\xymatrix{
\prescript{\phi}{} \K_{*,\ft }^H(W) \ar[r]^{\prescript{\phi}{} \mu_H^\Sigma } \ar[d]^{\tilde{\iota}_*} & \prescript{\phi}{} \K_Q^{-*,\ft + \sigma }(\hat{\Sigma }) \ar[d]^{\Ind _H^G } \\
\prescript{\phi}{} \K_{*,\ft}^G(V) \ar[r]^{\prescript{\phi}{} \mu_G^\Pi}  & \prescript{\phi}{} \K_P^{-*,\ft + \sigma}(\hat{\Pi})
}
\]	
commutes. 
\end{cor}
This is a variation of the result of Valette \cite{valetteBaumConnesAssemblyMap2003} for twisted partial assembly maps.

\subsection{Atomic insulator and induction}\label{sec:atomic.insulator}
We finish the section with an explanation of the terminology `atomic insulators'. An \emph{atomic insulator} is a (possibly topological) insulator which has nearly flat energy bands with energies corresponding to the electric spectrum of an isolated atom. In \cite{shiozakiTopologicalCrystallineMaterials2017}, atomic insulators in $K$-theory are constructed from information on configurations of atoms, known as Wyckoff positions. This construction is mathematically formulated by the results in this section.

We consider Theorem \ref{thm:ind} in the case that $H$ is a finite group. In this case $W$ as in Lemma \ref{lem:affsub} is a fixed point of the $H$-action onto $V$ and $G \times _H W$ corresponds to a $G$-orbit in $V$.
\begin{defn}
A $G$-orbit in $V$ is called a \emph{Wyckoff position}. A topological insulator induced from the stabilizer subgroup $G_x$ is called an \emph{atomic insulator}.
\end{defn}
We remark that, for a space group $G$, the stabilizer group $G_x$ of any point $x \in V$ is a finite group. In \cite{hahnInternationalTablesCrystallography1987}, a Wyckoff position is defined to be a subset of $V$ consisting of all points $x \in V$ for which $G_x$ are conjugate subgroups of $G$. 
Our definition of Wyckoff positions is slightly modified from the standard definitions, which is suitable for the purpose of defining atomic insulators.

\begin{exmp}
To provide examples, let us take $G = \bZ^2 \rtimes \bZ_4$ to be the $2$-dimensional space group \textsf{p4}, where $\Pi = \bZ^2 \subset V = \bR^2$ is the standard square lattice and the point group $P = \bZ_4 \subset SO(2)$ acts on $V$ by rotation. To suppress notations, we make use of the identifications $V = \bC$, $\Pi = \bZ \oplus \bZ i$ and $P = \bZ_4 = \{ \pm 1, \pm i\}$. In \cite{shiozakiTopologicalCrystallineMaterials2017}, the $K$-theory $K^0_P(\hat{\Pi})$ is determined as the following $R(\bZ_4)$-module
\[
\mathrm{K}^0_P(\hat{\Pi}) 
\cong R(\bZ_4) \oplus R(\bZ_4) \oplus (1 - t + t ^2  - t^3),
\]
where the representation ring $R(\bZ_4) \cong \bZ[t]/(1 - t^4)$ of $\bZ_4$ is generated by the irreducible representation $\bZ_4 \subset U(1)$. The first direct summand $R(\bZ_4)$ is generated by an equivariant line bundle with non-trivial Chern class, whereas the second direct summand is $\mathrm{K}^0_P(\mathrm{pt}) \cong R(\bZ_4)$.

\begin{itemize}
\item[(a)]
For the Wyckoff position $\mathcal{W}_a = \Pi$, we have the stabilizer groups $H \cong \bZ_4$. For $\mathrm{pt} \in \mathcal{W}_a$, the image of $1 \in \mathrm{K}^0_H(\mathrm{pt}) \cong R(\bZ_4)$ under $\mathrm{Ind}^G_H \colon \mathrm{K}^0_{H}(\mathrm{pt}) \to \mathrm{K}^0_P(\hat{\Pi})$ is $1 \in R(\bZ_4) \cong K^0_P(\mathrm{pt})$.

\item[(b)]
For the Wyckoff position $\mathcal{W}_b = \frac{1 + i}{2} + \Pi$, we have the stabilizer groups $H \cong \bZ_4$. For $\mathrm{pt} \in \mathcal{W}_b$, the image of $1 \in R(\bZ_4) \cong \mathrm{K}^0_H(\mathrm{pt})$ under $\mathrm{Ind}^G_H \colon \mathrm{K}^0_{H}(\mathrm{pt}) \to \mathrm{K}^0_P(\hat{\Pi})$ is represented by the product line bundle $\hat{\Pi} \times \bC$ with the $\bZ_4$-action
\[
\Big(
\begin{pmatrix}
k_1 \\ k_2
\end{pmatrix}
_{\textstyle ,} 
z\Big)
\overset{i}{\mapsto}
\Big(
\begin{pmatrix}
-k_1 \\ k_2
\end{pmatrix}
_{\textstyle ,}
e^{-2\pi i k_1}z\Big)_{\textstyle .}
\]

\item[(c)]
For the Wyckoff position $\mathcal{W}_c = (\frac{1}{2} + \Pi) \sqcup (\frac{i}{2} + \Pi)$, we have the stabilizer groups $H \cong \bZ_2$. For $\mathrm{pt} \in \mathcal{W}_c$, the image of $1 \in R(\bZ_2) \cong \mathrm{K}^0_H(\mathrm{pt})$ under $\mathrm{Ind}^G_H \colon \mathrm{K}^0_{H}(\mathrm{pt}) \to \mathrm{K}^0_P(\hat{\Pi})$ is represented by the product vector bundle $\hat{\Pi} \times \bC^2$ with the $\bZ_4$-action
\[
\Big(
\begin{pmatrix}
k_1 \\ k_2
\end{pmatrix}
_{\textstyle ,} 
\begin{pmatrix} z \\ w \end{pmatrix} \Big)
\overset{i}{\to}
\Big(
\begin{pmatrix}
-k_2 \\ k_1
\end{pmatrix}
_{\textstyle ,}
\begin{pmatrix}
0 & 1 \\
e^{-2\pi i k_2} & 0
\end{pmatrix}
\begin{pmatrix} z \\ w \end{pmatrix} 
\Big).
\]

\item[(d)]
For any $x \not\in \mathcal{W}_a \cup \mathcal{W}_b \cup \mathcal{W}_c$, we have the Wyckoff position $\mathcal{W}_d = \bigsqcup_{g \in \bZ_4}(g x + \Pi)$, for which the stabilizer groups are $H \cong 1$. Note that the Wyckoff positions of this type are homotopic to each other. For $\mathrm{pt} = x \in \mathcal{W}_d$, the image of $1 \in \mathrm{K}^0_H(\mathrm{pt}) \cong \bZ$ under $\mathrm{Ind}^G_H \colon \mathrm{K}^0_{H}(\mathrm{pt}) \to \mathrm{K}^0_P(\hat{\Pi})$ is represented by the product bundle $\hat{\Pi} \times \ell^2(\bZ_4)$, where $\ell^2(\bZ_4)$ is the left regular representation of $\bZ_4$. 
\end{itemize}

In \cite{shiozakiAtiyahHirzebruchSpectral2018}, the atomic insulators associated to $\mathcal{W}_a$ and $\mathcal{W}_c$ are shown to generate the submodule $R(\bZ_4) \oplus (1 - t + t^2 - t^3) \subset \mathrm{K}^0_P(\hat{\Pi})$ with trivial Chern classes. 
\end{exmp}

We say that two Wyckoff positions (in our sense) $\cW$ and $\cW'$ are homotopic if there is a continuous $G$-equivariant map $f \colon [0,1] \times G/H \to V$ such that $f(t,\cdot)$ is injective, $ \mathop{\mathrm{Im}} (f(0, \cdot ))=\cW$ and $\mathop{\mathrm{Im}} (f(1, \cdot ))=\cW'$ holds. 

\begin{lem}
There is a one-to-one correspondence between homotopy classes of Wyckoff positions and finite subgroups of $G$.
\end{lem}
\begin{proof}
Let $H$ be a finite subgroup. Then there is a homeomorphism between $\mathop{\mathrm{Map}} (G/H, V)^G$ and $V^{H}$ mapping $f \colon G/H \to V$ to $f(eH) \in V$. Now $V^H$ is an affine subspace of $V$, and hence is contractible.  
\end{proof}

Now let us rephrase Corollary \ref{cor:BCind} in the case where $H$ is finite.
 
\begin{cor}
the diagram 
\[
\xymatrix{
\prescript{\phi}{} \K^{*,\ft }_H( \pt ) \ar@{=}[r] \ar[d]^{\iota_!} & \prescript{\phi}{} \K_H^{*,\ft}(\pt ) \ar[d]^{\Ind _H^G } \\
\prescript{\phi}{} \K^{*-n,\ft - \fv}_P(V/\Pi) \ar[r]^{\prescript{\phi}{}{\mathrm{T}}^{\ft}_G } & \prescript{\phi}{} \K_P^{*,\ft + \sigma}(\hat{\Pi}).
}
\]
commutes. 
\end{cor}
The right vertical map is a mathematical formulation of the construction of atomic insulators in \cite{shiozakiTopologicalCrystallineMaterials2017}.

\appendix
\section{Equivariant cohomology}
\label{sec:equivariant cohomology}

This appendix is a brief account of equivariant cohomology. We refer to \cites{alldayCohomologicalMethodsTransformation1993,hsiangCohomologyTheoryTopological1975,tuGroupoidCohomologyExtensions2006} for more details.

Let $P$ be a finite group acting on a topological space $X$ (which we assume to be `nice' enough, like a smooth manifold or a $P$-CW complex). Then, for any $n \in \bZ$, the $n$th $P$-equivariant cohomology of $X$, in the sense of Borel, is defined to be the following (singular) cohomology 
\[
H^n_P(X; \cA) = H^n(EP \times_P X; \cA),
\]
where $\cA$ is an abelian group, and $EP \times_P X$ is the so-called the Borel construction or the homotopy quotient of $X$. This is the quotient of the product of the universal $P$-bundle $B\pi \colon EP \to BP$ and $X$ under the diagonal action of $P$. 

Being a cohomology of $EP \times_P X$, the equivariant cohomology above can be twisted by an element of $H^1_P(X; \bZ_2) = H^1(EP \times_P X ; \bZ_2)$. In particular, a homomorphism $\phi \colon P \to \bZ_2$ defines an element of $H^1(BP; \bZ_2) \cong \mathrm{Hom}(\pi_1(BP), \bZ_2)$, in view of the fact that the fundamental group of $BP$ is $P$. Then the natural projection $EP \times_P X \ \to BP$ induces an element of $H^1_P(X; \bZ_2)$ by pull-back. We write $H^n_P(X; \prescript{\phi}{}  \cA) = H^n(EP \times_P X; \prescript{\phi}{} \cA)$ for the $P$-equivariant cohomology of $X$ twisted by $\phi$.

Suppose that $X = \mathrm{pt}$ comprises a single point. In this case, the Borel construction agrees with the classifying space $BP$ of $P$. A classical fact is that the (singular) cohomology of $BP$ is isomorphic to the cohomology of the group $P$, so that we have
\[
H^n_P(\mathrm{pt}; \cA) \cong H^n(BP; \cA) \cong H^n(P; \cA).
\]
The coefficient $\cA$ can also be twisted by a homomorphism $\phi \colon  P \to \bZ_2$: 
\[
H^n_P(\mathrm{pt}; \prescript{\phi}{}{\cA}) 
\cong H^n(BP; \prescript{\phi}{}{\cA}) \cong H^n(P; \prescript{\phi}{}{\cA}).
\]
The $\phi$-twisted group cohomology with coefficients in a (right) $P$-module $\cA$ is defined as follows: Let $C^n(P; \prescript{\phi}{} \cA)$ be the group of maps $c \colon \overbrace{P \times \cdots \times P}^n \to \cA$. We define a homomorphism $\delta \colon C^n(P; \prescript{\phi}{} \cA) \to C^{n+1}(P; \prescript{\phi}{} \cA)$ by
\begin{align*}
(\delta c)(p_0, \ldots, p_n)
&= 
\phi(p_0) c(p_1, \ldots, p_n) 
+ \sum_{i = 1}^{n-1} (-1)^i c(p_0, \ldots, p_{i-1}, p_{i+1}, \ldots, p_n) \\
&\quad
+ (-1)^n c(p_0, \ldots, p_{n-1}) p_n.
\end{align*}
It turns out that $(C^*(P; \prescript{\phi}{} \cA), \delta)$ is a cochain complex, and its $n$th cohomology is $H^n(P; \prescript{\phi}{} \cA)$. 

For instance, in the simple case that $\phi$ is trivial and $\cA$ is a trivial $P$-module, we immediately see $H^1(P;\cA) \cong \Hom (P, \cA)$. In the case that $P = \bZ_2$ and $\phi \colon P \to \bZ_2$ is the identity homomorphism, we have
\[
H^2(\bZ_2; \prescript{\phi}{} \bT) \cong \bZ_2,
\]
where $\bT$ is trivial as a right $\bZ_2$-module. In the case that $P = \bZ_2 \times \bZ_2$ and $\phi \colon P \to \bZ_2$ is the projection onto the first factor, we have
\[
H^2(\bZ_2 \times \bZ_2; \prescript{\phi}{}\bT) \cong \bZ_2 \oplus \bZ_2.
\]

\section{$\phi$-twisted Chabert-Echterhoff twisted equivariant KK-theory}\label{section:app}
Here we summarize definitions and standard facts on the $\phi$-twisted Chabert-Echterhoff KK-theory \cite{chabertTwistedEquivariantKK2001} defined for a pair of $\phi$-twisted $\bZ_2$-graded $(G,\Pi)$-C*-algebras. This is a twisted version (in the sense of Chabert--Echterhoff) of $\phi$-twisted equivariant KK-theory introduced in \cite{kubotaNotesTwistedEquivariant2016}*{Section 2} (which includes the Real equivariant KK-theory of groupoids with involution introduced by Moutuou \cite{moutuouEquivariantKKTheory2014}). 
\subsection{Definitions}
Let $G$ be a discrete group and let $\Pi$ be a normal subgroup of $G$. Set $P:=G/\Pi$ and let $\phi \colon P \to \bZ_2$ be a homomorphism. 
We say that a continuous action of $G$ or $P$ on a complex Banach space (a Hilbert space, C*-algebra, Hilbert C*-module) is $\phi$-linear if each $\alpha_g$ is linear if $\phi(g)=0$ and antilinear if $\phi(g)=1$.
\begin{defn}\label{def:GCst}
A \emph{$\phi$-twisted $(G,\Pi)$-C*-algebra} is a triple $(A,\alpha, \sigma)$, where $A$ is a C*-algebra, $\alpha \colon G \curvearrowright A$ is a $\phi$-linear $G$-action and $\sigma \colon \Pi \to U(\cM(A))$ is a group homomorphism satisfying $\alpha_g(\sigma_t)=\sigma_{gtg^{-1}}$ and $\alpha_t = \Ad \sigma_t$ for $g \in G$ and $t \in \Pi$. 
\end{defn}
Moreover, we say that a $\phi$-twisted $(G,\Pi)$-C*-algebra $(A,\alpha , \sigma)$ is $\bZ_2$-graded if $A$ is equipped with a linear involutive $\ast$-automorphism which commutes with the $G$-action and preserves each $\sigma_t$ for $t \in \Pi$.

\begin{rmk}\label{rmk:twisted}
We remark that a $\phi$-twisted $P$-C*-algebra $A$ is regarded as a $\phi$-twisted $(G,\Pi)$-C*-algebra by choosing $\sigma _t=1$ for any $t\in \Pi$. 
\end{rmk}

\begin{rmk}\label{rmk:real}
We say that a Real C*-algebra is a C*-algebra equipped with an involutive antilinear automorphism (a reference is e.g., \cite{goodearlNotesRealComplex1982}). Typically, if $A$ is a Banach $\ast$-algebra over $\bR$ satisfying the C*-norm condition, then it complexification $A \otimes _\bR \bC$ is a complex C*-algebra and the complex conjugation $a \otimes \lambda \mapsto a \otimes \bar{\lambda}$ is an involutive antilinear automorphism. 
The most typical example of Real C*-algebra is the complex field $\bC$ equipped with the complex conjugation $\bar{\cdot}$. In this paper we simply write as $\bC$ for the Real C*-algebra $(\bC, \bar{\cdot})$.

If $A$ is a Real $G$-C*-algebra (i.e., a Real C*-algebra with a $G$-action commuting with the complex conjugation), then it is regarded as a $\phi'$-twisted $G \times \bZ_2$-C*-algebra, where $\phi':=\pr_{\bZ_2} \colon G \times \bZ_2 \to \bZ_2 $. Passing through $(\id , \phi) \colon G \to G \times \bZ_2$, $A$ is also regarded as a $\phi$-twisted $G$-C*-algebra. More explicitly, $G$ acts on $A$ as
\[\prescript{\phi}{} \alpha_g(a) := \begin{cases} \alpha_g(a) & \phi (g)=0, \\ \overline{\alpha_g(a)} & \phi(g) =1 .\end{cases}\]
If moreover we have $\sigma \colon \Pi \to \cM(A)$ implementing $\alpha|_\Pi$ and $\overline{\sigma_t}=\sigma_t$ holds, then $(A,\alpha, \sigma)$ is viewed as a $\phi$-twisted $(G,\Pi)$-C*-algebra. 
\end{rmk}

\begin{rmk}\label{rmk:cross}
There is a general construction of a $\phi$-twisted $(G,\Pi)$-C*-algebra associated to a $\phi$-twisted $G$-C*-algebra, under the assumption that $\phi|_\Pi$ is trivial. For a $\phi$-twisted $G$-C*-algebra $A$, let $A \rtimes \Pi$ denote the (maximal or reduced) crossed product C*-algebra. The group $G$ acts on $A \rtimes \Pi$ as 
\[ \alpha_g(\sum a_t u_t) = \sum \alpha_g(a_t) \cdot u_{gtg^{-1}} \]
and let $\sigma_t :=u_t$. Then $(A \rtimes \Pi, \alpha, \sigma)$ is a $\phi$-twisted $(G,\Pi)$-C*-algebra. The correspondence extends to a functor between equivariant KK categories in Definition \ref{defn:pdesc}.
\end{rmk}

\begin{defn}
Let $A$ be a $\bZ_2$-graded $\phi$-twisted $(G,\Pi)$-C*-algebra. A \emph{$\phi$-twisted $G$-equivariant Hilbert $A$-module} is a $\bZ_2$-graded Hilbert $A$-module $E$ with an isometric $\phi$-linear action $\rho \colon G \curvearrowright E$ preserving the $\bZ_2$-grading such that 
\begin{itemize}
\item $\langle \rho_g (\xi) , \rho_g (\eta) \rangle =\alpha _{g} (\langle \xi,\eta \rangle )$ and
\item $\rho_g(\xi a) = \rho_g(\xi) \alpha_g(a)$.
\end{itemize}
\end{defn}
We write $\bB(E)$ and $\bK(E)$ for the C*-algebra of bounded adjointable and compact operators on $E$ respectively. We write $\alpha^E$ for the $\phi$-linear $G$-action on $\bB(E)$ given by $\alpha^E_g(T)\xi=(\rho _g\circ T\circ \rho_g^{-1})(\xi)$ for any $\xi \in E$.
We remark that if $A$ is a $(G,\Pi)$-C*-algebra then $\xi\mapsto \rho_t(\xi) \cdot \sigma_t$ determines a unitary operator on $E$ for $t \in \Pi$.

\begin{defn}
Let $(A, \alpha , \sigma)$ and $(B, \beta, \theta)$ be $\bZ_2$-graded $\phi$-twisted $(G,\Pi)$-C*-algebras. A \emph{$\phi$-twisted $(G,\Pi)$-equivariant Kasparov $A$-$B$ bimodule} is a triple $(E,\varphi,F)$ where
\begin{itemize}
\item a countably generated $\phi $-twisted $G$-equivariant Hilbert $B$-module $E$,
\item a $\bZ_2$-graded $G$-equivariant $\ast$-homomorphism $\varphi:A \to \bB (E)$ satisfying 
\begin{align*}
\rho_t(\xi) &=\varphi(\sigma_t) \xi \cdot \theta_t^* ,
\end{align*}
\item an odd self-adjoint operator $F \in \bB (E)$ such that $[\varphi(a),F], \varphi (a)(F^2-1)$ and $\varphi (a)(\alpha^E_g (F) -F)$ are in $\bK (E)$ for any $a \in A$.
\end{itemize}
\end{defn}

We say that two $\phi$-twisted $G$-equivariant Kasparov $A$-$B$ bimodules $(E_i,\varphi _i,F_i)$ are homotopic if there is a $(\phi , c,\tau)$-twisted $G$-equivariant Kasparov $A$-$B[0,1]$ bimodule $(\tilde{E},\tilde{\varphi},\tilde{F})$ such that each $\ev_i ^*(\tilde{E},\tilde{\varphi},\tilde{F})$ is unitarily equivalent to $(E_i,\varphi _i,F_i)$ for $i=0,1$.

\begin{defn}
We define $\prescript{\phi}{} \KK ^{G,\Pi}(A,B)$ to be the set of homotopy classes of $\phi$-twisted $(G,\Pi)$-equivariant Kasparov $A$-$B$ bimodules. This set forms an abelian group under the summation $[E_1,\varphi_1,F_1]+[E_2,\varphi_2,F_2]=[E_1 \oplus E_2, \varphi_1 \oplus \varphi_2, F_1 \oplus F_2]$. Here the zero element is represented by any degenerate bimodule, i.e., a Kasparov $A$-$B$ bimodule $(E,\varphi,F)$ with $[\varphi(a),F]=0$, $F^2-1=0$ and $\alpha^E_g (F) - F=0$.
\end{defn}

\begin{defn}\label{defn:Kasparov}
Let $(E_1, \varphi_1, F_1)$ be a $\phi$-twisted $(G,\Pi)$-equivariant Kasparov $A$-$B$ bimodule and let $(E_2,\varphi_2,F_2)$ be a $\phi$-twisted $(G,\Pi)$-equivariant Kasparov $B$-$D$ bimodule.
Let $(E_1,\varphi_1,F_1) \sharp (E_2, \varphi_2,F_2)$ be the set of operators $F$ on $E=E_1 \hotimes_B E_2$ such that 
\begin{itemize}
\item $(E,\varphi_1 \otimes 1,F)$ is a $\phi $-twisted Kasparov bimodule,
\item $F$ is an $F_2$-connection, i.e.\ $T_xF_2-FT_x$ and $F_2T_x^*-T_x^*F$ are compact (where $T_x \in \bB(E_2,E)$ is defined by $\xi \mapsto x \otimes_B \xi $),
\item $\varphi(a)[F,F_1 \hotimes 1]\varphi(a)$ is positive modulo compacts.
\end{itemize}
It is proved in the same way as \cite{skandalisRemarksKasparovTheory1984} that the set $(E_1,\varphi_1,F_1) \sharp (E_2, \varphi_2,F_2) $ is path-connected and hence
\[[E_1,\varphi_1,F_1]\otimes [E_2, \varphi_2,F_2] \to [E,\varphi_1 \otimes 1 , F] \]
gives a well-defined twisted Kasparov product
\[ \prescript{\phi}{}\KK^{G,\Pi}(A,B) \otimes \prescript{\phi}{} \KK^{G,\Pi}(B,D) \to \prescript{\phi}{}\KK^{G,\Pi}(A,D).\]
\end{defn}
In the same way as the usual Kasparov product, this twisted Kasparov product also satisfies the associativity, which enables us to regard it as the composition of morphisms.  
We write $\prescript{\phi }{}{\mathfrak{Kas}}^{G,\Pi}$ for the additive category whose objects are separable $\phi$-twisted $(G,\Pi)$-C*-algebras, morphisms are $\prescript{\phi}{}{\KK}^{G,\Pi}$ groups and the composition is given by this Kasparov product (we call it the $\phi$-twisted $(G,\Pi)$-equivariant Kasparov category). Here, the element $\id_A:=[A,\id_A,0] \in \prescript{}{}{\KK}^{G,\Pi}(A,A)$ plays the role of identity maps. 

\begin{defn}
We say that two separable $\phi$-twisted $(G,\Pi)$-C*-algebras $A$ and $B$ are \emph{$\prescript{\phi}{}\KK^{G,\Pi}$-equivalent} if they are isomorphic in the category $\prescript{\phi}{}{\mathfrak{KK}}^{G,\Pi}$, i.e., there are $\xi \in \prescript{\phi}{}{\KK}^{G,\Pi}(A,B)$ and $\eta \in \prescript{\phi}{}{\KK}^{G,\Pi}(B,A)$ such that $\xi \hotimes _B \eta = \id_A$ and $\eta \hotimes _A \xi = \id_B$. 
\end{defn}

\begin{rmk}[{cf.~\cite{chabertTwistedEquivariantKK2001}*{Example 3.11}}]\label{rmk:twistedKK}
Following Remark \ref{rmk:twisted}, let us regard $\phi$-twisted $P$-C*-algebras $A$, $B$ as $\phi$-twisted $(G,\Pi)$-C*-algebras. Then a triple $(E,\varphi,F)$ is a $\phi$-twisted $P$-equivariant Kasparov $A$-$B$ bimodule if and only if it is a $\phi$-twisted $(G,\Pi)$-equivariant Kasparov $A$-$B$ bimodule. This shows that there is an isomorphism 
\[\prescript{\phi}{}\KK^{P}(A,B) \cong \prescript{\phi}{}\KK^{G,\Pi}(A,B). \]
For example, $\prescript{\phi}{}{\KK}^{G,\Pi}(\bC,\bC) \cong \prescript{\phi}{}{\KK}^P(\bC,\bC) \cong \prescript{\phi}{}R(P)$. 
Moreover, the Kasparov product is also consistent with this isomorphism, which induces a faithful functor $\prescript{\phi}{}{\mathfrak{KK}}^P \to \prescript{\phi}{}{\mathfrak{KK}}^{G,\Pi}$.  
In particular, if $\xi \in \prescript{\phi}{}\KK^{P}(A,B)$ is a $P$-equivariant $\KK$-equivalence, then it also gives a $\prescript{\phi}{}{\KK}^{G,\Pi}$-equivalence. 
\end{rmk}

\begin{rmk}\label{rmk:RealKK}
Following Remark \ref{rmk:real}, let us regard Real $(G,\Pi)$-C*-algebras $A$, $B$ as $\phi$-twisted $(G,\Pi)$-C*-algebras. Then, in the same way,  a Real $(G,\Pi)$-equivariant Kasparov $A$-$B$ bimodule is also regarded as a $\phi$-twisted $(G,\Pi)$-equivariant Kasparov $A$-$B$ bimodule. 
This gives rise to a homomorphism $\KKR^{(G,\Pi)}(A,B) \to \prescript{\phi}{}{\KK}^{(G,\Pi)}(A,B)$, and moreover a functor
\[ \mathfrak{KKR}^{G,\Pi} \to \prescript{\phi}{}{\mathfrak{KK}}^{G,\Pi},\]
where $\mathfrak{KKR}^{G,\Pi}$ denotes the Real $(G,\Pi)$-equivariant Kasparov category. 
In particular, if $\xi \in \KKR^{G,\Pi}(A,B)$ is a Real $(G,\Pi)$-equivariant $\KK$-equivalence, then it also gives a $\prescript{\phi}{}{\KK}^{G,\Pi}$-equivalence. 
For example, if $P$ acts on the Euclidean space $V=\bR^d$ isometrically, then $C_0(V)$ and $\Cl(V)$ are both $\bZ_2$-graded $\phi$-twisted $P$-C*-algebras (and hence $\phi$-twisted $(G,\Pi)$-C*-algebras). Now the Dirac and dual Dirac elements give their $\prescript{\phi}{}{\KK}^{G,\Pi}$-equivalence.   
\end{rmk}

\begin{exmp}\label{exmp:Morita}
Let $A$, $B$ be $\phi$-twisted $(G,\Pi)$-C*-algebras. An \emph{equivariant $A$-$B$ imprimitivity bimodule} (cf.\ \cite{raeburnMoritaEquivalenceContinuoustrace1998}*{Definition 3.1}) is a $\phi$-twisted $(G,\Pi)$-equivariant Hilbert $B$-module $E$ which is full (i.e., $\overline{\mathrm{span}} \{ \langle e_1,e_2\rangle \mid e_1,e_2 \in E\}=B$) and equipped with a $\phi$-twisted $(G,\Pi)$-equivariant $\ast$-isomorphism $\varphi \colon A \to \bK(E)$. 
We say that $A$ and $B$ are \emph{equivariantly Morita equivalent} if there is a $(G,\Pi)$-equivariant $A$-$B$ imprimitivity bimodule. If $E$ is an equivariant $A$-$B$ imprimitivity bimodule, the Kasparov bimodule $[E,\varphi,0] \in \prescript{\phi}{}\KK^{G,\Pi}(A,B)$ has a multiplicative inverse. That is, $A$ and $B$ are $\prescript{\phi}{}\KK^{G,\Pi}$-equivalent. In particular, the equivariant KK-groups $\prescript{\phi}{}\KK^{G,\Pi}(D,A)$ and $\prescript{\phi}{}\KK^{G,\Pi}(D,B)$ are isomorphic for any $D$. \end{exmp}

\subsection{Partial descent homomorphism}
We also mention that the partial descent homomorphism \cite{chabertTwistedEquivariantKK2001} is also generalized to the $\phi$-twisted setting.  Let $(A,\alpha)$ and $(B,\beta)$ be $\phi$-twisted $G$-C*-algebras and let $A \rtimes \Pi$ and $B \rtimes \Pi$ be the associated $\phi$-twisted $(G,\Pi)$-C*-algebra as in Remark \ref{rmk:cross}.
Let $(E,\varphi,F)$ be a $\phi$-twisted $G$-equivariant Kasparov $A$-$B$ bimodule. We associate to it a $(G,\Pi)$-equivariant Kasparov $A \rtimes \Pi$-$B \rtimes \Pi$ bimodule $(E \rtimes \Pi , \varphi \rtimes \Pi ,F_\Pi)$ in the following way.
\begin{itemize}
\item $E \rtimes \Pi$ is the completion of $c_c(\Pi, E)$ equipped with
\begin{itemize}
\item  the $c_c(\Pi,B)$-valued inner product $\langle \xi , \eta \rangle (t) := \sum_{s \in \Pi} \langle \xi(s) , \rho_s(\eta(s^{-1}t) \rangle$, 
\item the $c_c(\Pi, B)$-action $(\xi b)(t):= \sum_{s \in \Pi} \xi(s) \cdot \beta_s(b(s^{-1}t))$, and
\item the $\phi$-twisted $G$-action $\tilde{\rho}_g(\xi)(t):=\rho_g(\xi (g^{-1}tg))$,
\end{itemize}
with respect to the norm $\| \xi \|:=\| \langle \xi, \xi \rangle\| ^{1/2}_{B \rtimes \Pi}$,
\item $\varphi  \rtimes \Pi \colon A \rtimes \Pi \to \bB(E \rtimes \Pi)$ is given by
\[ ((\varphi \rtimes \Pi)(a)\xi)(t) := \sum_{s \in \Pi} a(s) \rho_s(\xi(s^{-1}t)) ,\]
\item $(F_\Pi\xi) (t):=F (\xi(t))$, which is a well-defined bounded operator on $E \rtimes \Pi$.
\end{itemize}
\begin{defn}[{cf. {\cite{chabertTwistedEquivariantKK2001}*{Theorem 4.5}}}]\label{defn:pdesc}
The correspondence $[E,\varphi,F] \mapsto [E \rtimes \Pi , \varphi \rtimes \Pi ,F_\Pi]$ gives a group homomorphism
\[\prescript{\phi}{} j_{G,\Pi} \colon\prescript{\phi}{} \KK^G(A,B) \to \prescript{\phi}{} \KK^{G,\Pi}(A \rtimes \Pi , B \rtimes \Pi ), \]
which we call the \emph{$\phi$-twisted partial descent map}.  
\end{defn}
It is proved in the same way as \cite{kasparovEquivariantKKTheory1988}*{Theorem 3.11} and \cite{chabertTwistedEquivariantKK2001}*{Theorem 4.5} that $\prescript{\phi}{} j_{G,\Pi}$ is functorial, that is, compatible with the Kasparov product. In particular, if $x \in \KK^G(A,B)$ is a $\KK^G$-equivalence, then $\prescript{\phi}{} j_{G,\Pi}(x)$ is a $\KK^{G,\Pi}$-equivalence. 

\begin{exmp}\label{exmp:Moritadescent}
Let $A$, $B$ be $\phi$-twisted $G$-C*-algebras and let $E$ be a $\phi$-twisted $G$-equivariant imprimitivity $A$-$B$ bimodule. Then $E \rtimes \Pi$ is a $\phi$-twisted $(G,\Pi)$-equivariant $A \rtimes \Pi$-$B\rtimes \Pi$ bimodule, and hence $A \rtimes \Pi$ and $B \rtimes \Pi$ are equivariantly Morita equivalent in the sense of Example \ref{exmp:Morita}. The KK-element $\prescript{\phi}{}j_{G,\Pi}[E]=[E \rtimes \Pi]$ induces a $\prescript{\phi}{}\KK^{G,\Pi}$-equivalence of $A \rtimes \Pi$ and $B \rtimes \Pi$.
\end{exmp}

\subsection{Twisted equivariant K-theory and Fredholm operators}
Here we give a description of the group $\prescript{\phi}{} \KK^{G,\Pi} (\bC , A)$ for a $\bZ_2$-graded $\phi$-twisted C*-algebra $A$ as the set of homotopy classes of Fredholm operators. Here we assume that the quotient group $P:=G/\Pi$ is a compact group.

We say that a $\phi$-twisted $G$-equivariant Hilbert $A$-module $\sE$ is $(G,\Pi)$-equivariant if $\rho_t\xi = \xi \cdot \sigma_t^*$ for any $t \in \Pi$ and $\xi \in \sE$. The following lemma is easily verified.

\begin{lem}
Let $(\sE , \rho)$ be a $\phi$-twisted $\bZ_2$-graded $G$-equivariant Hilbert A-module and let $\{ \theta_t \}_{t\in \Pi}$ be the collection of unitaries on $\sE$ determined by $\rho_t\xi = \theta_t \xi \sigma_t^*$. Then the right $A$-module
\[ \cI_{G,\Pi}(\sE):= \{ \xi  \in C_b(G,\sE) \mid \xi(t^{-1}g) = \theta_t^* \xi(g) \text{ for any $t \in \Pi$} \} \]
with the $A$-valued inner product 
\[ \langle \xi , \eta \rangle :=\sum_{p \in P} \xi(s(p))^*\eta(s(p)), \]
where $s \colon P \to G$ is a section, and the $G$-action is given by $\tilde{\rho}_h (\xi)(g) = \rho_h(\xi(h^{-1}g)) $, is a $\phi$-twisted $(G,\Pi)$-equivariant Hilbert $A$-module.
\end{lem}

For a $\phi$-twisted $G$-C*-algebra $A$, let $\prescript{\phi}{} \sH_A$ denote the direct sum of infinite copies of $A$ regarded as a $\phi$-twisted $G$-equivariant Hilbert $A$-module and let $\prescript{\phi}{} \sH_A^{\mathrm{op}}:= \prescript{\phi}{} \sH_A$ with the opposite $\bZ_2$-grading. Set
\begin{align}
\prescript{\phi}{} \sH_A^{G,\Pi} :=\cI_{G,\Pi}(\prescript{\phi}{} \sH_A \oplus \prescript{\phi}{} \sH_A^{\mathrm{op}}). \label{eq:Hilbmod}
\end{align}

\begin{lem}[Kasparov stabilization theorem]\label{lem:stab}
Let $A$ be a $\phi$-twisted $(G,\Pi)$-C*-algebra and let $\sE$ be a countably generated $\bZ_2$-graded $\phi$-twisted $(G,\Pi)$-equivariant Hilbert $A$-module. Then there is a $(G,\Pi)$-equivariant unitary isomorphism
\[ \sE \oplus  \prescript{\phi}{} \sH_{A}^{G,\Pi} \cong \prescript{\phi}{} \sH_{A}^{G,\Pi}.  \]
\end{lem}
\begin{proof}
The following proof is based on \cite{mingoEquivariantTrivialityTheorems1984}. 
By the (non-equivariant) Kasparov stabilization theorem \cite{kasparovHilbertAstModules1980}, there is a possibly non-equivariant $\bZ_2$-graded unitary isomorphism
\[ U \colon \sE^{\oplus \infty} \oplus \prescript{\phi}{} \sH_{A} \oplus \prescript{\phi}{} \sH_A^{\mathrm{op}} \to \prescript{\phi}{} \sH_{A} \oplus \prescript{\phi}{} \sH_A^{\mathrm{op}}. \]
This is by definition $\Pi$-equivariant. 
The induced unitary
\[ \cI_{G,\Pi} U \colon \cI_{G,\Pi} (\sE^{\oplus \infty} \oplus \prescript{\phi}{} \sH_{A} \oplus \prescript{\phi}{} \sH_A^{\mathrm{op}}) \to \cI_{G,\Pi} (\prescript{\phi}{} \sH_{A} \oplus \prescript{\phi}{} \sH_A^{\mathrm{op}}), \]
i.e., $\cI_{G,\Pi} U \colon \cI_{G,\Pi} (\sE) \oplus \prescript{\phi}{} \sH_A^{G,\Pi} \to \prescript{\phi}{} \sH_A^{G,\Pi}$, determined by $(\cI_{G,\Pi} U(\xi))(g)= U \cdot \xi(g)$, is $G$-equivariant. 
Since $\sE$ is a $(G,\Pi)$-equivariant Hilbert $A$-module, $\cI_{G,\Pi} (\sE)$ includes the submodule consisting of constant functions, which is isomorphic to $\sE$. Hence 
\[ \cI_{G,\Pi}(\sE^{\oplus \infty})  \supset  \sE^{\oplus \infty}  \]
absorbs $\sE$, i.e., there is a $(G,\Pi)$-equivariant unitary $V \colon \sE \oplus \cI_{G,\Pi}(\sE^{\oplus \infty}) \to \cI_{G,\Pi}(\sE^{\oplus \infty})$. 
Finally we obtain the desired unitary
\[ \cI_{G,\Pi}(U) \circ (V \oplus \id_{\prescript{\phi}{} \sH_A^{G,\Pi}}) \circ (\id_\sE \oplus \cI_{G,\Pi}(U)^*). \qedhere   \]
\end{proof}

Here we recall the strict (strong$\ast$) topology on the space of bounded adjointable operators on Hilbert C*-modules. A sequence of operators $T_n \in \bB(\sE )$ converges to $0$ in the strict topology if and only if $\|T_n\xi \| \to 0$ and $\| T_n^*\xi\| \to 0$ holds for any $\xi \in \sE $. Note that, for a locally compact Hausdorff space $X$, an operator $T \in \bB(\prescript{\phi}{} \sH_{A \otimes C(X)}^{G,\Pi})$ corresponds to a uniformly bounded strictly continuous function $T \colon X \to \bB(\prescript{\phi}{} \sH_{A}^{G,\Pi})$.

\begin{lem}\label{lem:conn}
Let $\prescript{\phi}{} \sH_{A}^{G,\Pi}$ be as above. 
\begin{enumerate}
\item There is a strictly continuous family of $G$-invariant even unitaries $V_s \colon \prescript{\phi}{} \sH_A^{G,\Pi} \oplus \prescript{\phi}{} \sH_{A}^{G,\Pi} \to \prescript{\phi}{} \sH_{A}^{G,\Pi}$ for $s \in [0,1)$ such that the isometry $V_s|_{\prescript{\phi}{} \sH_{A}^{G,\Pi} \oplus 0}$ strictly converges to a $G$-invariant even unitary $V_1$.
\item The space $\cU^0(\prescript{\phi}{} \sH_{A}^{G,\Pi})^G$ of $G$-equivariant even unitaries is connected with respect to the strict topology.
\end{enumerate}
\end{lem}

\begin{proof}
We apply the Kasparov stabilization theorem \ref{lem:stab} for the Hilbert $A\otimes C[0,1]$-module $\prescript{\phi}{} \sH _{A \otimes C_0[0,1)}^{G,\Pi}$ to obtain a $G$-equivariant even unitary
\[V \colon \prescript{\phi}{} \sH _{A \otimes C[0,1]}^{G,\Pi} \oplus \prescript{\phi}{} \sH _{A \otimes C_0[0,1)}^{G,\Pi}  \to \prescript{\phi}{} \sH _{A \otimes C[0,1]}^{G,\Pi}. \]
This corresponds to a strictly continuous family of $G$-equivariant even unitaries
\[ V_s \colon \prescript{\phi}{} \sH _{A}^{G,\Pi} \oplus \prescript{\phi}{} \sH _{A}^{G,\Pi} \to \prescript{\phi}{} \sH _{A}^{G,\Pi}, \]
for $s \in [0,1) $. This gives a strictly continuous path $V_s|_{\prescript{\phi}{} \sH _{A,P} \oplus 0} \in \bB(\prescript{\phi}{} \sH_A^{G,\Pi})$ of even isometries converging strictly to $V_1 \in \cU^0(\prescript{\phi}{} \sH _{A}^{G,\Pi})$. This shows (1).

To see (2) we notice that, for any $T,S \in \bB(\prescript{\phi}{} \sH _{A}^{G,\Pi})$, the family $V_s \mathop{\mathrm{diag}} (T,S) V_s^*$ converges strictly to $V_1TV_1^*$. Moreover, Lemma \ref{lem:stab} also gives a unitary
\[W \colon (\prescript{\phi}{} \sH _{A}^{G,\Pi})^{\oplus \infty} = \prescript{\phi}{} \sH _{A}^{G,\Pi} \oplus (\prescript{\phi}{} \sH _{A}^{G,\Pi})^{\oplus \infty} \to  \prescript{\phi}{} \sH _{A}^{G,\Pi}.\]
Let $U$ be an arbitrary unitary on $\prescript{\phi}{} \sH_A^{G,\Pi}$. Then the above $V$ and $W$ gives a homotopy
\begin{align*}
U&=V_1^*V_1 U V_1^*V_1 \sim V_1^* V_0 \begin{pmatrix} U & 0 \\ 0 & 1 \end{pmatrix} V_0^*V_1\\
&= V_1^*V_0 \begin{pmatrix}1 & 0 \\ 0 & W \end{pmatrix} \begin{pmatrix} U & 0  \\ 0 & 1_{\infty}  \end{pmatrix} \begin{pmatrix}1 & 0 \\ 0 & W^* \end{pmatrix}V_0^*V_1. 
\end{align*}
Also, the standard homotopy $R_\theta \mathop{\mathrm{diag}}(U,1)R_\theta^* \mathop{\mathrm{diag}}(1,U^*)$, using the rotation matrix $R_\theta = \big( \begin{smallmatrix}\cos \theta & -\sin \theta \\ \sin \theta & \cos \theta \end{smallmatrix} \big)$, connects $1$ with $\mathop{\mathrm{diag}}(U,U^*)$. Hence we obtain a homotopy
\[ \mathop{\mathrm{diag}} (U, 1_{\infty})=\mathop{\mathrm{diag}} (U, 1_{\infty}, 1_\infty ) \sim \mathop{\mathrm{diag}} (U, U_\infty, U_\infty^*) = \mathop{\mathrm{diag}} (U_\infty, U_\infty^*) \sim 1_\infty\]
of unitaries in $(\prescript{\phi}{} \sH_A^{G,\Pi})^{\oplus \infty}$. This gives a homotopy connecting $U$ with $1$. 
\end{proof}

We write $\Fred (\prescript{\phi}{} \sH_A^{G,\Pi})^G$ for the space of $G$-invariant odd self-adjoint operators $F \in \bB(\prescript{\phi}{} \sH_A^{G,\Pi}) $ such that $F^2-1 \in \bK(\prescript{\phi}{} \sH_A^{G,\Pi})$. We impose the topology on this set generated by the strict topology of $\bB(\prescript{\phi}{} \sH_A^{G,\Pi})$ and the pull-back of the norm topology of $\bK(\prescript{\phi}{} \sH_A^{G,\Pi})$ by the map $F \mapsto F^2-1$. 
The above lemmas enable us to put the abelian semigroup structure on $\pi_0(\Fred (\prescript{\phi}{} \sH_A^{G,\Pi})^G)$ as
\[[F] + [F'] := [V(\mathrm{diag}(F,F'))V^*], \]
where $V$ is an arbitrary choice of $(G,\Pi)$-equivariant unitary $\prescript{\phi}{} \sH_A^{G,\Pi} \to \prescript{\phi}{} \sH_A^{G,\Pi}  \oplus \prescript{\phi}{} \sH_A^{G,\Pi} $.

\begin{thm}\label{thm:twKK}
The $\phi$-twisted equivariant KK-group $\prescript{\phi}{} \KK^{G,\Pi} (\bC, A)$ is isomorphic to $\pi_0(\Fred (\prescript{\phi}{} \sH_A^{G,\Pi})^G)$. 
\end{thm}
\begin{proof}
The map $\pi_0(\Fred (\prescript{\phi}{} \sH_A^{G,\Pi})^G) \to \prescript{\phi}{} \KK^{G,\Pi}(\bC, A)$ is given by
\[ [F] \mapsto [\prescript{\phi}{} \sH_A^{G,\Pi}, 1 ,F ]. \]
By definition of the equivalence relation on $\prescript{\phi}{} \KK^{G,\Pi}(\bC,A)$, this is a well-defined group homomorphism. Its surjectivity follows from Lemma \ref{lem:stab}. Hence we show its injectivity.  

Assume that $F_0, F_1 \in \Fred (\prescript{\phi}{} \sH_A^{G,\Pi})^G$ is mapped to the same element in the KK-group. Then there are degenerate triples $(\sE_i, \pi_i , F'_i)$ for $i=0,1$ and a homotopy $(\tilde{\sE}, \tilde{\pi}, \tilde{F})$ such that $(\ev_i)_*(\tilde{\sE}, \tilde{\pi}, \tilde{F}) =(\prescript{\phi}{} \sH_A^{G,\Pi} \oplus \sE_i, 1 \oplus \pi_i, F_i \oplus F'_i)$. 
By replacing $\sE_i$ with $\pi_i(1)\sE_i$ and $\tilde{\sE}_i$ with $\tilde{\pi}(1)\tilde{\sE}$, we may assume that $\pi_i=1$ and $\tilde{\pi}=1$ (and hence $F_0', F_1'$ are self-adjoint unitaries).
Moreover, by Lemma \ref{lem:stab}, we may assume that $\sE_i \cong \prescript{\phi}{} \sH_A^{G,\Pi}$ and $\widetilde{\sE}= (\prescript{\phi}{} \sH_{A[0,1]}^{G,\Pi})^{\oplus 2}$ by adding an extra degenerate triple $(\prescript{\phi}{} \sH_{A[0,1]}^{G,\Pi}, 1 , F')$ if necessary. Therefore, we get a homotopy $\mathop{\mathrm{diag}}(F_0,F_0') \sim \mathop{\mathrm{diag}} (F_1,F_1')$ in $\Fred ((\prescript{\phi}{} \sH_A^{G,\Pi})^{\oplus 2})$.  

Moreover, in the same way as the proof of Lemma \ref{lem:conn} (2), we have a homotopy 
\[F_0 = V_1^*V_1 F_0 V_1^*V_1 \sim V_1^*V_0 \begin{pmatrix}F_0 & 0 \\ 0 & F_0' \end{pmatrix} V_0^*V_1 . \]
Finally we get 
\[F_0 \sim  V_1^*V_0 \begin{pmatrix}F_0 & 0 \\ 0 & F_0' \end{pmatrix} V_0^*V_1 \sim  V_1^*V_0 \begin{pmatrix}F_1 & 0 \\ 0 & F_1' \end{pmatrix} V_0^*V_1 \sim F_1. \]
This shows the desired injectivity. 
\end{proof}

Let $\Fred_{p,q}(\prescript{\phi}{} \sH_A^{G,\Pi})$ denote the subspace of $\Fred  (\prescript{\phi}{} \sH_A^{G,\Pi} \hotimes \Delta_{p,q})$ consisting of Fredholm operators $F$ with $[F,\fc(v)]=0$ for any $v \in \Cl_{p,q}$. Then 
\[[F ] \mapsto  [\prescript{\phi}{} \sH_A^{G,\Pi} \hotimes \Delta_{p,q} , 1 , F]\]
gives a homomorphism $\pi_0 \Fred_{p,q}(\prescript{\phi}{} \sH_A^{G,\Pi}) \to \prescript{\phi}{} \KK^{G,\Pi}(\Cl_{p,q} , A)$.
\begin{cor}\label{cor:twKK}
The $\phi$-twisted equivariant KK-group $\prescript{\phi}{} \KK^{G,\Pi} (\Cl_{p,q}, A)$ is isomorphic to $\pi_0(\Fred_{p,q}(\prescript{\phi}{} \sH_A^{G,\Pi})^G)$.
\end{cor}
\begin{proof}
Note that $\prescript{\phi}{} \sH_{A \hotimes \Cl_{q,p}}^{G,\Pi}$ is isomorphic to $P(\prescript{\phi}{} \sH_A^{G,\Pi } \hotimes \Delta_{p,q} \hotimes \Cl_{q,p})$, where $P$ is a rank $1$ even projection of $\Cl_{p,q} \hotimes \Cl_{q,p}$.  
Now the map
\[ \vartheta \colon \Fred_{p,q}(\prescript{\phi}{} \sH_A^{G,\Pi})^G \to \Fred(\prescript{\phi}{} \sH_{A \hotimes \Cl_{q,p}}^{G,\Pi})^G\]
given by $F \mapsto P(F \hotimes \id_{\Cl_{q,p}})P$ is a homeomorphism. This is compatible with the isomorphism of equivariant KK-theory
\begin{align*} 
\theta \colon \prescript{\phi}{} \KK^{G,\Pi} (\Cl_{p,q},A) \xrightarrow{\blank \hotimes_{\bC} \id_{\Cl_{q,p}}} & \prescript{\phi}{} \KK^{G,\Pi} (\Cl_{p,q} \hotimes \Cl_{q,p},A \hotimes \Cl_{q,p})\\\xrightarrow{[P] \hotimes_{\Cl_{p,q} \hotimes \Cl_{q,p}} \blank } & \prescript{\phi}{} \KK^{G,\Pi} (\bC ,A \hotimes\Cl_{q,p} ). 
\end{align*}
This finishes the proof since the diagram
\[
\xymatrix{
\pi_0\Fred_{p,q}(\prescript{\phi}{} \sH_A^{G,\Pi})^G \ar[r] \ar[d]^{\vartheta}  & \prescript{\phi}{} \KK^{G,\Pi}(\Cl_{p,q},A) \ar[d]^\theta \\ 
\pi_0\Fred(\prescript{\phi}{} \sH_{A \hotimes \Cl_{q,p}}^{G,\Pi})^G \ar[r]^\cong  & \prescript{\phi}{} \KK^{G,\Pi}(\bC, A \hotimes \Cl_{q,p})
}
\]
commutes and the second horizontal map is an isomorphism by Theorem \ref{thm:twKK}. 
\end{proof}

\bibliographystyle{abbrv}
\bibliography{ref.bib}
\end{document}